\documentclass{article}

\usepackage{amsmath, amssymb, amsthm}
\usepackage{cleveref}
\usepackage[margin=1in]{geometry}
\usepackage{url}
\usepackage{bm}
\usepackage{dsfont}
\usepackage{relsize}

\theoremstyle{plain}
\newtheorem{thm}{Theorem}
\numberwithin{thm}{subsection}

\newtheorem{lem}[thm]{Lemma}
\newtheorem{prop}[thm]{Proposition}
\newtheorem{cor}[thm]{Corollary}
\newtheorem{rmk}[thm]{Remark}

\theoremstyle{definition}
\newtheorem{defn}[thm]{Definition}
\newtheorem{conj}[thm]{Conjecture}
\newtheorem{ex}[thm]{Example}

\usepackage{tikz}
\usetikzlibrary{hobby, arrows.meta,bending, decorations.markings}
\tikzstyle{V}=[draw, fill =black, circle, inner sep=0pt, minimum size=3pt]

\begin{document}

\title{Defining an Affine Partition Algebra}

\author{Samuel Creedon and Maud De Visscher}

\date{}

\maketitle

\begin{abstract}
We define an affine partition algebra by generators and relations and prove a variety of basic results regarding this new algebra analogous to those of other affine diagram algebras. In particular we show that it extends the Schur-Weyl duality between the symmetric group and the partition algebra. We also relate it to the affine partition category recently defined by J. Brundan and M. Vargas. Moreover, we show that this affine partition category is a \emph{full} monoidal subcategory of the Heisenberg category.
\end{abstract}


\section{Introduction}

Classical Schur-Weyl duality relates the representations of the symmetric group and the general linear group via their commuting actions on tensor space. The Brauer algebra was introduced in \cite{B37} to play the role of the symmetric group in a corresponding duality for the symplectic and orthogonal groups. The partition algebra was originally defined by P. Martin in \cite{M91} in the context of Statistical Mechanics. V. Jones showed in \cite{J94} that it appears in another version of Schur-Weyl duality. More precisely, if one replaces the general linear group by the finite subgroup of all permutation matrices then the centraliser algebra of its action on tensor space is precisely the partition algebra. The aim of this paper is to define an affine version of the partition algebra. 

There are different \lq affinization' processes for such algebras. One such process amounts to making the Jucys-Murphy elements of the ordinary algebra into variables, retaining some of the relations between these variables and the standard generators of the ordinary algebra. Starting with the symmetric group algebra, this \lq affinization' process gives rise to the much-studied degenerate affine Hecke algebra (see for example \cite{K05}). In the case of the Brauer algebra, M. Nazarov used this process in \cite{Naz96} to define the affine Wenzl algebra (also called the Nazarov-Wenzl algebra or the degenerate affine BMW algebra in the literature). This process was also employed independently in both \cite{RS15} and \cite{Sar13} to  define a  degenerate affine walled Brauer algebra. 

R. Orellana and A. Ram introduced a different \lq affinization' process in \cite{OR07} focussed on extending the Schur-Weyl dualities, via the affine braid group. They have applied it to the symmetric group, Brauer algebra and their quantum analogues. This process naturally leads to cyclotomic quotients of the affine algebras.

In this paper, we follow the first approach to define an affine partition algebra by turning the Jucys-Murphy elements for the partition algebra into variables and ask them to retain certain relations with the generators. But we also show that this affine partition algebra naturally extends the commuting action on tensor space with the symmetric group. We started this work by trying to use the presentation for the partition algebra given  by T. Halverson and A. Ram in \cite{HR05} but were unable to define an algebra with the expected properties in this way. So we instead used the more recent presentation given by J. Enyang in \cite{Eny12} (which uses a new set of generators)  to define the affine partition algebra $\mathcal{A}_{2k}^{\text{aff}}$. We prove that it satisfies many properties analogous to those for other affine diagram algebras.

While writing this paper, J. Brundan and M. Vargas produced a preprint \cite{BV21} defining an affine partition category $\mathsf{APar}$ as a monoidal subcategory of  the Heisenberg category generated by some objects and morphisms.  Taking an endomorphism algebra in their category gives an alternative definition of an affine partition algebra, which they denote by $AP_k$. They prove many properties for this category and use it to give a new approach to the representation theory of the partition category. However, as they note in \cite[Remark 4.12]{BV21} they have not attempted to give a basis for the morphism spaces in their category, or to give a presentation for it.
Inspired by their work, we have explored the connection between our affine partition algebra and the Heisenberg category.
We have added a section at the end of our paper where we construct a surjective homomorphism from $\mathcal{A}_{2k}^{\text{aff}}$ to an endomorphism algebra in the Heisenberg category.  Our argument generalises to show that the affine partition category $\mathsf{APar}$ of  Brundan and Vargas is in fact the full monoidal subcategory of the Heisenberg category generated by one object. Using work of Khovanov \cite{Kho14}, this gives a basis for all morphism spaces in $\mathsf{APar}$ and hence also for $AP_k$. 
We also obtain as a corollary that $AP_k$ is a quotient of $\mathcal{A}_{2k}^{\text{aff}}$. We do not know whether these two algebras are in fact isomorphic. If they were, then our definition of $\mathcal{A}_{2k}^{\text{aff}}$ would also give a presentation for $AP_k$.
 
\medskip

The paper is structured as follows. Section 2 deals with the ordinary partition algebra. In Section 2.1 we recall the diagram basis and the original presentation of the partition algebra given by T. Halverson and A. Ram. In Section 2.2, we recall the definition of the Jucys-Murphy elements and the more recent presentation of the algebra given by J. Enyang. We also state some relations which will be needed in defining an affine partition algebra. Section 2.3 introduces a new normalisation of the Jucys-Murphy elements and of Enyang's generators which has the advantage of simplifying many of the relations in our definition. Finally, in Section 2.4, we recall explicitly the Schur-Weyl duality between the partition algebra and the symmetric group. 

Section 3 gives the definition of the affine partition algebra $\mathcal{A}_{2k}^{\text{aff}}$ in terms of generators and relations and proves some properties. In particular, we show in Section 3.1 that the ordinary partition algebra appears both as a subalgebra and as a quotient of $\mathcal{A}_{2k}^{\text{aff}}$. In Section 3.2, we describe a family of central elements in $\mathcal{A}_{2k}^{\text{aff}}$ and formulate a conjecture about its centre. Finally, in Section 3.3 we show that $\mathcal{A}_{2k}^{\text{aff}}$ extends the action of the partition algebra on tensor space as desired. 

Section 4 deals with the connections with the Heisenberg category and the work of J. Brundan and M. Vargas on their affine partition category. 
In Section 4.1, we recall the definition of the Heisenberg category including the basis of the morphism spaces given by M. Khovanov in \cite{Kho14}. In Section 4.2 we define a  homomorphism from $\mathcal{A}_{2k}^{\text{aff}}$ to the endomorphism space of a particular object in the Heisenberg category and prove that it is surjective. In Section 4.3, we generalise the arguments from Section 4.2 to show that $\mathsf{APar}$ is the full monoidal subcategory of the Heisenberg category generated by one object and deduce that $AP_k$ is a quotient of $\mathcal{A}_{2k}^{\text{aff}}$.


\section{Partition algebra}


\subsection{Diagrammatics and presentation}


For this section we give the definition of the partition algebra $\mathcal{A}_{2k}(z)$ and its presentation established in \cite{HR05} (and independently in \cite{East11}). For $k \in \mathbb{N}$, we let $[k] := \{1,2,\dots,k\}$, and $[k'] := \{1',2',\dots,k'\}$. We view $[k]\cup [k']$ as a formal set on $2k$ elements, and let $\Pi_{2k}$ denote the set of all set partitions of $[k]\cup[k']$. Given any $\alpha \in \Pi_{2k}$, we say a partition diagram of $\alpha$ is  any graph with vertex set $[k]\cup [k']$ whose connected components partition the vertices according to the blocks of $\alpha$. We do not distinguish between $\alpha$ and any partition diagram of $\alpha$, in particular we will only care about the connected components of such graphs, not the edges forming the components. When drawing such a diagram, we will arrange the vertices in two rows with the top row going from 1 to $k$, and the bottom row from $1'$ to $k'$. For example, in $\Pi_{10}$ we have the identification
\[ \{\{1,2,2',3\},\{3'\},\{1',4,4'\},\{5,5'\}\} = \begin{matrix}\begin{tikzpicture}[scale=0.7]
	\node[V, label=above:{$1$}] (1) at (0,1){};
	\node[V, label=above:{$2$}] (2) at (1,1){};
	\node[V, label=above:{$3$}] (3) at (2,1){};
	\node[V, label=above:{$4$}] (4) at (3,1){};
	\node[V, label=above:{$5$}] (5) at (4,1){};
	\node[V, label=below:{$1'$}] (1') at (0,0){};
	\node[V, label=below:{$2'$}] (2') at (1,0){};
	\node[V, label=below:{$3'$}] (3') at (2,0){};
	\node[V, label=below:{$4'$}] (4') at (3,0){};
	\node[V, label=below:{$5'$}] (5') at (4,0){};
	\draw (1) to [bend right] (2) to [bend right] (3) to (2') (4) to (4') to [bend right] (1') (5) to (5');
	\end{tikzpicture}\end{matrix}. \]
We define a product $\circ$ on $\Pi_{2k}$ as follows: Given $\alpha, \beta \in \Pi_{2k}$, we let $\alpha \circ \beta \in \Pi_{2k}$ be the set partition obtained by first stacking the diagram of $\alpha$ on top of that of $\beta$, identifying the bottom row of $\alpha$ with the top row of $\beta$, removing any connected components lying entirely within the middle row, and then reading off the connected components formed between the top row of $\alpha$ and the bottom row of $\beta$. For example consider
\[ \alpha = \begin{matrix}\begin{tikzpicture}[scale=0.7]
	\node[V, label=above:{$1$}] (1) at (0,1){};
	\node[V, label=above:{$2$}] (2) at (1,1){};
	\node[V, label=above:{$3$}] (3) at (2,1){};
	\node[V, label=above:{$4$}] (4) at (3,1){};
	\node[V, label=above:{$5$}] (5) at (4,1){};
	\node[V, label=below:{$1'$}] (1') at (0,0){};
	\node[V, label=below:{$2'$}] (2') at (1,0){};
	\node[V, label=below:{$3'$}] (3') at (2,0){};
	\node[V, label=below:{$4'$}] (4') at (3,0){};
	\node[V, label=below:{$5'$}] (5') at (4,0){};
	\draw (1) to [bend right] (2) to [bend right] (3) to (2') (4) to (4') to [bend right] (1') (5) to (5');
	\end{tikzpicture}\end{matrix} \hspace{2mm} \text{ and } \hspace{2mm}  
	\beta = \begin{matrix}\begin{tikzpicture}[scale=0.7]
	\node[V, label=above:{$1$}] (1) at (0,1){};
	\node[V, label=above:{$2$}] (2) at (1,1){};
	\node[V, label=above:{$3$}] (3) at (2,1){};
	\node[V, label=above:{$4$}] (4) at (3,1){};
	\node[V, label=above:{$5$}] (5) at (4,1){};
	\node[V, label=below:{$1'$}] (1') at (0,0){};
	\node[V, label=below:{$2'$}] (2') at (1,0){};
	\node[V, label=below:{$3'$}] (3') at (2,0){};
	\node[V, label=below:{$4'$}] (4') at (3,0){};
	\node[V, label=below:{$5'$}] (5') at (4,0){};
	\draw (1) to (1') to [bend left] (2') (2) to (4') (5) to (5') to [bend right] (3');
	\end{tikzpicture}\end{matrix} \]
in $\Pi_{10}$. Then we have 
\[ \alpha \circ \beta = \begin{matrix}\begin{tikzpicture}[scale=0.7]
	\node[V, label=above:{$1$}] (1) at (0,1){};
	\node[V, label=above:{$2$}] (2) at (1,1){};
	\node[V, label=above:{$3$}] (3) at (2,1){};
	\node[V, label=above:{$4$}] (4) at (3,1){};
	\node[V, label=above:{$5$}] (5) at (4,1){};
	\node[V] (1') at (0,0){};
	\node[V] (2') at (1,0){};
	\node[V] (3') at (2,0){};
	\node[V] (4') at (3,0){};
	\node[V] (5') at (4,0){};
	\node[V, label=below:{$1'$}] (1'') at (0,-1){};
	\node[V, label=below:{$2'$}] (2'') at (1,-1){};
	\node[V, label=below:{$3'$}] (3'') at (2,-1){};
	\node[V, label=below:{$4'$}] (4'') at (3,-1){};
	\node[V, label=below:{$5'$}] (5'') at (4,-1){};
	\draw (1) to [bend right] (2) to [bend right] (3) to (2') (4) to (4') to [bend right] (1') (5) to (5') (1') to (1'') to [bend left] (2'') (2') to (4'') (5') to (5'') to [bend right] (3''); \end{tikzpicture}\end{matrix} = \begin{matrix}\begin{tikzpicture}[scale=0.7]
	\node[V, label=above:{$1$}] (1) at (0,1){};
	\node[V, label=above:{$2$}] (2) at (1,1){};
	\node[V, label=above:{$3$}] (3) at (2,1){};
	\node[V, label=above:{$4$}] (4) at (3,1){};
	\node[V, label=above:{$5$}] (5) at (4,1){};
	\node[V, label=below:{$1'$}] (1') at (0,0){};
	\node[V, label=below:{$2'$}] (2') at (1,0){};
	\node[V, label=below:{$3'$}] (3') at (2,0){};
	\node[V, label=below:{$4'$}] (4') at (3,0){};
	\node[V, label=below:{$5'$}] (5') at (4,0){};
	\draw (1) to [bend right] (2) to [bend right] (3) to (4') (1') to [bend left] (2') to (4) (5) to (5') to [bend right] (3');
	\end{tikzpicture}\end{matrix}. \]
Clearly this product is associative and independent of the choice of graphs used to represent the set partitions. The element $1=\{\{i,i'\} \ | \ i \in [k]\} \in \Pi_{2k}$ is an identity element, and thus $(\Pi_{2k},\circ)$ is in fact a monoid. Given $\alpha,\beta \in \Pi_{2k}$, we let $m(\alpha,\beta)$ denote the number of middle components removed in evaluating $\alpha \circ \beta$. In the above example, we have $m(\alpha,\beta)=1$. Now let $z$ be a formal variable and $\mathbb{C}[z]$ the polynomial ring. The partition algebra $\mathcal{A}_{2k}(z)$ is the $\mathbb{C}[z]$-algebra whose basis as a free $\mathbb{C}[z]$-module is given by the set $\Pi_{2k}$, and whose product is given by
\[ \alpha\beta := z^{m(\alpha,\beta)}\alpha \circ \beta \]
for all $\alpha, \beta \in \Pi_{2k}$, extended linearly over $\mathbb{C}[z]$.

For $0\leq l \leq k$, we identify $\Pi_{2l}$ as a submonoid of $\Pi_{2k}$ given diagrammatically by
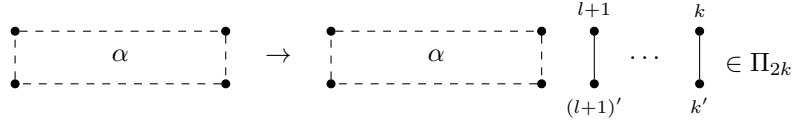
\[ \begin{matrix}\begin{tikzpicture}[scale=0.7]
       \node[V] (1) at (0,1){};
       \node[label=below:{$\alpha$}] (a) at (2,1){};
       \node[V] (2) at (4,1){};
       \node[V] (1') at (0,0){};
       \node[V] (2') at (4,0){};

       \draw (1) to (2) to (2') to (1') to (1) [dashed];

       \node (arrow) at (5,0.5){$\rightarrow$};

       \node[V] (3) at (6,1){};
       \node[label=below:{$\alpha$}] (a) at (8,1){};
       \node[V] (4) at (10,1){};
       \node[V] (3') at (6,0){};
       \node[V] (4') at (10,0){};

       \node[V, label=above:{$\scriptstyle l+1$}] (l+1) at (11,1){};
       \node (dots1) at (12,0.5){$\dots$};
       \node[V, label=below:{$\scriptstyle (l+1)'$}] (l+1') at (11,0){};

       \node[V, label=above:{$\scriptstyle k$}] (k) at (13,1){};
       \node[V, label=below:{$\scriptstyle k'$}] (k') at (13,0){};

       \draw (3) to (4) to (4') to (3') to (3) [dashed];
       \draw (l+1) to (l+1');
       \draw (k) to (k');
    \end{tikzpicture}\end{matrix} \in \Pi_{2k} \] 
for any $\alpha \in \Pi_{2l}$. Define $\Pi_{2k-1}$ to be the submonoid of $\Pi_{2k}$ consisting of all set partitions of $[k]\cup [k']$ where $k$ and $k'$ belong to the same block. We have a chain of monoids $\emptyset = \Pi_{0} \subset \Pi_{1} \subset \Pi_{2} \subset \dots$. For any $0\leq r \leq 2k$, we let $\mathcal{A}_{r}(z)$ denote the subalgebra of $\mathcal{A}_{2k}(z)$ generated by $\Pi_{r}$. We obtain an analogous chain
\[ \mathbb{C}[z] = \mathcal{A}_{0}(z) \subseteq \mathcal{A}_{1}(z) \subset \mathcal{A}_{2}(z) \subset \dots. \]
of $\mathbb{C}[z]$-algebras. The rank of $\mathcal{A}_{r}(z)$ over $\mathbb{C}[z]$ is $|\Pi_{r}| = B_{r}$, where $B_{r}$ is the $r^{\text{th}}$ Bell number. We can view $\mathcal{A}_{r}(z)$ as an infinte dimensional algebra over $\mathbb{C}$ with basis $\{z^{n}\alpha \ | \ n\in\mathbb{Z}_{\geq 0}, \ \alpha\in\Pi_{r}\}$. When we do so, we use the notation $\mathcal{A}_{r}$ instead. For any $\delta \in \mathbb{C}$, let $(z-\delta)$ denote the ideal of $\mathcal{A}_{r}$ generated by $z-\delta$. Then we let $\mathcal{A}_{r}(\delta):=\mathcal{A}_{r}/(z-\delta)$, which is a finite dimensional $\mathbb{C}$-algebra with dim$(\mathcal{A}_{r}(\delta))=B_{r}$.

For $i \in [k-1]$ and $j \in [k]$, we define the following elements of $\Pi_{2k}$:
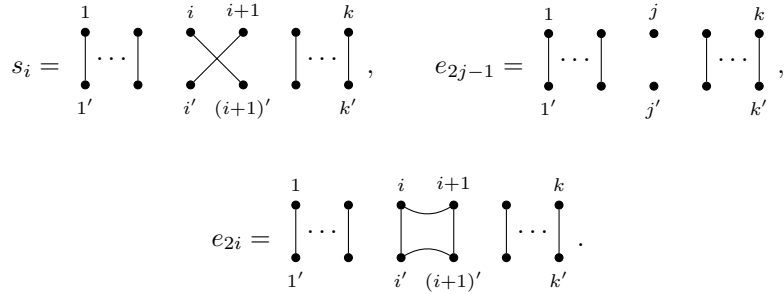
\[ s_{i} = \begin{matrix}\begin{tikzpicture}[scale=0.7]
	
	\node[V, label=above:{$\scriptstyle 1$}] (1) at (0,1){};
	
	\node[draw=none] (0.5) at (0.5,0.5){$\ldots$};
	
	\node[V, label=below:{$\scriptstyle 1'$}] (1') at (0,0){};
	\node[V] (2) at (1,1){};
	\node[V, label=above:{$\scriptstyle i$}] (i) at (2,1){};
	\node[V, label=above:{$\scriptstyle i+1$}] (i+1) at (3,1){};
	\node[V] (3) at (4,1){};
	\node[V] (2') at (1,0){};
	\node[V, label=below:{$\scriptstyle i'$}] (i') at (2,0){};
	\node[V, label=below:{$\scriptstyle (i+1)'$}] (i+1') at (3,0){};
	\node[V] (3') at (4,0){};
	
	\node[draw=none] (4.5) at (4.5,0.5){$\ldots$};
	
	\node[V, label=above:{$\scriptstyle k$}] (k) at (5,1){};
	\node[V, label=below:{$\scriptstyle k'$}] (k') at (5,0){};
	
	\draw (1) to (1') (2) to (2') (i) to (i+1') (i+1) to (i') (3) to (3') (k) to (k');
	
	\end{tikzpicture}\end{matrix}, \hspace{7.5mm}	
	e_{2j-1} = \begin{matrix}\begin{tikzpicture}[scale=0.7]
	
	\node[V, label=above:{$\scriptstyle 1$}] (1) at (0,1){};
	
	\node[draw=none] (0.5) at (0.5,0.5){$\ldots$};
	
	\node[V, label=below:{$\scriptstyle 1'$}] (1') at (0,0){};
	\node[V] (2) at (1,1){};
	\node[V, label=above:{$\scriptstyle j$}] (j) at (2,1){};
	\node[V] (3) at (3,1){};
	\node[V] (2') at (1,0){};
	\node[V, label=below:{$\scriptstyle j'$}] (j') at (2,0){};
	\node[V] (3') at (3,0){};
	
	\node[draw=none] (3.5) at (3.5,0.5){$\ldots$};
	
	\node[V, label=above:{$\scriptstyle k$}] (k) at (4,1){};
	\node[V, label=below:{$\scriptstyle k'$}] (k') at (4,0){};
	
	\draw (1) to (1') (2) to (2') (3) to (3') (k) to (k');
	
	\end{tikzpicture}\end{matrix}, \]

\[ e_{2i} = \begin{matrix}\begin{tikzpicture}[scale=0.7]
	
	\node[V, label=above:{$\scriptstyle 1$}] (1) at (0,1){};
	
	\node[draw=none] (0.5) at (0.5,0.5){$\ldots$};
	
	\node[V, label=below:{$\scriptstyle 1'$}] (1') at (0,0){};
	\node[V] (2) at (1,1){};
	\node[V, label=above:{$\scriptstyle i$}] (i) at (2,1){};
	\node[V, label=above:{$\scriptstyle i+1$}] (i+1) at (3,1){};
	\node[V] (3) at (4,1){};
	\node[V] (2') at (1,0){};
	\node[V, label=below:{$\scriptstyle i'$}] (i') at (2,0){};
	\node[V, label=below:{$\scriptstyle (i+1)'$}] (i+1') at (3,0){};
	\node[V] (3') at (4,0){};
	
	\node[draw=none] (4.5) at (4.5,0.5){$\ldots$};
	
	\node[V, label=above:{$\scriptstyle k$}] (k) at (5,1){};
	\node[V, label=below:{$\scriptstyle k'$}] (k') at (5,0){};
	
	\draw (1) to (1') (2) to (2') (i) to [bend right] (i+1) (i) to (i') (i') to [bend left] (i+1') (i+1) to (i+1') (3) to (3') (k) to (k');
	
	\end{tikzpicture}\end{matrix}. \]
These elements generate the monoid $(\Pi_{2k}, \circ)$, and in turn the algebra $\mathcal{A}_{2k}(z)$. Moreover, a presentation in terms of these generators, which we display below, was given in \cite[Theorem 1.11]{HR05}, see also \emph{Theorem 36} and \emph{Section 6.3} of \cite{East11}.

\begin{thm} \label{HRPres}
The partition algebra $\mathcal{A}_{2k}(z)$ has a presentation with generating set
\[ \{s_{i}, e_{j} \ | \ i \in [k-1], j \in [2k-1]\} \]
and relations
\begin{itemize}
\item[(HR1)] (Coxeter relations)
\begin{itemize}
\item[(i)] $s_{i}^{2} = 1$, for $i\in [k-1]$.
\item[(ii)] $s_{i}s_{j} = s_{j}s_{i}$, for $j\neq i+1$.
\item[(iii)] $s_{i}s_{i+1}s_{i} = s_{i+1}s_{i}s_{i+1}$, for $i\in [k-2]$.
\end{itemize}
\item[(HR2)] (Idempotent relations)
\begin{itemize}
\item[(i)] $e_{2i-1}^{2} = ze_{2i-1}$, for $i \in [k]$.
\item[(ii)] $e_{2i}^{2} = e_{2i}$, for $i \in [k-1]$.
\item[(iii)] $s_{i}e_{2i} = e_{2i}s_{i} = e_{2i}$, for $i \in [k-1]$.
\item[(iv)] $s_{i}e_{2i-1}e_{2i+1} = e_{2i-1}e_{2i+1}s_{i} = e_{2i-1}e_{2i+1}$, for $i \in [k-1]$.
\end{itemize}
\item[(HR3)] (Commutation relations)
\begin{itemize}
\item[(i)] $e_{2i-1}e_{2j-1} = e_{2j-1}e_{2i-1}$, for $i,j \in [k]$.
\item[(ii)] $e_{2i}e_{2j} = e_{2j}e_{2i}$, for $i,j \in [k-1]$.
\item[(iii)] $e_{2i-1}e_{2j} = e_{2j}e_{2i-1}$, for $j\neq i-1, i$.
\item[(iv)] $s_{i}e_{2j-1} = e_{2j-1}s_{i}$, for $j\neq i,i+1$.
\item[(v)] $s_{i}e_{2j} = e_{2j}s_{i}$, for $j\neq i-1,i+1$.
\item[(vi)] $s_{i}e_{2i-1}s_{i} = e_{2i+1}$, for $i\in [k-1]$.
\item[(vii)] $s_{i}e_{2i-2}s_{i} = s_{i-1}e_{2i}s_{i-1}$, for $i\in [k-1]$.
\end{itemize}
\item[(HR4)] (Contraction relations)
\begin{itemize}
\item[(i)] $e_{i}e_{i+1}e_{i} = e_{i}$ for $i \in [2k-2]$. 
\item[(ii)] $e_{i+1}e_{i}e_{i+1} = e_{i+1}$, for $i \in [2k-2]$.
\end{itemize}
\end{itemize}
\begin{flushright} $\square$ \end{flushright}
\end{thm}

The presentation above extends to one for the $\mathbb{C}$-algebra $\mathcal{A}_{2k}$ by simply adding $z$ as a central generator. The $\mathbb{C}$-algebra $\mathcal{A}_{2k}(\delta)$ has a presentation identical to above with the exception of replacing $z$ with $\delta$ in relation \emph{(HR2)(i)}. From the symmetry of the above presentation, one can deduce that we have an anti-involution $*:\mathcal{A}_{2k}(z) \rightarrow \mathcal{A}_{2k}(z)$ given by flipping a partition diagram up-side-down, and extending linearly over $\mathbb{C}[z]$. We denote the image of an element $a \in \mathcal{A}_{2k}(z)$ under this anti-involution by $a^{*}$.


\subsection{Jucys-Murphy elements and Enyang's presentation}


In this section we give the definition of the Jucys-Murphy elements of the partition algebra. These elements were originally defined diagrammatically by Halverson and Ram in \cite{HR05}. They were later given a recursive definition by Enyang in \cite{Eny12}. For this recursive definition, Enyang introduced new elements $\sigma_{i}$ which resemble the Coxeter generators $s_{i}$. We recall this recursive definition, and a new presentation of the partition algebra given in \cite{Eny12} in terms of the generators $e_{i}$ and $\sigma_{i}$. The following definition is the one given in \emph{Section 2.3} of \cite{Eny13}.

\begin{defn} \label{EnyJMDefn}
Let $L_{1}=0, L_{2}=e_{1}, \sigma_{2}=1$, and $\sigma_{3}=s_{1}$. Then for $i=1,2,\dots,$ define
\[ L_{2i+2} = s_{i}L_{2i}s_{i} - s_{i}L_{2i}e_{2i} - e_{2i}L_{2i}s_{i} + e_{2i}L_{2i}e_{2i+1}e_{2i} + \sigma_{2i+1}, \]
where, for $i=2,3,\dots,$ we have
\begin{align*}
\sigma_{2i+1} = s_{i-1}s_{i}\sigma_{2i-1}s_{i}&s_{i-1}+s_{i}e_{2i-2}L_{2i-2}s_{i}e_{2i-2}s_{i} + e_{2i-2}L_{2i-2}s_{i}e_{2i-2} \\
&- s_{i}e_{2i-2}L_{2i-2}s_{i-1}e_{2i}e_{2i-1}e_{2i-2} - e_{2i-2}e_{2i-1}e_{2i}s_{i-1}L_{2i-2}e_{2i-2}s_{i}.
\end{align*}
Also for $i=1,2,\dots,$ define
\[ L_{2i+1} = s_{i}L_{2i-1}s_{i} - L_{2i}e_{2i} - e_{2i}L_{2i} + (z-L_{2i-1})e_{2i} + \sigma_{2i}, \]
where, for $i=2,3,\dots,$ we have
\begin{align*}
\sigma_{2i} = s_{i-1}s_{i}\sigma_{2i-2}s_{i}&s_{i-1}+e_{2i-2}L_{2i-2}s_{i}e_{2i-2}s_{i} + s_{i}e_{2i-2}L_{2i-2}s_{i}e_{2i-2} \\
&-e_{2i-2}L_{2i-2}s_{i-1}e_{2i}e_{2i-1}e_{2i-2} - s_{i}e_{2i-2}e_{2i-1}e_{2i}s_{i-1}L_{2i-2}e_{2i-2}s_{i}.
\end{align*}
\end{defn}  

\begin{ex}
The first few non-trivial examples are
\[
L_{3} = -\begin{matrix}\begin{tikzpicture}[scale=0.7]
	\node[V] (1) at (0,1){};
	\node[V] (2) at (1,1){};
	\node[V] (1') at (0,0){};
	\node[V] (2') at (1,0){};
	
          \draw (1') to[bend left] (2') (2') to (2);
	\end{tikzpicture}\end{matrix} - 
\begin{matrix}\begin{tikzpicture}[scale=0.7]
	\node[V] (1) at (0,1){};
	\node[V] (2) at (1,1){};
	\node[V] (1') at (0,0){};
	\node[V] (2') at (1,0){};
	
          \draw (1) to[bend right] (2) (2) to (2');
	\end{tikzpicture}\end{matrix} +
z\begin{matrix}\begin{tikzpicture}[scale=0.7]
	\node[V] (1) at (0,1){};
	\node[V] (2) at (1,1){};
	\node[V] (1') at (0,0){};
	\node[V] (2') at (1,0){};
	
          \draw (1) to [bend right] (2) (2) to (2') (1) to (1') (1') to [bend left] (2');
	\end{tikzpicture}\end{matrix} + 
\begin{matrix}\begin{tikzpicture}[scale=0.7]
	\node[V] (1) at (0,1){};
	\node[V] (2) at (1,1){};
	\node[V] (1') at (0,0){};
	\node[V] (2') at (1,0){};
	
          \draw (1) to (1') (2) to (2');
	\end{tikzpicture}\end{matrix},
\]

\[ L_{4} = \begin{matrix}\begin{tikzpicture}[scale=0.7]
	\node[V] (1) at (0,1){};
	\node[V] (2) at (1,1){};
	\node[V] (1') at (0,0){};
	\node[V] (2') at (1,0){};
	
          \draw (1) to (1');
	\end{tikzpicture}\end{matrix} - 
\begin{matrix}\begin{tikzpicture}[scale=0.7]
	\node[V] (1) at (0,1){};
	\node[V] (2) at (1,1){};
	\node[V] (1') at (0,0){};
	\node[V] (2') at (1,0){};
	
          \draw (1) to (1') (1') to [bend left] (2');
	\end{tikzpicture}\end{matrix} -
\begin{matrix}\begin{tikzpicture}[scale=0.7]
	\node[V] (1) at (0,1){};
	\node[V] (2) at (1,1){};
	\node[V] (1') at (0,0){};
	\node[V] (2') at (1,0){};
	
          \draw (1) to [bend right] (2) (1) to (1');
	\end{tikzpicture}\end{matrix} + 
\begin{matrix}\begin{tikzpicture}[scale=0.7]
	\node[V] (1) at (0,1){};
	\node[V] (2) at (1,1){};
	\node[V] (1') at (0,0){};
	\node[V] (2') at (1,0){};
	
          \draw (1) to [bend right] (2) (1') to [bend left] (2');
	\end{tikzpicture}\end{matrix} + 
\begin{matrix}\begin{tikzpicture}[scale=0.7]
	\node[V] (1) at (0,1){};
	\node[V] (2) at (1,1){};
	\node[V] (1') at (0,0){};
	\node[V] (2') at (1,0){};
	
          \draw (1) to (2') (2) to (1');
	\end{tikzpicture}\end{matrix},
\]

\[ \sigma_{4} = \begin{matrix}\begin{tikzpicture}[scale=0.7]
	\node[V] (1) at (0,1){};
	\node[V] (2) at (1,1){};
	\node[V] (3) at (2,1){};
	\node[V] (1') at (0,0){};
	\node[V] (2') at (1,0){};
	\node[V] (3') at (2,0){};
	
          \draw (1) to (1') (2) to (2') (3) to (3');
	\end{tikzpicture}\end{matrix} + 
\begin{matrix}\begin{tikzpicture}[scale=0.7]
	\node[V] (1) at (0,1){};
	\node[V] (2) at (1,1){};
	\node[V] (3) at (2,1){};
	\node[V] (1') at (0,0){};
	\node[V] (2') at (1,0){};
	\node[V] (3') at (2,0){};
	
          \draw (1) to [bend right] (2) (2) to (2') (1') to [bend left] (3') (3') to (3);
	\end{tikzpicture}\end{matrix} + 
\begin{matrix}\begin{tikzpicture}[scale=0.7]
	\node[V] (1) at (0,1){};
	\node[V] (2) at (1,1){};
	\node[V] (3) at (2,1){};
	\node[V] (1') at (0,0){};
	\node[V] (2') at (1,0){};
	\node[V] (3') at (2,0){};
	
          \draw (1) to [bend right] (3) (3) to (3') (1') to [bend left] (2') (2') to (2);
	\end{tikzpicture}\end{matrix} - 
\begin{matrix}\begin{tikzpicture}[scale=0.7]
	\node[V] (1) at (0,1){};
	\node[V] (2) at (1,1){};
	\node[V] (3) at (2,1){};
	\node[V] (1') at (0,0){};
	\node[V] (2') at (1,0){};
	\node[V] (3') at (2,0){};
	
          \draw (1) to [bend right] (2) (2) to (2') (1) to (1') (1') to [bend left] (2') (3) to (3');
	\end{tikzpicture}\end{matrix} - 
\begin{matrix}\begin{tikzpicture}[scale=0.7]
	\node[V] (1) at (0,1){};
	\node[V] (2) at (1,1){};
	\node[V] (3) at (2,1){};
	\node[V] (1') at (0,0){};
	\node[V] (2') at (1,0){};
	\node[V] (3') at (2,0){};
	
          \draw (1) to [bend right] (3) (3) to (3') (1) to (1') (1') to [bend left] (3') (2) to (2');
	\end{tikzpicture}\end{matrix},
\]

\[ \sigma_{5} = \begin{matrix}\begin{tikzpicture}[scale=0.7]
	\node[V] (1) at (0,1){};
	\node[V] (2) at (1,1){};
	\node[V] (3) at (2,1){};
	\node[V] (1') at (0,0){};
	\node[V] (2') at (1,0){};
	\node[V] (3') at (2,0){};
	
          \draw (1) to (1') (2) to (3') (3) to (2');
	\end{tikzpicture}\end{matrix} + 
\begin{matrix}\begin{tikzpicture}[scale=0.7]
	\node[V] (1) at (0,1){};
	\node[V] (2) at (1,1){};
	\node[V] (3) at (2,1){};
	\node[V] (1') at (0,0){};
	\node[V] (2') at (1,0){};
	\node[V] (3') at (2,0){};
	
          \draw (1) to (2') (2') to (3) (1') to (2) (2) to (3');
	\end{tikzpicture}\end{matrix} + 
\begin{matrix}\begin{tikzpicture}[scale=0.7]
	\node[V] (1) at (0,1){};
	\node[V] (2) at (1,1){};
	\node[V] (3) at (2,1){};
	\node[V] (1') at (0,0){};
	\node[V] (2') at (1,0){};
	\node[V] (3') at (2,0){};
	
          \draw (1) to [bend right] (2) (2) to (3') (1') to [bend left] (2') (2') to (3);
	\end{tikzpicture}\end{matrix} - 
\begin{matrix}\begin{tikzpicture}[scale=0.7]
	\node[V] (1) at (0,1){};
	\node[V] (2) at (1,1){};
	\node[V] (3) at (2,1){};
	\node[V] (1') at (0,0){};
	\node[V] (2') at (1,0){};
	\node[V] (3') at (2,0){};
	
          \draw (1) to (1') (1') to [bend left] (2') (2') to (3) (2) to (3');
	\end{tikzpicture}\end{matrix} - 
\begin{matrix}\begin{tikzpicture}[scale=0.7]
	\node[V] (1) at (0,1){};
	\node[V] (2) at (1,1){};
	\node[V] (3) at (2,1){};
	\node[V] (1') at (0,0){};
	\node[V] (2') at (1,0){};
	\node[V] (3') at (2,0){};
	
          \draw (1) to (1') (1) to [bend right] (2) to (3') (2') to (3);
	\end{tikzpicture}\end{matrix}.
\]
\end{ex}

We will refer to the elements $L_{i}$ as the JM-elements, and the elements $\sigma_{i}$ as Enyang's generators. A simple proof by induction tells us the following:

\begin{lem} \label{MinAlgBelong}
For each $i \in \mathbb{N}$ we have that $L_{i} \in \mathcal{A}_{i}(z)$ and $\sigma_{i} \in \mathcal{A}_{i+1}(z)$.
\begin{flushright} $\square$ \end{flushright}
\end{lem}

It was shown in \cite{Eny12} that these elements are invariant under the automorphism $*$. They also commute with smaller partition algebras with respect to the chain described in the previous section: For any $i\leq r$, let $Z_{i}(\mathcal{A}_{r}(z)):=\langle a\in\mathcal{A}_{r}(z) \ | \ ab=ba, \ \forall \ b \in \mathcal{A}_{i}(z)\rangle$, then it was shown in \cite[Theorem 3.8]{Eny12} that
\begin{equation} \label{ComRel}
L_{i}, \sigma_{i+1} \in Z_{i-1}(\mathcal{A}_{i}(z)).
\end{equation}
In particular this shows that the JM-elements pairwise commute. We now give the new presentation of $\mathcal{A}_{2k}(z)$ established in \cite{Eny12}. This presentation is given in terms of the generators $e_{i}$ and Enyang's generator's $\sigma_{i}$. Remarkably, although the definition and diagrammatic description of the $\sigma_{i}$ is rather complicated, the defining relations in the following presentation are very simple. This is less surprising when one considers how these elements act on tensor space. This will be discussed in \emph{Section 2.4}.

\begin{thm}\label{EnyPres} \emph{{\cite[Theorem 4.1]{Eny12}}}
The partition algebra $\mathcal{A}_{2k}(z)$ has a presentation with generating set
\[ \{\sigma_{i}, e_{j} \ | \ 3\leq i \leq 2k-1, \ j \in [2k-1]\} \]
and relations:
\begin{itemize}
\item[(E1)] (Involution)
\begin{itemize}
\item[(i)] $\sigma_{2i+2}^{2}=1$ for $i\in[k-2]$.
\item[(ii)] $\sigma_{2i+1}^{2}=1$ for $i\in[k-1]$.
\end{itemize}
\item[(E2)] (Braid relations)
\begin{itemize}
\item[(i)] $\sigma_{2i+1}\sigma_{2j} = \sigma_{2j}\sigma_{2i+1}$ for $j \neq i+1$.
\item[(ii)] $\sigma_{2i+1}\sigma_{2j+1} = \sigma_{2j+1}\sigma_{2i+1}$ for $j \neq i\pm1$.
\item[(iii)] $\sigma_{2i}\sigma_{2j} = \sigma_{2j}\sigma_{2i}$ for $j \neq i\pm1$.
\item[(iv)] $s_{i}s_{i+1}s_{i} = s_{i+1}s_{i}s_{i+1}$, for $i\in[k-2]$, where 
$s_{j} =
\begin{cases}
   \sigma_{3}, & j=1 \\
   \sigma_{2j}\sigma_{2j+1}, & j>1
\end{cases}$
\end{itemize}
\item[(E3)] (Idempotent relations)
\begin{itemize}
\item[(i)] $e_{2i-1}^{2} = ze_{2i-1}$ for $i \in [k]$. 
\item[(ii)] $e_{2i}^{2} = e_{2i}$ for $2\geq i \leq k-1$.
\item[(iii)] $\sigma_{2i+1}e_{2i} = e_{2i}\sigma_{2i+1} = e_{2i}$ for $i \in [k-1]$.
\item[(iv)] $\sigma_{2i}e_{2i} = e_{2i}\sigma_{2i} = e_{2i}$ for $2\leq i \leq k-1$.
\item[(v)] $\sigma_{2i}e_{2i-1}e_{2i+1} = \sigma_{2i+1}e_{2i-1}e_{2i+1}$ for $2\leq i \leq k-1$.
\item[(vi)] $e_{2i+1}e_{2i-1}\sigma_{2i} = e_{2i+1}e_{2i-1}\sigma_{2i+1}$ for $2\leq i \leq k-1$.
\end{itemize}
\item[(E4)] (Commutation relations)
\begin{itemize}
\item[(i)] $e_{i}e_{j} = e_{j}e_{i}$, if $|i-j| \geq 2$. 
\item[(ii)] $\sigma_{2i-1}e_{2j-1} = e_{2j-1}\sigma_{2i-1}$, if $j \neq i-1,i$.
\item[(iii)] $\sigma_{2i-1}e_{2j} = e_{2j}\sigma_{2i-1}$, if $j \neq i$.
\item[(iv)] $\sigma_{2i}e_{2j-1} = e_{2j-1}\sigma_{2i}$, if $j \neq i,i+1$.
\item[(v)] $\sigma_{2i}e_{2j} = e_{2j}\sigma_{2i}$, if $j \neq i-1$.
\end{itemize}
\item[(E5)] (Contractions) 
\begin{itemize}
\item[(i)] $e_{i}e_{i+1}e_{i} = e_{i}$ and $e_{i+1}e_{i}e_{i+1} = e_{i+1}$, for $i \in [2k-2]$.
\item[(ii)] $\sigma_{2i}e_{2i-1}\sigma_{2i} = \sigma_{2i+1}e_{2i+1}\sigma_{2i+1}$, for $2\leq i \leq k-1$.
\item[(iii)] $\sigma_{2i}e_{2i-2}\sigma_{2i} = \sigma_{2i-1}e_{2i}\sigma_{2i-1}$, for $2\leq i \leq k-1$.
\end{itemize} 
\end{itemize}
\begin{flushright} $\square$ \end{flushright}
\end{thm}

Note we only worked with the elements $\sigma_{i}$ for $i\geq3$, since $\sigma_{2}=1$. The elements $s_{j}$ in the above presentation are precisely the Coxeter generators. From the involution relations we have that $s_{i}\sigma_{2i}=\sigma_{2i}s_{i}=\sigma_{2i+1}$. From \emph{\Cref{ComRel}} one can deduce that $L_{i}$ and $\sigma_{j}$ commute whenever $j\neq i-1,i,i+1$. We end this section by giving relations which tell us how the JM-elements interact with Enyang's generators when they do not commute. We use results established in \cite{Eny12}, although we have adopted the notation of \cite{Eny13}.

\begin{rmk}
The change of notation between \cite{Eny12} and \cite{Eny13} is given respectively by $p_{i} \sim e_{2i-1}$, $p_{i+\frac{1}{2}} \sim e_{2i}$, $\sigma_{i} \sim \sigma_{2i-1}$, $\sigma_{i+\frac{1}{2}} \sim \sigma_{2i}$, $L_{i} \sim L_{2i}$, and $L_{i+\frac{1}{2}} \sim L_{2i+1}$.
\end{rmk}

\begin{prop} \label{SkeinRels}
The following relations hold:
\begin{itemize}
\item[(i)] $L_{2i+1}=\sigma_{2i}L_{2i-1}\sigma_{2i}-e_{2i}e_{2i-1}\sigma_{2i}-\sigma_{2i}e_{2i-1}e_{2i}+e_{2i}e_{2i+1}\sigma_{2i}e_{2i+1}e_{2i}+\sigma_{2i}$.
\item[(ii)] $L_{2i+2}=\sigma_{2i+1}L_{2i}\sigma_{2i+1}-e_{2i}e_{2i+1}-e_{2i+1}e_{2i}+e_{2i}e_{2i+1}\sigma_{2i+1}e_{2i+1}e_{2i}+\sigma_{2i+1}$.
\item[(iii)] $L_{2i}=\sigma_{2i}L_{2i}\sigma_{2i}+e_{2i}e_{2i-1}\sigma_{2i}+\sigma_{2i}e_{2i-1}e_{2i}-e_{2i}e_{2i+1}-e_{2i+1}e_{2i}$.
\item[(iv)] $L_{2i+1}=\sigma_{2i+1}L_{2i+1}\sigma_{2i+1}-e_{2i}e_{2i+1}\sigma_{2i+1}-\sigma_{2i+1}e_{2i+1}e_{2i}+e_{2i}e_{2i+1}+e_{2i+1}e_{2i}$.
\end{itemize}
\end{prop}

\begin{proof}
\emph{(i)}: By definition,
\begin{equation}\label{Eq1}
L_{2i+1} = s_{i}L_{2i-1}s_{i}-L_{2i}e_{2i}-e_{2i}L_{2i}+(z-L_{2i-1})e_{2i}+\sigma_{2i}.
\end{equation}
We examining the right hand side term by term. For the first term we have
\[
s_{i}L_{2i-1}s_{i} = \sigma_{2i}\sigma_{2i+1}L_{2i-1}\sigma_{2i+1}\sigma_{2i}
= \sigma_{2i}\sigma_{2i+1}^{2}L_{2i-1}\sigma_{2i} 
= \sigma_{2i}L_{2i-1}\sigma_{2i}.
\]
For the second term multiplying \cite[Proposition 3.2 (3)]{Eny12} on the left by $s_{i}$ to get $\sigma_{2i}e_{2i-1}e_{2i}=L_{2i}e_{2i}$. Acting by the automorphism $*$ yields $e_{2i}e_{2i-1}\sigma_{2i}=e_{2i}L_{2i}$ for the third term. Lastly for the forth term
\begin{align*}
(z-L_{2i-1})e_{2i} &= (z-L_{2i-1})e_{2i}e_{2i+1}e_{2i} \\
&= e_{2i}(z-L_{2i-1})e_{2i+1}e_{2i} \\
&= e_{2i}e_{2i+1}\sigma_{2i}e_{2i+1}e_{2i}
\end{align*}
where the last equality follows by \cite[Proposition 4.3 (2)]{Eny12}. Substituting these terms back into \emph{\Cref{Eq1}} yields \emph{(i)}.

\vspace{2mm}
\noindent
\emph{(ii)}: By definition,
\begin{equation}\label{Eq2}
L_{2i+2} = s_{i}L_{2i}s_{i} - s_{i}L_{2i}e_{2i} - e_{2i}L_{2i}s_{i} + e_{2i}L_{2i}e_{2i+1}e_{2i} + \sigma_{2i+1}.
\end{equation}
Multiplying this equation on the left and right by $\sigma_{2i}$ gives
\begin{align*}
L_{2i+2} &= \sigma_{2i+1}L_{2i}\sigma_{2i+1}-\sigma_{2i+1}L_{2i}e_{2i}-e_{2i}L_{2i}\sigma_{2i+1}+e_{2i}L_{2i}e_{2i+1}e_{2i}+\sigma_{2i+1} \\
&= \sigma_{2i+1}L_{2i}\sigma_{2i+1}-\sigma_{2i+1}^{2}e_{2i+1}e_{2i}-e_{2i}e_{2i+1}\sigma_{2i+1}^{2}+e_{2i}e_{2i+1}\sigma_{2i+1}e_{2i+1}e_{2i}+\sigma_{2i+1} \\
&=  \sigma_{2i+1}L_{2i}\sigma_{2i+1}-e_{2i+1}e_{2i}-e_{2i}e_{2i+1}+e_{2i}e_{2i+1}\sigma_{2i+1}e_{2i+1}e_{2i}+\sigma_{2i+1}
\end{align*}
which gives \emph{(ii)}, where the second equality follows by relation $\sigma_{2i+1}e_{2i+1}e_{2i}=L_{2i}e_{2i}$ and its dual $e_{2i}e_{2i+1}\sigma_{2i+1}=e_{2i}L_{2i}$, which follow from \cite[Proposition 3.2 (3)]{Eny12}.

\vspace{2mm}
\noindent
\emph{(iii)}: It was shown in \cite[Proposition 3.10]{Eny12} that the element $L_{1}+L_{2}+\dots+L_{r}$ belongs to the center of $\mathcal{A}_{r}(z)$. From this, and the fact that $L_{i}$ and $\sigma_{j}$ commute whenever $j\neq i-1,i,i+1$, one may deduce that
\[ \sigma_{2i}(L_{2i-1}+L_{2i}+L_{2i+1})\sigma_{2i} = L_{2i-1}+L_{2i}+L_{2i+1}. \]
Rearranging gives
\begin{equation}\label{Eq3}
L_{2i} = \sigma_{2i}L_{2i}\sigma_{2i} + (\sigma_{2i}L_{2i-1}\sigma_{2i} - L_{2i+1}) + (\sigma_{2i}L_{2i+1}\sigma_{2i} - L_{2i-1}).
\end{equation}
We examine the bracketed terms in \emph{\Cref{Eq3}}. Rearranging \emph{(i)} gives the first bracketed term as
\[ \sigma_{2i}L_{2i-1}\sigma_{2i} - L_{2i+1} = e_{2i}e_{2i-1}\sigma_{2i}+\sigma_{2i}e_{2i-1}e_{2i}-e_{2i}e_{2i+1}\sigma_{2i}e_{2i+1}e_{2i}-\sigma_{2i}. \]
Multiplying this on the left and right by $\sigma_{2i}$, and then rearranging gives the second bracketed term
\[ \sigma_{2i}L_{2i+1}\sigma_{2i} - L_{2i-1} = -e_{2i}e_{2i-1}-e_{2i-1}e_{2i}+e_{2i}e_{2i+1}\sigma_{2i}e_{2i+1}e_{2i}+\sigma_{2i}. \]
Substituting these back into equation \eqref{Eq3} yields \emph{(iii)}.

\vspace{2mm}
\noindent
\emph{(iv)}: Analogously to the previous case, one can deduce that
\[ \sigma_{2i+1}(L_{2i}+L_{2i+1}+L_{2i+2})\sigma_{2i+1} = L_{2i}+L_{2i+1}+L_{2i+2}. \]
Rearranging gives
\begin{equation}\label{Eq4}
L_{2i+1} = \sigma_{2i+1}L_{2i+1}\sigma_{2i+1} + (\sigma_{2i+1}L_{2i}\sigma_{2i+1} - L_{2i+2}) + (\sigma_{2i+1}L_{2i+2}\sigma_{2i+1} - L_{2i}).
\end{equation}
We examine the bracketed terms in \emph{\Cref{Eq4}}. Rearranging \emph{(2)(ii)} gives the first bracketed term as
\[ \sigma_{2i+1}L_{2i}\sigma_{2i+1} - L_{2i+2} = e_{2i}e_{2i+1}+e_{2i+1}e_{2i}-e_{2i}e_{2i+1}\sigma_{2i+1}e_{2i+1}e_{2i}-\sigma_{2i+1}. \]
Multiplying this on the left and right by $\sigma_{2i+1}$, and then rearranging gives the second bracketed term
\[ \sigma_{2i+1}L_{2i+2}\sigma_{2i+1} - L_{2i} = -e_{2i}e_{2i+1}\sigma_{2i+1}-\sigma_{2i+1}e_{2i+1}e_{2i}+e_{2i}e_{2i+1}\sigma_{2i+1}e_{2i+1}e_{2i}+\sigma_{2i+1}. \]
Substituting these back into equation \eqref{Eq4} yields \emph{(iv)}.

\end{proof}


\subsection{Normalisation}


As mentioned in the introduction, we seek to \lq affinize' the partition algebra by replacing the Jucys-Murphy elements with commuting variables, and asking them to retain various relations with the generators. In preparation for this construction, this section collects all the relations we seek to retain in one place. However, instead of working with the JM-elements and Enyang's generators, it turns out to be easier to work with the following elements: For each $i \in \mathbb{N}$ we set
\[ t_{2i} := \sigma_{2i} - e_{2i}, \hspace{4mm} t_{2i+1} := \sigma_{2i+1} - e_{2i}. \]
For each $i \in \mathbb{N}$ we set
\[ X_{i} := \begin{cases}
                 z-1-L_{i}, & \text{ if $i$ odd} \\
                 L_{i}-1, & \text{ if $i$ even} 
               \end{cases} \]
We also call the elements $X_{i}$ the JM-elements and the elements $t_{i}$ Enyang's generators. By definitions we have that $t_{i}\in\mathcal{A}_{i+1}(z)$, $X_{i}\in\mathcal{A}_{i}(z)$, and that $t_{i}^{*}=t_{i}$ and $X_{i}^{*}=X_{i}$. One can also deduce that $s_{i}t_{2i}=t_{2i}s_{i}=t_{2i+1}$. We briefly collect some simple relations to ease the proof of the proceeding proposition.  

\begin{lem} \label{NewEnyMixRels}
The following relations hold:
\begin{itemize}
\item[(i)] $e_{2i+1}t_{2i}e_{2i+1} = X_{2i-1}e_{2i+1}$
\item[(ii)] $t_{2i}e_{2i-1}e_{2i}=X_{2i}e_{2i}$, and $e_{2i}e_{2i-1}t_{2i}=e_{2i}X_{2i}$
\item[(iii)] $t_{2i+1}e_{2i+1}e_{2i}=X_{2i}e_{2i}$, and $e_{2i}e_{2i+1}t_{2i+1}=e_{2i}X_{2i}$
\end{itemize}
\end{lem}

\begin{proof}
\emph{(i)}: We have that
\begin{align*}
e_{2i+1}t_{2i}e_{2i+1} &= e_{2i+1}(\sigma_{2i}-e_{2i})e_{2i+1} \\
&= e_{2i+1}\sigma_{2i}e_{2i+1}-e_{2i+1} \hspace{13mm} \text{by \emph{(E5)}} \\
&= (z-L_{2i-1})e_{2i+1}-e_{2i+1} \hspace{8mm} \text{by \cite[Proposition 4.3 (2)]{Eny12}} \\
&= (X_{2i-1}+1)e_{2i+1}-e_{2i+1} \\
&= X_{2i-1}e_{2i+1}
\end{align*}
\noindent
\emph{(ii)}: We have that
\begin{align*}
t_{2i}e_{2i-1}e_{2i} &= (\sigma_{2i}-e_{2i})e_{2i-1}e_{2i} \\
&= \sigma_{2i}e_{2i-1}e_{2i} - e_{2i} \hspace{13mm} \text{by \emph{(E5)}} \\
&= L_{2i}e_{2i} - e_{2i} \hspace{20.5mm} \text{by \cite[Proposition 3.2 (3)]{Eny12}} \\
&= (X_{2i}+1)e_{2i}-e_{2i} \\
&= X_{2i}e_{2i}
\end{align*}
The relation $e_{2i}e_{2i-1}t_{2i}=e_{2i}X_{2i}$ is obtained by acting by $*$.

\vspace{2mm}
\noindent
\emph{(iii)}: We have $t_{2i+1}e_{2i+1}e_{2i}=t_{2i}s_{i}e_{2i+1}e_{2i} = t_{2i}e_{2i-1}e_{2i} = X_{2i}e_{2i}$. Again the relation $e_{2i}e_{2i+1}t_{2i+1} = e_{2i}X_{2i}$ is obtained by acting by $*$.

\end{proof}

The following proposition contains all the relations we seek to retain for our construction of the affine partition algebra, as such some are identical to relations we have already stated. It provides a presentation of the partition algebra $\mathcal{A}_{2k}(z)$ which is simply Enyang's presentation \emph{\Cref{EnyPres}} except working with the generators $t_{i}$ instead of $\sigma_{i}$. For those relations we have adopted the same naming conventions given in \emph{\Cref{EnyPres}}.

\begin{prop} \label{AffPrep}
The partition algebra $\mathcal{A}_{2k}(z)$ has a presentation with generating set
\[ \{t_{i}, e_{j} \ | \ 3\leq i \leq 2k-1, \ j \in [2k-1]\} \]
and relations:
\begin{itemize}
\item[(1)] (Involutions)
\begin{itemize}
\item[(i)] $t_{2i+2}^{2} = 1-e_{2i}$, for $i\in [k-2]$.
\item[(ii)] $t_{2i+1}^{2} = 1-e_{2i}$, for $i\in [k-1]$.
\end{itemize}
\item[(2)] (Braid relations)
\begin{itemize}
\item[(i)] $t_{2i+1}t_{2j} = t_{2j}t_{2i+1}$ for $j \neq i+1$.
\item[(ii)] $t_{2i+1}t_{2j+1} = t_{2j+1}t_{2i+1}$ for $j \neq i\pm1$.
\item[(iii)] $t_{2i}t_{2j} = t_{2j}t_{2i}$ for $j \neq i\pm1$.
\item[(iv)] $s_{i}s_{i+1}s_{i} = s_{i+1}s_{i}s_{i+1}$, for $i\in[k-2]$, where 
$s_{j} = 
\begin{cases}
    t_{3}+e_{2}, & j=1 \\
    t_{2j}t_{2j+1}+e_{2j}, & j>1
\end{cases}$.
\end{itemize}
\item[(3)] (Idempotent relations)
\begin{itemize}
\item[(i)] $e_{2i-1}^{2} = ze_{2i-1}$ for $i \in [k]$. 
\item[(ii)] $e_{2i}^{2} = e_{2i}$ for $i \in [k-1]$.
\item[(iii)] $t_{2i+1}e_{2i} = e_{2i}t_{2i+1} = 0$ for $i\in [k-1]$.
\item[(iv)] $t_{2i}e_{2i} = e_{2i}t_{2i} = 0$ for $2\leq i \leq k-1$.
\item[(v)] $t_{2i}e_{2i-1}e_{2i+1} = t_{2i+1}e_{2i-1}e_{2i+1}$ for $2\leq i \leq k-1$.
\item[(vi)] $e_{2i+1}e_{2i-1}t_{2i} = e_{2i+1}e_{2i-1}t_{2i+1}$ for $2\leq i \leq k-1$.
\end{itemize}
\item[(4)] (Commutation relations)
\begin{itemize}
\item[(i)] $e_{i}e_{j} = e_{j}e_{i}$, if $|i-j| \geq 2$. 
\item[(ii)] $t_{2i-1}e_{2j-1} = e_{2j-1}t_{2i-1}$, if $j \neq i-1,i$.
\item[(iii)] $t_{2i-1}e_{2j} = e_{2j}t_{2i-1}$, if $j \neq i$.
\item[(iv)] $t_{2i}e_{2j-1} = e_{2j-1}t_{2i}$, if $j \neq i,i+1$.
\item[(v)] $t_{2i}e_{2j} = e_{2j}t_{2i}$, if $j \neq i-1$.
\end{itemize}
\item[(5)] (Contractions)
\begin{itemize}
\item[(i)] $e_{i}e_{i+1}e_{i} = e_{i}$ and $e_{i+1}e_{i}e_{i+1} = e_{i+1}$, for $i \in [2k-2]$.
\item[(ii)] $t_{2i}e_{2i-1}t_{2i} = t_{2i+1}e_{2i+1}t_{2i+1}$, for $i \in [k-1]$.
\item[(iii)] $t_{2i}e_{2i-2}t_{2i} = t_{2i-1}e_{2i}t_{2i-1}$, for $2\leq i \leq k-1$.
\end{itemize}
\end{itemize}
Furthermore, the following relations are satisfied in $\mathcal{A}_{2k}(z)$:
\begin{itemize}
\item[(6)] (JM Commutation Relations)
\begin{itemize}
\item[(i)] $X_{i}X_{j}=X_{j}X_{i}$ for all $i,j \in [2k]$
\item[(ii)] $t_{i}X_{j} = X_{j}t_{i}$ for $j \neq i-1, i, i+1$
\item[(iii)] $e_{i}X_{j}=X_{j}e_{i}$ for $j\neq i, i+1$
\end{itemize}
\item[(7)] (Braid-like Relations)
\begin{itemize}
\item[(i)] $t_{2i-2}t_{2i}t_{2i-2} = t_{2i}t_{2i-2}t_{2i}(1-e_{2i-2})$
\item[(ii)] $t_{2i+1}t_{2i-1}t_{2i+1} = t_{2i-1}t_{2i+1}t_{2i-1}(1-e_{2i})$
\item[(iii)] $t_{2i-1}t_{2i}t_{2i-1} = t_{2i}-e_{2i-2}t_{2i}-t_{2i}e_{2i-2}$
\item[(iv)] $t_{2i}t_{2i-1}t_{2i} = t_{2i-1}-e_{2i}t_{2i-1}-t_{2i-1}e_{2i}$
\end{itemize}
\item[(8)] (Skein-like Relations)
\begin{itemize}
\item[(i)] $X_{2i+1}=t_{2i}X_{2i-1}t_{2i}+e_{2i}e_{2i-1}t_{2i}+t_{2i}e_{2i-1}e_{2i}-t_{2i}$.
\item[(ii)] $X_{2i+2}=t_{2i+1}X_{2i}t_{2i+1}+e_{2i}e_{2i+1}t_{2i+1}e_{2i+1}e_{2i}+t_{2i+1}$.
\item[(iii)] $X_{2i}=t_{2i}X_{2i}t_{2i}+e_{2i}e_{2i-1}t_{2i}+t_{2i}e_{2i-1}e_{2i}$.
\item[(iv)] $X_{2i+1}=t_{2i+1}X_{2i+1}t_{2i+1}+e_{2i}e_{2i+1}t_{2i+1}+t_{2i+1}e_{2i+1}e_{2i}$.
\end{itemize}
\item[(9)] (Anti-symmetry Relations)
\begin{itemize}
\item[(i)] $e_{i}(X_{i}-X_{i+1}) = 0$ for $i\in [2k-1]$.
\item[(ii)] $(X_{i}-X_{i+1})e_{i} = 0$ for $i\in [2k-1]$.
\end{itemize}
\item[(10)] (Bubble Relations)
\begin{itemize}
\item[(i)] $e_{1}X_{1}^{l}e_{1} = z(z-1)^{l}e_{1}$, for all $l \in \mathbb{Z}_{\geq 0}$.
\end{itemize}
\end{itemize}

\end{prop}

\begin{proof}
Although lengthy, it is simple to check that relations (1) to (5) give an alternative presentation for $\mathcal{A}_{2k}(z)$ since we merely exchanged the elements $\sigma_{i}$ with $t_{i}$ from Enyang's presentation given in \emph{\Cref{EnyPres}}. 

\noindent
\emph{(6)}: Follows from \emph{\Cref{ComRel}} and \emph{\Cref{MinAlgBelong}}.

\noindent
\emph{(7)}: These relations will be proven in the next section in \emph{\Cref{BraidLikeRels}}.

\noindent
\emph{(9)}: Follows from \cite[Proposition 3.9]{Eny12} (1) and (2).

\noindent
\emph{(10)}: We have that $X_{1} = z-1-L_{1} = z-1$. Thus for any $l\in \mathbb{N}$,
\[ e_{1}X_{1}^{l}e_{1} = (z-1)^{l}e_{1}^{2} = z(z-1)^{l}e_{1}. \]

\noindent
\emph{(8)(i)}: From \emph{\Cref{SkeinRels} (i)} we have
\begin{equation} \label{Eq5}
L_{2i+1}=\sigma_{2i}L_{2i-1}\sigma_{2i}-e_{2i}e_{2i-1}\sigma_{2i}-\sigma_{2i}e_{2i-1}e_{2i}+e_{2i}e_{2i+1}\sigma_{2i}e_{2i+1}e_{2i}+\sigma_{2i}.
\end{equation}
Examining the right hand side term by term: For the first term,
\begin{align*}
\sigma_{2i}L_{2i-1}\sigma_{2i} &= (t_{2i}+e_{2i})(-X_{2i-1})(t_{2i}+e_{2i}) + (z-1) \\
&= -t_{2i}X_{2i-1}t_{2i}-t_{2i}X_{2i-1}e_{2i}-e_{2i}X_{2i-1}t_{2i}-e_{2i}X_{2i-1}e_{2i} + (z-1) \\
&= -t_{2i}X_{2i-1}t_{2i} -X_{2i-1}e_{2i} + (z-1)
\end{align*}
where the last equality follows since $X_{2i-1}$ commutes with $e_{2i}$ and $t_{2i}e_{2i}=e_{2i}t_{2i}=0$. For the second and third term of \emph{\Cref{Eq5}}, we have
\[
-e_{2i}e_{2i-1}\sigma_{2i} = -e_{2i}-e_{2i}e_{2i-1}t_{2i}, \hspace{3mm} \text{and} \hspace{3mm}
-\sigma_{2i}e_{2i-1}e_{2i} = -t_{2i}e_{2i-1}e_{2i}-e_{2i}.
\]
For the forth term of \emph{\Cref{Eq5}},
\begin{align*}
e_{2i}e_{2i+1}\sigma_{2i}e_{2i+1}e_{2i} &= e_{2i}e_{2i+1}t_{2i}e_{2i+1}e_{2i} + e_{2i}e_{2i+1}e_{2i}e_{2i+1}e_{2i} \\
&= e_{2i}e_{2i+1}t_{2i}e_{2i+1}e_{2i} + e_{2i} \\
&= e_{2i}e_{2i-1}t_{2i+1}e_{2i+1}e_{2i} +e_{2i} \hspace{13mm} \text{by $t_{2i}=s_{i}t_{2i+1}$} \\
&= e_{2i}e_{2i-1}X_{2i+1}e_{2i} + e_{2i} \hspace{19mm} \text{by \emph{\Cref{NewEnyMixRels} (iii)}} \\
&=  e_{2i}e_{2i-1}e_{2i}X_{2i-1}+e_{2i} \hspace{19mm} \text{by \emph{(9)(i), (ii)}} \\
&= e_{2i}X_{2i-1}+e_{2i}.
\end{align*}
Substituting all these back into \emph{\Cref{Eq5}} yields
\begin{align*}
z-1-X_{2i+1} &= -t_{2i}X_{2i-1}t_{2i} -X_{2i-1}e_{2i}+(z-1)-e_{2i}-e_{2i}e_{2i-1}t_{2i}-t_{2i}e_{2i-1}e_{2i}-e_{2i} \\
& \hspace{12mm} +e_{2i}X_{2i-1} + e_{2i}+t_{2i}+e_{2i} \\
\iff X_{2i+1} &= t_{2i}X_{2i-1}t_{2i}+e_{2i}e_{2i-1}t_{2i}+t_{2i}e_{2i-1}e_{2i}-t_{2i}
\end{align*}
giving \emph{(8)(i)}.

\noindent
\emph{(8)(ii)}: From \emph{\Cref{SkeinRels} (ii)} we have
\begin{equation} \label{Eq6}
L_{2i+2}=\sigma_{2i+1}L_{2i}\sigma_{2i+1}-e_{2i}e_{2i+1}-e_{2i+1}e_{2i}+e_{2i}e_{2i+1}\sigma_{2i+1}e_{2i+1}e_{2i}+\sigma_{2i+1}.
\end{equation}
We examine two terms on the right hand side: The first term gives
\begin{align*}
\sigma_{2i+1}L_{2i}\sigma_{2i+1} &= (t_{2i+1}+e_{2i})(X_{2i}+1)(t_{2i+1}+e_{2i}) \\
&= t_{2i+1}X_{2i}t_{2i+1}+t_{2i+1}X_{2i}e_{2i}+e_{2i}X_{2i}t_{2i+1}+e_{2i}X_{2i}e_{2i}+1 \\
&= t_{2i+1}X_{2i}t_{2i+1}+t_{2i+1}^{2}e_{2i+1}e_{2i}+e_{2i}e_{2i+1}t_{2i+1}^{2}+1 \\
&= t_{2i+1}X_{2i}t_{2i+1}+e_{2i+1}e_{2i}+e_{2i}e_{2i+1}-2e_{2i}+1
\end{align*}
where the second equality follows since $(t_{2i+1}+e_{2i})^{2}=1$, and the third from \emph{\Cref{NewEnyMixRels} (iii)} and since $e_{2i}X_{2i}e_{2i}=e_{2i}e_{2i-1}t_{2i}e_{2i}=0$. The forth term in \emph{\Cref{Eq6}} gives
\begin{align*}
e_{2i}e_{2i+1}\sigma_{2i+1}e_{2i+1}e_{2i} &= e_{2i}e_{2i+1}t_{2i+1}e_{2i+1}e_{2i} + e_{2i}e_{2i+1}e_{2i}e_{2i+1}e_{2i} \\
&= e_{2i}e_{2i+1}t_{2i+1}e_{2i+1}e_{2i} + e_{2i}.
\end{align*}
 Substituting these back into equation \emph{\Cref{Eq6}} yields
\begin{align*}
X_{2i+2}+1 &= t_{2i+1}X_{2i}t_{2i+1}+e_{2i+1}e_{2i}+e_{2i}e_{2i+1}-2e_{2i}+1 \\
& \hspace{8mm} -e_{2i}e_{2i+1}-e_{2i+1}e_{2i}+e_{2i}e_{2i+1}t_{2i+1}e_{2i+1}e_{2i} + e_{2i} + t_{2i+1}+e_{2i} \\
\iff X_{2i+2} &=  t_{2i+1}X_{2i}t_{2i+1}+e_{2i}e_{2i+1}t_{2i+1}e_{2i+1}e_{2i}+t_{2i+1}
\end{align*}
giving \emph{(8)(ii)}.

\noindent
\emph{(8)(iii)}: From \emph{\Cref{SkeinRels} (iii)} we have
\begin{equation} \label{Eq7}
L_{2i}=\sigma_{2i}L_{2i}\sigma_{2i}+e_{2i}e_{2i-1}\sigma_{2i}+\sigma_{2i}e_{2i-1}e_{2i}-e_{2i}e_{2i+1}-e_{2i+1}e_{2i}.
\end{equation}
We have that
\begin{align*}
\sigma_{2i}L_{2i}\sigma_{2i} &= (t_{2i}+e_{2i})(X_{2i}+1)(t_{2i}+e_{2i}) \\
&= t_{2i}X_{2i}t_{2i}+t_{2i}X_{2i}e_{2i}+e_{2i}X_{2i}t_{2i}+e_{2i}X_{2i}e_{2i} +1 \\
&= t_{2i}X_{2i}t_{2i}+t_{2i}^{2}e_{2i-1}e_{2i}+e_{2i}e_{2i-1}t_{2i}^{2}+1 \\
&= t_{2i}X_{2i}t_{2i}+e_{2i-1}e_{2i}+e_{2i}e_{2i-1}-2e_{2i}+1
\end{align*}
where the second equality follows since $(t_{2i}+e_{2i})^{2}=1$, and the third equality from \emph{\Cref{NewEnyMixRels} (ii)} and the since $t_{2i}e_{2i}=e_{2i}t_{2i}=0$. Substituting this, and relations
\[
e_{2i}e_{2i-1}\sigma_{2i} = e_{2i}e_{2i-1}t_{2i}+e_{2i} \hspace{3mm} \text{and} \hspace{3mm}
\sigma_{2i}e_{2i-1}e_{2i} = t_{2i}e_{2i-1}e_{2i}+e_{2i},
\]
back into equation \emph{\Cref{Eq7}} yields
\begin{align*}
X_{2i}+1 &= t_{2i}X_{2i}t_{2i}+e_{2i-1}e_{2i}+e_{2i}e_{2i-1}-2e_{2i}+1+e_{2i}e_{2i-1}t_{2i}+e_{2i} \\
& \hspace{12mm} +t_{2i}e_{2i-1}e_{2i}+e_{2i}-e_{2i}e_{2i+1}-e_{2i+1}e_{2i} \\
\iff X_{2i} &= t_{2i}X_{2i}t_{2i}+e_{2i}e_{2i-1}t_{2i}+t_{2i}e_{2i-1}e_{2i}+1
\end{align*}
giving \emph{(8)(iii)}.

\noindent
\emph{(8)(iv)}: From \emph{\Cref{SkeinRels} (iv)} we have
\begin{equation} \label{Eq8}
L_{2i+1}=\sigma_{2i+1}L_{2i+1}\sigma_{2i+1}-e_{2i}e_{2i+1}\sigma_{2i+1}-\sigma_{2i+1}e_{2i+1}e_{2i}+e_{2i}e_{2i+1}+e_{2i+1}e_{2i}.
\end{equation}
We have that
\begin{align*}
\sigma_{2i+1}L_{2i+1}\sigma_{2i+1} &= (t_{2i+1}+e_{2i})(-X_{2i+1})(t_{2i+1}+e_{2i}) + (z-1) \\
&= -t_{2i+1}X_{2i+1}t_{2i+1}-t_{2i+1}X_{2i+1}e_{2i}-e_{2i}X_{2i+1}t_{2i+1}-e_{2i}X_{2i+1}e_{2i}+(z-1) \\
&= -t_{2i+1}X_{2i+1}t_{2i+1}-t_{2i+1}^{2}e_{2i+1}e_{2i}-e_{2i}e_{2i+1}t_{2i+1}^{2}+(z-1) \\
&= -t_{2i+1}X_{2i+1}t_{2i+1}-e_{2i+1}e_{2i}-e_{2i}e_{2i+1}+2e_{2i}+(z-1)
\end{align*}
where the third equality follows from \emph{\Cref{NewEnyMixRels} (iii)}, and noting that $e_{2i}X_{2i+1}e_{2i}=e_{2i}X_{2i}e_{2i}=e_{2i}e_{2i-1}t_{2i}e_{2i}=0$. Substituting this, and the equations
\[
-e_{2i}e_{2i+1}\sigma_{2i+1} = -e_{2i}e_{2i+1}t_{2i+1}-e_{2i} \hspace{3mm} \text{and} \hspace{3mm}
\sigma_{2i}e_{2i-1}e_{2i} = -t_{2i+1}e_{2i+1}e_{2i}-e_{2i},
\]
back into equation \emph{\Cref{Eq8}} yields
\begin{align*}
(z-1)-X_{2i+1} &= -t_{2i+1}X_{2i+1}t_{2i+1}-e_{2i+1}e_{2i}-e_{2i}e_{2i+1}+2e_{2i}+(z-1)-e_{2i}e_{2i+1}t_{2i+1} \\
& \hspace{15mm} -e_{2i}-t_{2i+1}e_{2i+1}e_{2i}-e_{2i}+e_{2i}e_{2i+1}+e_{2i+1}e_{2i} \\
\iff X_{2i+1} &= t_{2i+1}X_{2i+1}t_{2i+1}+t_{2i+1}e_{2i+1}e_{2i}+e_{2i}e_{2i+1}t_{2i+1}
\end{align*}
giving \emph{(8)(iv)}.

\end{proof}


\subsection{Schur-Weyl Duality}


In this section we recall the Schur-Weyl duality between the partition algebra $\mathcal{A}_{2k}(n)$ and the group algebra of the symmetric group $\mathbb{C}S(n)$ via their actions on tensor space. We will also highlight how $X_{i}$ and $t_{i}$ act on this tensor space, and complete the proof of \emph{\Cref{AffPrep}}. Consider the permutation module $V=\text{Span}_{\mathbb{C}}\{v_{1},\dots,v_{n}\}$ of $\mathbb{C}S(n)$ with action given by $\pi v_{a} = v_{\pi(a)}$ for all $\pi \in S(n)$ and $a\in[n]$, extended $\mathbb{C}$-linearly to $\mathbb{C}S(n)$. For $k\geq 0$, the tensor space
\[ V^{\otimes k} := V\otimes \dots \otimes V \hspace{5mm} \text{($k$ tensor components)} \]
is an $\mathbb{C}S(n)$-module via the diagonal action. For any $k$-tuple $\bm{a}=(a_{1},\dots,a_{k})\in[n]^{k}$ we let $v_{\bm{a}}:=v_{a_{1}}\otimes \dots\otimes v_{a_{k}} \in V^{\otimes k}$, and for any $\pi \in S(n)$ let $\pi(\bm{a}):=(\pi(a_{1}),\dots,\pi(a_{k}))$. Then $V^{\otimes k} = \text{Span}_{\mathbb{C}}\{v_{\bm{a}}\ | \ \bm{a} \in [n]^{k}\}$, and the diagonal action is given by the $\mathbb{C}$-linear extension of $\pi v_{\bm{a}} = v_{\pi(\bm{a})}$ for all $\pi\in S(n)$ and $\bm{a}\in[n]^{k}$. We let End$(V^{\otimes k})$ be the algebra of all vector space endomorphims $V^{\otimes k}\rightarrow V^{\otimes k}$. We identify any $g\in \mathbb{C}S(n)$ with the corresponding endomorphism in End$(V^{\otimes k})$ given by the diagonal action. Then consider the subalgebra
\[ \text{End}_{S(n)}(V^{\otimes k}) := \{ f \in \mathsf{End}(V^{\otimes k}) \ | \ f\pi = \pi f, \text{ for all } \pi \in S(n)\} \] 
of all $S(n)$ commuting endomorphisms. For the following result see \cite[Section 3]{HR05}.

\begin{thm}\label{SWD}
For any $n,k\geq0$, we have a surjective $\mathbb{C}$-algebra homomorphism
\[ \psi_{n,k}:\mathcal{A}_{2k} \rightarrow \mathsf{End}_{S(n)}(V^{\otimes k}) \]
defined on the generators $z$, $s_{i}$, and $e_{j}$, by letting $z\mapsto n$ and
\begin{align*}
\psi_{n,k}(s_{i})(v_{\bm{a}}) &= v_{a_{1}}\otimes \dots \otimes v_{a_{i-1}}\otimes v_{a_{i+1}}\otimes v_{a_{i}}\otimes v_{a_{i+2}}\otimes \dots \otimes v_{a_{k}} \\
\psi_{n,k}(e_{2j-1})(v_{\bm{a}}) &= \sum_{b=1}^{n} v_{a_{1}}\otimes \dots \otimes v_{a_{j-1}}\otimes v_{b}\otimes v_{a_{j+1}}\otimes \dots \otimes v_{a_{k}} \\
\psi_{n,k}(e_{2j})(v_{\bm{a}}) &= \delta_{a_{j},a_{j+1}}v_{\bm{a}}
\end{align*}
for all $\bm{a}=(a_{1},\dots,a_{k})\in [n]^{k}$, $i\in[k-1]$, $j\in[2k-1]$, where $\delta_{a,b}$ is the Kronecker delta. Moreover, we have that $\mathsf{Ker}(\psi_{n,k}) = (z-n)$ if and only if $n\geq 2k$, in which case $\mathcal{A}_{2k}(n)\cong\mathsf{End}_{S(n)}(V^{\otimes k})$.
\end{thm}

For any $a,b \in [n]$, we let $(a,b) \in S(n)$ denote the transposition exchanging $a$ and $b$, and let $\varepsilon_{a,b} := 1-\delta_{a,b}$. We now recall how the elements $X_{i}$ and $t_{i}$ act under $\psi_{n,k}$, which was proven in \cite{Eny12}.

\begin{prop} \label{CrossingAction} \emph{{\cite[Proposition 5.2]{Eny12}}}
For any $\bm{a}\in[n]^{k}$, we have
\begin{align*}
\psi_{n,k}(t_{2i})(v_{\bm{a}}) &= \varepsilon_{a_{i},a_{i+1}}(a_{i},a_{i+1})(v_{a_{1}}\otimes \dots \otimes v_{a_{i-1}})\otimes v_{a_{i}}\otimes \dots \otimes v_{a_{k}} \\
 \psi_{n,k}(t_{2i+1})(v_{\bm{a}}) &= \varepsilon_{a_{i},a_{i+1}}(a_{i},a_{i+1})(v_{a_{1}}\otimes \dots \otimes v_{a_{i+1}})\otimes v_{a_{i+2}}\otimes \dots \otimes v_{a_{k}}
\end{align*}
for all $\bm{a}=(a_{1},\dots,a_{k})\in [n]^{k}$, and $i\in [k-1]$.
\end{prop}

\begin{prop} \label{JMAction} \emph{{\cite[Proposition 5.3]{Eny12}}}
For any $\bm{a}\in[n]^{k}$, we have
\begin{align*}
\psi_{n,k}(X_{2i-1})(v_{\bm{a}}) &= \sum\limits_{\substack{b=1 \\ b\neq a_{i}}}^{n}(a_{i},b)(v_{a_{1}}\otimes \dots \otimes v_{a_{i-1}})\otimes v_{a_{i}}\otimes \dots \otimes v_{a_{k}} \\
 \psi_{n,k}(X_{2i})(v_{\bm{a}}) &= \sum\limits_{\substack{b=1 \\ b\neq a_{i}}}^{n}(a_{i},b)(v_{a_{1}}\otimes \dots \otimes v_{a_{i}})\otimes v_{a_{i+1}}\otimes \dots \otimes v_{a_{k}}
\end{align*}
for all $\bm{a}=(a_{1},\dots,a_{k})\in [n]^{k}$, and $i\in [k]$.
\end{prop}

The following result tells us that a relation holds in $\mathcal{A}_{2k}$ if and only if it holds under $\psi_{n,k}$ for all $n$. We will use this result to complete \emph{\Cref{AffPrep}}.

\begin{lem}\label{TrivKer}
Let $R_{1},R_{2}\in\mathcal{A}_{2k}$. If $\psi_{n,k}(R_{1})=\psi_{n,k}(R_{2})$ for all $n\geq 1$, then $R_{1}=R_{2}$.
\end{lem}

\begin{proof}
Follows since
\[ R_{1}-R_{2} \in \bigcap_{n=1}^{\infty}\mathsf{Ker}(\psi_{n,k})\subset \bigcap_{n\geq 2k}(z-n)=0. \]

\end{proof}

\begin{lem}\label{BraidLikeRels}
The relations
\begin{itemize}
\item[(7)] (Braid-like relations)
\begin{itemize}
\item[(i)] $t_{2i-2}t_{2i}t_{2i-2} = t_{2i}t_{2i-2}t_{2i}(1-e_{2i-2})$
\item[(ii)] $t_{2i+1}t_{2i-1}t_{2i+1} = t_{2i-1}t_{2i+1}t_{2i-1}(1-e_{2i})$
\item[(iii)] $t_{2i-1}t_{2i}t_{2i-1} = t_{2i}-e_{2i-2}t_{2i}-t_{2i}e_{2i-2}$
\item[(iv)] $t_{2i}t_{2i-1}t_{2i} = t_{2i-1}-e_{2i}t_{2i-1}-t_{2i-1}e_{2i}$
\end{itemize}
\end{itemize}
hold in $\mathcal{A}_{2k}$, thus completing the proof of \emph{\Cref{AffPrep}}.
\end{lem}

\begin{proof}
To ease notation, for any tuple $\bm{a}=(a_{1},\dots,a_{k})\in [n]^{k}$, we represent a simple tensor in $V^{\otimes k}$ by a word in the entries of $\bm{a}$, that is $a_{1}\cdots a_{k}:=v_{a_{1}}\otimes \cdots \otimes v_{a_{k}}$. We will prove these relations by showing that they hold under $\psi_{n,k}$ for all $n\geq 1$, and then employ \emph{\Cref{TrivKer}}. For each relation we will have to consider different cases based on the relative values of the entries $a_{i-1}$, $a_{i}$, and $a_{i+1}$, although most cases are trivial. Also note that $\psi_{n,k}(1-e_{2i})(\bm{a}) = \varepsilon_{a_{i},a_{i+1}}\bm{a}$.

\vspace{2mm}
\noindent
\emph{(7)(i)}: If $a_{i-1}=a_{i}$ or $a_{i}=a_{i+1}$, then it is easy to check that both $t_{2i-2}t_{2i}t_{2i-2}$ and $t_{2i}t_{2i-2}t_{2i}(1-e_{2i-2})$ will act on $\bm{a}$ by 0. Assume that $a_{i}\neq a_{i-1}=a_{i+1}$, then
\begin{align*}
\psi_{n,k}(t_{2i-2}t_{2i}&t_{2i-2})(\bm{a}) = \psi_{n,k}(t_{2i-2}t_{2i})\Big((a_{i-1},a_{i})(a_{1}\cdots a_{i-2})a_{i-1}\cdots a_{k}\Big) \\
&= \psi_{n,k}(t_{2i-2})\Big((a_{i},a_{i+1})(a_{i-1},a_{i})(a_{1}\cdots a_{i-2})a_{i}a_{i}\cdots a_{k}\Big) = 0.
\end{align*}
Similarly one can show that $\psi_{n,k}(t_{2i}t_{2i-2}t_{2i}(1-e_{2i-2}))(\bm{a})=0$ when $a_{i}\neq a_{i-1}=a_{i+1}$. Lastly assume that $a_{i-1},a_{i}$, and $a_{i+1}$ are pairwise distinct, in particular $\varepsilon_{a,b}=1$ for any $a,b\in\{a_{i-1},a_{i},a_{i+1}\}$. Then
\begin{align*}
\psi_{n,k}(t_{2i}t_{2i-2}t_{2i}(1-e_{2i-2}))(\bm{a}) &= \psi_{n,k}(t_{2i}t_{2i-2}t_{2i})(\bm{a}) \\
&=\psi_{n,k}(t_{2i}t_{2i-2})\Big( (a_{i},a_{i+1})(a_{1}\cdots a_{i-1})a_{i}\cdots a_{k}\Big) \\
&=\psi_{n,k}(t_{2i}t_{2i-2})\Big( (a_{i},a_{i+1})(a_{1}\cdots a_{i-2})a_{i-1}\cdots a_{k}\Big) \\
&=\psi_{n,k}(t_{2i})\Big( (a_{i-1},a_{i})(a_{i},a_{i+1})(a_{1}\cdots a_{i-2})a_{i-1}\cdots a_{k}\Big) \\
&=\Big( (a_{i},a_{i+1})(a_{i-1},a_{i})(a_{i},a_{i+1})(a_{1}\cdots a_{i-2})a_{i-1}\cdots a_{k}\Big) \\
&=\Big( (a_{i-1},a_{i})(a_{i},a_{i+1})(a_{i-1},a_{i})(a_{1}\cdots a_{i-2})a_{i-1}\cdots a_{k}\Big) \\
&=\psi_{n,k}(t_{2i-2})\Big((a_{i},a_{i+1})(a_{i-1},a_{i})(a_{1}\cdots a_{i-2})a_{i-1}\cdots a_{k}\Big) \\
&=\psi_{n,k}(t_{2i-2}t_{2i})\Big((a_{i-1},a_{i})(a_{1}\cdots a_{i-2})a_{i-1}\cdots a_{k}\Big) \\
&=\psi_{n,k}(t_{2i-2}t_{2i}t_{2i-2})(\bm{a})
\end{align*}

\vspace{2mm}
\noindent
\emph{(7)(ii)}: If $a_{i}=a_{i+1}$ then its clear that both $t_{2i+1}t_{2i-1}t_{2i+1}$ and $t_{2i-1}t_{2i+1}t_{2i-1}(1-e_{2i})$ act on $\bm{a}$ by 0. Assume that $a_{i}\neq a_{i+1}$ and $a_{i-1}\in\{a_{i},a_{i+1}\}$, then
\begin{align*}
\psi_{n,k}&(t_{2i+1}t_{2i-1}t_{2i+1})(\bm{a}) = \psi_{n,k}(t_{2i+1}t_{2i-1})\Big((a_{i},a_{i+1})(a_{1}\cdots a_{i-1})a_{i+1}a_{i}a_{i+2}\cdots a_{k}\Big) \\
&= \varepsilon_{b,a_{i+1}}\psi_{n,k}(t_{2i+1})\Big((b,a_{i+1})(a_{i},a_{i+1})(a_{1}\cdots a_{i-1})ba_{i}a_{i+2}\cdots a_{k}\Big) \\
&= \varepsilon_{b,a_{i}}\varepsilon_{b,a_{i+1}}\Big((b,a_{i})(b,a_{i+1})(a_{i},a_{i+1})(a_{1}\cdots a_{i-1})a_{i}ba_{i+2}\cdots a_{k}\Big)
\end{align*}
where $b=(a_{i},a_{i+1})(a_{i-1})$. Since $a_{i-1}\in\{a_{i},a_{i+1}\}$, we have that $\varepsilon_{b,a_{i}}\varepsilon_{b,a_{i+1}}=0$, and so $t_{2i+1}t_{2i-1}t_{2i+1}$ acts on $\bm{a}$ by 0. Similarly one can check that $t_{2i-1}t_{2i+1}t_{2i-1}(1-e_{2i})$ also acts on $\bm{a}$ by 0. Lastly assume that $a_{i-1},a_{i}$, and $a_{i+1}$ are pairwise distinct. Then
\begin{align*}
\psi_{n,k}(t_{2i-1}&t_{2i+1}t_{2i-1}(1-e_{2i}))(\bm{a}) = \psi_{n,k}(t_{2i-1}t_{2i+1}t_{2i-1})(\bm{a}) \\
&=\psi_{n,k}(t_{2i-1}t_{2i+1})\Big( (a_{i-1},a_{i})(a_{1}\cdots a_{i-2})a_{i}a_{i-1}a_{i+1}\cdots a_{k}\Big) \\
&=\psi_{n,k}(t_{2i-1})\Big( (a_{i-1},a_{i+1})(a_{i-1},a_{i})(a_{1}\cdots a_{i-2})a_{i}a_{i+1}a_{i-1}a_{i+2}\cdots a_{k}\Big) \\
&=(a_{i},a_{i+1})(a_{i-1},a_{i+1})(a_{i-1},a_{i})(a_{1}\cdots a_{i-2})a_{i+1}a_{i}a_{i-1}a_{i+2}\cdots a_{k} \\
&=(a_{i-1},a_{i})(a_{i-1},a_{i+1})(a_{i},a_{i+1})(a_{1}\cdots a_{i-2})a_{i+1}a_{i}a_{i-1}a_{i+2}\cdots a_{k} \\
&=\psi_{n,k}(t_{2i+1})\Big((a_{i-1},a_{i+1})(a_{i},a_{i+1})(a_{1}\cdots a_{i-2})a_{i+1}a_{i-1}a_{i}a_{i+2}\cdots a_{k}\Big) \\
&=\psi_{n,k}(t_{2i+1}t_{2i-1})\Big((a_{i},a_{i+1})(a_{1}\cdots a_{i-2})a_{i-1}a_{i+1}a_{i}a_{i+2}\cdots a_{k}\Big) \\
&=\psi_{n,k}(t_{2i+1}t_{2i-1}t_{2i+1})(\bm{a}) \\
\end{align*}

\vspace{2mm}
\noindent
\emph{(7)(iii)}: Assume $a_{i}=a_{i+1}$, then it is easy to check that $ t_{2i}-e_{2i-2}t_{2i}-t_{2i}e_{2i-2}$ acts on $\bm{a}$ by 0. Similarly
\begin{align*}
\psi_{n,k}&(t_{2i-1}t_{2i}t_{2i-1})(\bm{a}) = \varepsilon_{a_{i-1},a_{i}}\psi_{n,k}(t_{2i-1}t_{2i})\Big( (a_{i-1},a_{i})(a_{1}\cdots a_{i-2})a_{i}a_{i-1}a_{i+1}\cdots a_{k}\Big) \\
&= \varepsilon_{a_{i-1},a_{i+1}}\varepsilon_{a_{i-1},a_{i}}\psi_{n,k}(t_{2i-1})\Big( (a_{i-1},a_{i+1})(a_{i-1},a_{i})(a_{1}\cdots a_{i-2})a_{i-1}a_{i-1}a_{i+1}\cdots a_{k}\Big) \\
&= 0.
\end{align*}
Now assume $a_{i}\neq a_{i+1}$ and $a_{i-1}\in\{a_{i},a_{i+1}\}$. Then
\[ \psi_{n,k}(t_{2i}-e_{2i-2}t_{2i}-t_{2i}e_{2i-2})(\bm{a}) = (1-\delta_{(a_{i},a_{i+1})(a_{i-1}),a_{i}})(a_{i},a_{i+1})(a_{1}\cdots a_{i-1})a_{i}\cdots a_{k}. \]
In either case for $a_{i-1}=a_{i}$ or $a_{i-1}=a_{i+1}$, we have $\psi_{n,k}(t_{2i}-e_{2i-2}t_{2i}-t_{2i}e_{2i-2})(\bm{a})=0$. Also, from above we see that $\psi_{n,k}(t_{2i-1}t_{2i}t_{2i-1})(\bm{a})=0$ since the factor $\varepsilon_{a_{i-1},a_{i+1}}\varepsilon_{a_{i-1},a_{i}}$ comes into play. Lastly, assume that $a_{i-1}$, $a_{i}$, and $a_{i+1}$ are pairwise distinct, then it is easy to check that $\psi_{n,k}(e_{2i-2}t_{2i})(\bm{a})=\psi_{n,k}(t_{2i}e_{2i-2})(\bm{a})=0$. Also,
\begin{align*}
\psi_{n,k}(t_{2i-1}t_{2i}t_{2i-1})(\bm{a}) &= \psi_{n,k}(t_{2i-1}t_{2i})\Big( (a_{i-1},a_{i})(a_{1}\cdots a_{i-2})a_{i}a_{i-1}a_{i+1}\cdots a_{k}\Big) \\
&= \psi_{n,k}(t_{2i-1})\Big( (a_{i-1},a_{i+1})(a_{i-1},a_{i})(a_{1}\cdots a_{i-2})a_{i}a_{i-1}a_{i+1}\cdots a_{k}\Big) \\
&= (a_{i-1},a_{i})(a_{i-1},a_{i+1})(a_{i-1},a_{i})(a_{1}\cdots a_{i-2})a_{i-1}\cdots a_{k} \\
&= (a_{i},a_{i+1})(a_{1}\cdots a_{i-2})a_{i-1}\cdots a_{k} \\
&= \psi_{n,k}(t_{2i})(\bm{a}) = \psi_{n,k}(t_{2i}-e_{2i-2}t_{2i}-t_{2i}e_{2i-2})(\bm{a}).
\end{align*}

\vspace{2mm}
\noindent
\emph{(7)(iv)}: This relation can be proved by analogous computations to \emph{(7)(iii)} above.

\end{proof}


\section{Affine partition algebra}


\subsection{Definition of $\mathcal{A}_{2k}^{\text{aff}}(z)$ and basic results}


In this section we give the definition of the affine partition algebra $\mathcal{A}_{2k}^{\text{aff}}$ by generators and relations. We prove some basic properties about this algebra including the fact that the partition algebra $\mathcal{A}_{2k}$ is both a quotient and subalgebra of $\mathcal{A}_{2k}^{\text{aff}}$. We also show that the polynomial algebra $\mathbb{C}[x_{1},\dots,x_{2k}]$ is a subalgebra and that $\mathcal{H}_{k}\otimes \mathcal{H}_{k}$ is a quotient, where $\mathcal{H}_{k}$ is the degenerate affine Hecke algebra. We will prove a variety of relations in $\mathcal{A}_{2k}^{\text{aff}}$ including counterparts to the recursive definition of both the Jucys-Murphy elements and Enyang's generators.

\begin{defn} \label{APADefn}
We define the \emph{affine partition algebra} $\mathcal{A}_{2k}^{\text{aff}}$ to be the associative unitial $\mathbb{C}$-algebra with set of generators
\[ \{ \tau_{i}, e_{j}, x_{r}, z_{l} \ | \ 2\leq i \leq 2k-1, 1\leq j \leq 2k-1, r\in[2k], l\in\mathbb{Z}_{\geq 0}\} \]
and defining relations
\begin{itemize}
\item[(1)] (Involutions)
\begin{itemize}
\item[(i)] $\tau_{2i}^{2} = 1-e_{2i}$, for $i\in [k-1]$.
\item[(ii)] $\tau_{2i+1}^{2} = 1-e_{2i}$, for $i\in [k-1]$.
\end{itemize}
\item[(2)] (Braid relations)
\begin{itemize}
\item[(i)] $\tau_{2i+1}\tau_{2j} = \tau_{2j}\tau_{2i+1}$ for $j \neq i+1$.
\item[(ii)] $\tau_{2i+1}\tau_{2j+1} = \tau_{2j+1}\tau_{2i+1}$ for $j \neq i\pm1$.
\item[(iii)] $\tau_{2i}\tau_{2j} = \tau_{2j}\tau_{2i}$ for $j \neq i\pm1$.
\item[(iv)] $s_{i}s_{i+1}s_{i} = s_{i+1}s_{i}s_{i+1}$, for $i\in[k-2]$, where $s_{j} := \tau_{2j}\tau_{2j+1}+e_{2j}$.
\end{itemize}
\item[(3)] (Idempotent relations)
\begin{itemize}
\item[(i)] $e_{2i-1}^{2} = z_{0}e_{2i-1}$ for $i \in [k]$. 
\item[(ii)] $e_{2i}^{2} = e_{2i}$ for $i \in [k-1]$.
\item[(iii)] $\tau_{2i+1}e_{2i} = e_{2i}\tau_{2i+1} = 0$ for $i\in [k-1]$.
\item[(iv)] $\tau_{2i}e_{2i} = e_{2i}\tau_{2i} = 0$ for $i\in [k-1]$.
\item[(v)] $\tau_{2i}e_{2i-1}e_{2i+1} = \tau_{2i+1}e_{2i-1}e_{2i+1}$ for $i\in [k-1]$.
\item[(vi)] $e_{2i+1}e_{2i-1}\tau_{2i} = e_{2i+1}e_{2i-1}\tau_{2i+1}$ for $i\in [k-1]$.
\end{itemize}
\item[(4)] (Commutation relations)
\begin{itemize}
\item[(i)] $e_{i}e_{j} = e_{j}e_{i}$, if $|i-j| \geq 2$. 
\item[(ii)] $\tau_{2i-1}e_{2j-1} = e_{2j-1}\tau_{2i-1}$, if $j \neq i-1,i$.
\item[(iii)] $\tau_{2i-1}e_{2j} = e_{2j}\tau_{2i-1}$, if $j \neq i$.
\item[(iv)] $\tau_{2i}e_{2j-1} = e_{2j-1}\tau_{2i}$, if $j \neq i,i+1$.
\item[(v)] $\tau_{2i}e_{2j} = e_{2j}\tau_{2i}$, if $j \neq i-1$.
\end{itemize}
\item[(5)] (Contractions)
\begin{itemize}
\item[(i)] $e_{i}e_{i+1}e_{i} = e_{i}$ and $e_{i+1}e_{i}e_{i+1} = e_{i+1}$, for $i \in [2n-2]$.
\item[(ii)] $\tau_{2i}e_{2i-1}\tau_{2i} = \tau_{2i+1}e_{2i+1}\tau_{2i+1}$, for $i \in [k-1]$.
\item[(iii)] $\tau_{2i}e_{2i-2}\tau_{2i} = \tau_{2i-1}e_{2i}\tau_{2i-1}$, for $2\leq i \leq k-1$.
\end{itemize}
\item[(6)] (Affine Commuting Relations)
\begin{itemize}
\item[(i)] $x_{i}x_{j}=x_{j}x_{i}$ for all $i,j \in [2k]$
\item[(ii)] $\tau_{i}x_{j} = x_{j}\tau_{i}$ for $j \neq i-1, i, i+1$
\item[(iii)] $e_{i}x_{j}=x_{j}e_{i}$ for $j\neq i, i+1$
\end{itemize}
\item[(7)] (Braid-like relations)
\begin{itemize}
\item[(i)] $\tau_{2i-2}\tau_{2i}\tau_{2i-2} = \tau_{2i}\tau_{2i-2}\tau_{2i}(1-e_{2i-2})$.
\item[(ii)] $\tau_{2i+1}\tau_{2i-1}\tau_{2i+1} = \tau_{2i-1}\tau_{2i+1}\tau_{2i-1}(1-e_{2i})$.
\item[(iii)] $\tau_{2i-1}\tau_{2i}\tau_{2i-1} = \tau_{2i}-e_{2i-2}\tau_{2i}-\tau_{2i}e_{2i-2}$.
\item[(iv)] $\tau_{2i}\tau_{2i-1}\tau_{2i} = \tau_{2i-1}-e_{2i}\tau_{2i-1}-\tau_{2i-1}e_{2i}$.
\end{itemize}
\item[(8)] (Skein-like Relations)
\begin{itemize}
\item[(i)] $x_{2i+1}=\tau_{2i}x_{2i-1}\tau_{2i}+e_{2i}e_{2i-1}\tau_{2i}+\tau_{2i}e_{2i-1}e_{2i}-\tau_{2i}$.
\item[(ii)] $x_{2i+2}=\tau_{2i+1}x_{2i}\tau_{2i+1}+e_{2i}e_{2i+1}\tau_{2i+1}e_{2i+1}e_{2i}+\tau_{2i+1}$.
\item[(iii)] $x_{2i}=\tau_{2i}x_{2i}\tau_{2i}+e_{2i}e_{2i-1}\tau_{2i}+\tau_{2i}e_{2i-1}e_{2i}$.
\item[(iv)] $x_{2i+1}=\tau_{2i+1}x_{2i+1}\tau_{2i+1}+e_{2i}e_{2i+1}\tau_{2i+1}+\tau_{2i+1}e_{2i+1}e_{2i}$.
\end{itemize}
\item[(9)] (Anti-symmetry Relations)
\begin{itemize}
\item[(i)] $e_{i}(x_{i}-x_{i+1}) = 0$ for $i\in [2k-1]$.
\item[(ii)] $(x_{i}-x_{i+1})e_{i} = 0$ for $i\in [2k-1]$.
\end{itemize}
\item[(10)] (Bubble Relations)
\begin{itemize}
\item[(i)] $e_{1}x_{1}^{l}e_{1} = z_{l}e_{1}$, for all $l \in \mathbb{N}$.
\item[(ii)] $z_{l}$ is central for all $l \in \mathbb{Z}_{\geq 0}$.
\end{itemize}
\end{itemize}

\end{defn}

Note we have overloaded the symbols $e_{i}$ and $s_{j}$ as elements in $\mathcal{A}_{2k}$ and $\mathcal{A}_{2k}^{\text{aff}}$, however we will show shortly that the mapping $\mathcal{A}_{2k}\rightarrow \mathcal{A}_{2k}^{\text{aff}}$ via $z\mapsto z_{0}$, $e_{i}\mapsto e_{i}$, and $s_{j}\mapsto s_{j}$ realises the subalgebra $\langle e_{i}, s_{j}, z_{0} \rangle$ of $\mathcal{A}_{2k}^{\text{aff}}$ as an isomorphic copy of the partition algebra $\mathcal{A}_{2k}$. The defining relations above are those present in \emph{\Cref{AffPrep}}, except where the Jucys-Murphy elements $X_{i}$ have been replaced with the affine generators $x_{i}$, Enyang's generators $t_{j}$ have been replaced by new generators $\tau_{j}$, and the polynomials $z(z-1)^{l}$ have been replaced by central generators $z_{l}$. It is worth mentioning that the map $\mathcal{A}_{2k}\rightarrow \mathcal{A}_{2k}^{\text{aff}}$ given by $z\mapsto z_{0}$, $e_{i}\mapsto e_{i}$, and $t_{j}\mapsto \tau_{j}$ does not realise an algebra homomorphism. This is since $\tau_{2}$ is a non-trivial generator in $\mathcal{A}_{2k}^{\text{aff}}$, while $t_{2}$ is abscent in the presentation of \emph{\Cref{EnyPres}} since it equals $1-e_{2}$, hence the braid relation \emph{(E2)(iv)} is not respected under such a map. The subalgebra $\langle e_{i}, \tau_{j}, z_{0} \rangle$ of $\mathcal{A}_{2k}^{\text{aff}}$ is not isomorphic to the partition algebra, and in fact one can show that this subalgebra is infinite dimensional as an $\mathbb{C}[z_{0}]$-module (see \Cref{infinitesub} below).

Replacing the Jucys-Murphy elements with commuting variables, and introducing new central generators is very much analogous to the \lq affinization' process employed on other diagram algebras. In particular relations $(6)$ to $(10)$ (except $(7)$) are comparable to the relations in \cite[Section 4]{Naz96} which were chosen as the defining relations for the affine Wenzl algebra. The Skein-like relations $(8)$ tell us how the affine generators $x_{i}$ interact with the generators $\tau_{j}$ when they do not commute. These relations are to $\mathcal{A}_{2k}^{\text{aff}}$ what the defining relation $y_{i+1}=s_{i}y_{i}s_{i}+s_{i}$ is to the degenerate affine  Hecke algebra $\mathcal{H}_{k}$. In the next section we provide a projection of $\mathcal{A}_{2k}^{\text{aff}}$ onto a diagram algebra living within the Heisenberg category. Under this projection the Skein-like relations will correspond to moving a decoration over crossings.

We have also chosen to replace the generators $t_{j}$ with new generators $\tau_{j}$, which appears to be a departure from the \lq affinization' process. However, we will show that these elements are not needed to generate the algebra, that is $\mathcal{A}_{2k}^{\text{aff}}=\langle e_{i},s_{i},x_{i},z_{l}\rangle$. Hence to go from $\mathcal{A}_{2k}$ to $\mathcal{A}_{2k}^{\text{aff}}$ we have indeed just adjoined new affine and central genrators. The reason for letting the elements $\tau_{j}$ play the role of generators is to allow us to give a cleaner presentation which is more comparable to its counterparts within the literature. We have chosen to include the Braid-like relations $(7)$ as they tell us how none commuting $\tau_{j}$ generators interact in a manner which resembles the braid relations of the Coxeter generators $s_{i}$. These relations will allow us to give counterparts to the recursive definition of Enyang's generators (see \Cref{enyangrecursion} below). 

We begin by showing that the partition algebra is a quotient of the affine partition algebra. This follows naturally from its construction.

\begin{lem}
We have a surjective $\mathbb{C}$-algebra homomorphism $\mathsf{pr}:\mathcal{A}_{2k}^{\emph{aff}}\rightarrow \mathcal{A}_{2k}$, given on the generators by
\[ \mathsf{pr}(\tau_{i})=t_{i}, \hspace{7mm} \mathsf{pr}(e_{i})=e_{i}, \hspace{7mm} \mathsf{pr}(x_{i})=X_{i}, \hspace{7mm} \mathsf{pr}(z_{l}) = z(z-1)^{l}. \]
\end{lem}

\begin{proof}
This follows by comparing the defining relations with those of the same numbering in \emph{\Cref{AffPrep}}, and surjectivity follows since $\langle t_{i}, e_{j}, z \rangle = \mathcal{A}_{2k}$.

\end{proof}

Similar to the partition algebra, the affine partition algebra has a corresponding anti-automorphism which fixes the generators.

\begin{lem}
The mapping $*:\mathcal{A}_{2k}^{\emph{aff}}\rightarrow\mathcal{A}_{2k}^{\emph{aff}}$ which fixes the generators, extended $\mathbb{C}$-linearly, gives an anti-automorphism.
\end{lem} 

\begin{proof}
All defining relations of \emph{\Cref{APADefn}} are symmetric in the generators except relations (7)(i) and (7)(ii). Thus it is clear that the result holds if we can show that $e_{2i-2}$ and $\tau_{2i}\tau_{2i-2}\tau_{2i}$ commute, and that $e_{2i}$ and $\tau_{2i-1}\tau_{2i+1}\tau_{2i-1}$ commute. For the former we have
\begin{align*}
\tau_{2i}\tau_{2i-2}\tau_{2i}e_{2i-2} &= \tau_{2i}\tau_{2i-2}\tau_{2i-1}e_{2i}\tau_{2i-1}\tau_{2i} \\
&= \tau_{2i}\tau_{2i-1}e_{2i}\tau_{2i-1}\tau_{2i-2}\tau_{2i} \\
&= e_{2i-2}\tau_{2i}\tau_{2i-2}\tau_{2i}
\end{align*}
where the first equaltiy can be deduced from relation (5)(iii) of \emph{\Cref{APADefn}}, the second relations follows since $\tau_{2i-2}$ commutes with $\tau_{2i-1}$ and $e_{2i}$, then the last equality again is deducable from relation (5)(iii) of \emph{\Cref{APADefn}}. Showing that $e_{2i}$ and $\tau_{2i-1}\tau_{2i+1}\tau_{2i-1}$ commute follows in a similar manner.

\end{proof}

We now seek to show that $\mathcal{A}_{2k}$ is the subalgebra $\langle s_{i}, e_{j}, z_{0} \rangle$ of $\mathcal{A}_{2k}^{\text{aff}}$. We first prove a few helpful relations.

\begin{lem} \label{AffMixRels}
The following relations hold:
\begin{itemize}
\item[(i)] $e_{2i}x_{2i} = e_{2i}e_{2i-1}\tau_{2i}$, and \ $x_{2i}e_{2i}=\tau_{2i}e_{2i-1}e_{2i}$
\item[(ii)] $e_{2i}x_{2i+1} = e_{2i}e_{2i+1}\tau_{2i+1}$, and \ $x_{2i+1}e_{2i}=\tau_{2i+1}e_{2i+1}e_{2i}$
\item[(ii)] $e_{2i}e_{2i-1}\tau_{2i}=e_{2i}e_{2i+1}\tau_{2i+1}$, and \ $\tau_{2i}e_{2i-1}e_{2i}=\tau_{2i+1}e_{2i+1}e_{2i}$
\end{itemize}
\end{lem}

\begin{proof}
\emph{(i)}: Multiplying \emph{(8)(iii)} of \emph{\Cref{APADefn}} on the left by $e_{2i}$ gives
\[
e_{2i}x_{2i} = e_{2i}\tau_{2i}x_{2i}\tau_{2i}+e_{2i}e_{2i}e_{2i-1}\tau_{2i}+e_{2i}\tau_{2i}e_{2i-1}e_{2i}= e_{2i}e_{2i-1}\tau_{2i}
\]
since $e_{2i}\tau_{2i}=0$ and $e_{2i}e_{2i}=e_{2i}$. The relation $x_{2i}e_{2i}=\tau_{2i}e_{2i-1}e_{2i}$ follows by $*$.

\vspace{2mm}
\noindent
\emph{(ii)}: Multiplying \emph{(8)(iv)} of \emph{\Cref{APADefn}} on the left by $e_{2i}$ gives
\[
e_{2i}x_{2i+1} = e_{2i}\tau_{2i+1}x_{2i+1}\tau_{2i+1}+e_{2i}e_{2i}e_{2i+1}\tau_{2i+1}+e_{2i}\tau_{2i+1}e_{2i+1}e_{2i}= e_{2i}e_{2i+1}\tau_{2i+1}
\]
since $e_{2i}\tau_{2i+1}=0$ and $e_{2i}e_{2i}=e_{2i}$. The relation $x_{2i+1}e_{2i}=\tau_{2i+1}e_{2i+1}e_{2i}$ follows by $*$.

\vspace{2mm}
\noindent
\emph{(iii)}: By \emph{(9)(i), (ii)} of \emph{\Cref{APADefn}}, $e_{2i}x_{2i}=e_{2i}x_{2i+1}$ and $x_{2i}e_{2i}=x_{2i+1}e_{2i}$. So \emph{(i)} and \emph{(ii)} imply \emph{(iii)}.

\end{proof}

\begin{prop}
We have a injective $\mathbb{C}$-algebra homomorphism $\iota:\mathcal{A}_{2k} \rightarrow \mathcal{A}_{2k}^{\emph{aff}}$ given on the generators by $\iota(z)=z_{0}$, $\iota(s_{i}) = \tau_{2i}\tau_{2i+1}+e_{2i}$, and $\iota(e_{i})=e_{i}$.
\end{prop}

\begin{proof}
We first prove that $\iota$ is a homomorphism. To do this we show that each of the defining relations of $\mathcal{A}_{2k}$ given in \emph{\Cref{HRPres}} is respected under $\iota$. We only check the relations involving $s_{i}$ since the others are accounted for in the definition of $\mathcal{A}_{2k}^{\text{aff}}$.

\vspace{2mm}
\noindent
\emph{(HR1)(i)}:
\[
\iota(s_{i}^{2}) = (\tau_{2i}\tau_{2i+1}+e_{2i})(\tau_{2i}\tau_{2i+1}+e_{2i})
= \tau_{2i}^{2}\tau_{2i+1}^{2}+e_{2i}
= (1-e_{2i})(1-e_{2i})+e_{2i}
= 1-2e_{2i}+2e_{2i} =1
\]
where we used \emph{(1), (2)(i), (3)(ii), (3)(iii)}, and \emph{(3)(iv)}.

\vspace{2mm}
\noindent
\emph{(HR1)(ii)}: This holds by relations \emph{(2)(i), (2)(ii), (2)(iii)} and \emph{(4)}.

\vspace{2mm}
\noindent
\emph{(HR1)(iii)}: This is precisely \emph{(2)(iv)}.

\vspace{2mm}
\noindent
\emph{(HR2)(iii)}: 
\[ \iota(e_{2i}s_{i}) = e_{2i}(\tau_{2i}\tau_{2i+1}+e_{2i}) = e_{2i} = \iota(e_{2i}) \]
where we used \emph{(3)(iii)} and \emph{(3)(ii)}. Similarly we have $\iota(s_{i}e_{2i})=\iota(e_{2i})$.

\vspace{2mm}
\noindent
\emph{(HR2)(iv)}:
\begin{align*}
\iota(s_{i}e_{2i-1}e_{2i+1}) &= (\tau_{2i}\tau_{2i+1}+e_{2i})e_{2i-1}e_{2i+1} \\
&= \tau_{2i}\tau_{2i+1}e_{2i-1}e_{2i+1} + e_{2i}e_{2i-1}e_{2i+1} \\
&= \tau_{2i}^{2}e_{2i-1}e_{2i+1} + e_{2i}e_{2i-1}e_{2i+1} \\
&= e_{2i-1}e_{2i+1}-e_{2i}e_{2i-1}e_{2i+1}+e_{2i}e_{2i-1}e_{2i+1} \\
&= e_{2i-1}e_{2i+1} = \iota(e_{2i-1}e_{2i+1})
\end{align*}
where the third equality follows from \emph{(3)(v)} and the forth from \emph{(1)(i)}. Similarly we have $\iota(e_{2i-1}e_{2i+1}s_{i})=\iota(e_{2i-1}e_{2i+1})$.

\vspace{2mm}
\noindent
\emph{(HR3)(iv)}: Follows from commuting relations \emph{(4)(i), (4)(ii)}, and \emph{(4)(iv)}. 

\vspace{2mm}
\noindent
\emph{(HR3)(v)}: Follows from commuting relations \emph{(4)(i), (4)(iii)}, and \emph{(4)(v)}. 

\vspace{2mm}
\noindent
\emph{(HR3)(vi)}: 
\begin{align*}
\iota(s_{i}e_{2i-1}s_{i}) &= (\tau_{2i}\tau_{2i+1}+e_{2i})e_{2i-1}(\tau_{2i}\tau_{2i+1}+e_{2i}) \\
&= \tau_{2i+1}\tau_{2i}e_{2i-1}\tau_{2i}\tau_{2i+1}+\tau_{2i+1}\tau_{2i}e_{2i-1}e_{2i}+e_{2i}e_{2i-1}\tau_{2i}\tau_{2+1}+e_{2i} \\
&= \tau_{2i+1}^{2}e_{2i+1}\tau_{2i+1}^{2}+\tau_{2i+1}^{2}e_{2i+1}e_{2i}+e_{2i}e_{2i+1}\tau_{2i+1}^{2}+e_{2i} \\
&= (1-e_{2i})e_{2i+1}(1-e_{2i})+e_{2i+1}e_{2i}-e_{2i}+e_{2i}e_{2i+1}-e_{2i}+e_{2i} \\
&= e_{2i+1}-e_{2i}e_{2i+1}-e_{2i+1}e_{2i}+e_{2i}+e_{2i+1}e_{2i}-e_{2i}+e_{2i}e_{2i+1} \\
&= e_{2i+1} = \iota(e_{2i+1})
\end{align*}
where the third equality follows by \emph{\Cref{AffMixRels} (iii)} and \emph{(5)(ii)}, and the forth from $\tau_{2i+1}^{2}=1-e_{2i}$. 

\vspace{2mm}
\noindent
\emph{(HR3)(vii)}:
\begin{align*}
\iota(s_{i}e_{2i-2}s_{i}) &= (\tau_{2i+1}\tau_{2i}+e_{2i})e_{2i-2}(\tau_{2i}\tau_{2i+1}+e_{2i}) \\
&= \tau_{2i}\tau_{2i+1}e_{2i-2}\tau_{2i+1}\tau_{2i}+\tau_{2i+1}\tau_{2i}e_{2i-2}e_{2i}+e_{2i}e_{2i-2}\tau_{2i}\tau_{2i+1}+e_{2i}e_{2i-2}e_{2i} \\
&= \tau_{2i}\tau_{2i+1}^{2}e_{2i-2}\tau_{2i}+e_{2i}e_{2i-2} \\
&= \tau_{2i}e_{2i-2}\tau_{2i}+e_{2i}e_{2i-2} \\
&= \tau_{2i-1}e_{2i}\tau_{2i-1}+e_{2i}e_{2i-2}
\end{align*}
where the third equality follows since $\tau_{2i+1}$ and $e_{2i}$ commute with $e_{2i-2}$, $e_{2i}^{2}=e_{2i}$, and $e_{2i}\tau_{2i}=\tau_{2i}e_{2i}=0$. We also have
\begin{align*}
\iota(s_{i-1}e_{2i}s_{i-1}) &= (\tau_{2i-2}\tau_{2i-1}+e_{2i-2})e_{2i}(\tau_{2i-2}\tau_{2i-1}+e_{2i-2}) \\
&= \tau_{2i-1}\tau_{2i-2}e_{2i}\tau_{2i-2}\tau_{2i-1}+\tau_{2i-1}\tau_{2i-2}e_{2i}e_{2i-2}+e_{2i-2}e_{2i}\tau_{2i-2}\tau_{2i-1}+e_{2i-2}e_{2i}e_{2i-2} \\
&= \tau_{2i-1}\tau_{2i-2}^{2}e_{2i}\tau_{2i-1}+e_{2i}e_{2i-2} \\
&= \tau_{2i-1}e_{2i}\tau_{2i-1}-\tau_{2i-1}e_{2i-2}e_{2i}\tau_{2i-1}+e_{2i}e_{2i-2} \\
&= \tau_{2i-1}e_{2i}\tau_{2i-1}+e_{2i}e_{2i-2}
\end{align*}
where the third equality follows since $\tau_{2i-2}$ and $e_{2i-2}$ commute with $e_{2i}$, $e_{2i-2}^{2}=e_{2i-2}$, and $e_{2i-2}\tau_{2i-2}=\tau_{2i-2}e_{2i-2}=0$. The forth equality follows since $\tau_{2i-1}e_{2i-2}=0$. Comparing to above, we see that $\iota(s_{i}e_{2i-2}s_{i})=\iota(s_{i-1}e_{2i}s_{i-1})$.

\vspace{2mm}
Hence we have shown that $\iota$ is indeed an algebra homomorphism. For injectivity, note that $\mathsf{pr}\circ\iota = \mathsf{id}$ where $\mathsf{id}:\mathcal{A}_{2k}\rightarrow\mathcal{A}_{2k}$ is the identity morphism. Thus $\iota$ has a left inverse, and so is injective.

\end{proof}

Therefore the partition algebra $\mathcal{A}_{2k}$ is both a subalgebra and quotient of the affine partition algebra $\mathcal{A}_{2k}^{\text{aff}}$. Also note that restricting $*$ down to the partition algebra coincides with the anti-automorphism of flipping a diagram. We now seek to give affine counterparts to the recursive definition of the Jucys-Murphy elements given in \emph{\Cref{EnyJMDefn}}. 

\begin{lem} \label{AffRecRels}
The following relations hold in $\mathcal{A}_{2k}^{\emph{aff}}$:
\begin{itemize}
\item[(i)] $x_{2i+1} = s_{i}x_{2i-1}s_{i} + x_{2i}e_{2i} + e_{2i}x_{2i} - x_{2i-1}e_{2i} - \tau_{2i}$
\item[(ii)] $x_{2i+2} = s_{i}x_{2i}s_{i} - s_{i}x_{2i}e_{2i} - e_{2i}x_{2i}s_{i} + e_{2i}x_{2i}e_{2i+1}e_{2i} + \tau_{2i+1}$
\end{itemize}
\end{lem}

\begin{proof}
\emph{(i)}: Multiplying on the left and right of equation \emph{(8)(i)} in \emph{\Cref{APADefn}} by $\tau_{2i+1}$ gives
\begin{align*}
\tau_{2i+1}x_{2i+1}\tau_{2i+1} &= \tau_{2i+1}\tau_{2i}x_{2i-1}\tau_{2i}\tau_{2i+1}-\tau_{2i+1}\tau_{2i}\tau_{2i+1} \\
&= (s_{i}-e_{2i})x_{2i-1}(s_{i}-e_{2i})-(s_{i}-e_{2i})\tau_{2i+1} \\
&= s_{i}x_{2i-1}s_{i}-e_{2i}x_{2i-1}s_{i}-s_{i}x_{2i-1}e_{2i}+x_{2i+1}-\tau_{2i} \\
&= s_{i}x_{2i-1}s_{i}-x_{2i-1}e_{2i}-\tau_{2i} 
\end{align*}
where, in the first equality we used the fact that $\tau_{2i+1}e_{2i}=e_{2i}\tau_{2i+1}=0$, the second equaltiy we used the substitution $\tau_{2i}\tau_{2i+1}=\tau_{2i+1}\tau_{2i}=s_{i}-e_{2i}$, and the last equality we used the fact that $e_{2i}$ and $x_{2i-1}$ commute. Now applying \emph{(8)(iv)} from \emph{\Cref{APADefn}} to the left hand side of above, we obtain
\[ x_{2i+1}-e_{2i}e_{2i+1}\tau_{2i+1}-\tau_{2i+1}e_{2i+1}e_{2i} = s_{i}x_{2i-1}s_{i}-x_{2i-1}e_{2i}-\tau_{2i}. \]
By applying \emph{\Cref{AffMixRels} (ii)}, and rearranging, we arrive at \emph{(i)}. Item \emph{(ii)} is proved in an analogous manner were we instead employ relations \emph{(8)(ii)} and \emph{(8)(iii)} from \emph{\Cref{APADefn}}.

\end{proof}

By rearranging the relations in the above Lemma in terms of the generators $\tau_{2i}$ and $\tau_{2i+1}$, we immediately obtain the following:

\begin{cor}
We have that $\mathcal{A}_{2k}^{\emph{aff}}=\langle e_{i},s_{j},x_{k},z_{l}\rangle_{i,j,k,l}$.
\begin{flushright} $\square$ \end{flushright}
\end{cor}

Recall that the degenerate affine  Hecke algebra $\mathcal{H}_{k}$ is the $\mathbb{C}$-algebra given as a vector space by the tensor prooduct $\mathbb{C}[y_{1},\dots,y_{k}]\otimes\mathbb{C}S(k)$, where $\mathbb{C}[y_{1},\dots,y_{k}]$ is the polynomial algebra in commuting variables $y_{1},\dots,y_{k}$. The defining relations of $\mathcal{H}_{k}$ are such that $\mathbb{C}[y_{1},\dots,y_{k}]$ and $\mathbb{C}S(k)$ are subalgebras, and
\begin{align*}
s_{i}y_{j} &= y_{j}s_{i}, \hspace{12.7mm} \text{ for all } j\neq i,i+1, \\
y_{i+1} &= s_{i}y_{i}s_{i}+s_{i}, \hspace{3mm} \text{ for each } i\in[k-1].
\end{align*} 
It turns out that $\mathcal{H}_{k}\otimes\mathcal{H}_{k}$ is a quotient of $\mathcal{A}_{2k}^{\text{aff}}$. 

\begin{prop}
Let $\bm{\lambda}=(\lambda_{l})_{l=0}^{\infty}$ be any sequence of constants in $\mathbb{C}$. Then we have a surjective $\mathbb{C}$-algebra homomorphism $f_{\bm{\lambda}}:\mathcal{A}_{2k}^{\emph{aff}}\rightarrow \mathcal{H}_{k}\otimes \mathcal{H}_{k}$ given on the generators by
\begin{center}
\begin{itemize}
\begin{minipage}{0.52\linewidth}
    \item[ ] \hspace{17mm} $f_{\bm{\lambda}}(\tau_{2i+1}) = s_{i}\otimes 1$,
    \item[ ] \hspace{17mm} $f_{\bm{\lambda}}(\tau_{2i}) = 1 \otimes s_{i}$,
    \item[ ] \hspace{17mm} $f_{\bm{\lambda}}(e_{i})=0$,
\end{minipage}
\begin{minipage}{0.4\linewidth}
    \item[ ] $f_{\bm{\lambda}}(x_{2i-1}) = -1\otimes y_{i}$,
    \item[ ] $f_{\bm{\lambda}}(x_{2i}) = y_{i}\otimes 1$,
    \item[ ] $f_{\bm{\lambda}}(z_{l}) = \lambda_{l}$.
\end{minipage}
\end{itemize}
\end{center}
\end{prop}

\begin{proof}
We show that each of the defining relations of $\mathcal{A}_{2k}^{\text{aff}}$ are upheld under $f_{\bm{\lambda}}$. Since $f_{\bm{\lambda}}(e_{i})=0$, one may observe that most of the defining relations involving generators $e_{i}$ are trivially upheld.

\vspace{2mm}
\noindent
\emph{(1)(i)}: $f_{\bm{\lambda}}(\tau_{2i}^{2}) = (1\otimes s_{i})(1\otimes s_{i}) = 1\otimes s_{i}^{2} = 1 = f_{\bm{\lambda}}(1-e_{2i})$.

\vspace{2mm}
\noindent
\emph{(1)(ii)}: Similar to \emph{(1)(i)} above.

\vspace{2mm}
\noindent
\emph{(2)(i)}: For any $j\neq i+1$, $f_{\bm{\lambda}}(\tau_{2i+1}\tau_{2j})=(s_{i}\otimes 1)(1\otimes s_{j}) = (1\otimes s_{j})(s_{i}\otimes 1) = f_{\bm{\lambda}}(\tau_{2j}\tau_{2i+1})$.

\vspace{2mm}
\noindent
\emph{(2)(ii)}: For any $j\neq i\pm1$,
\[
f(\tau_{2i+1}\tau_{2j+1}) = (s_{i}\otimes 1)(s_{j}\otimes 1)
= s_{i}s_{j}\otimes 1
= s_{j}s_{i}\otimes 1
= (s_{j}\otimes 1)(s_{i}\otimes 1) = f(\tau_{2j+1}\tau_{2i+1}).
\]

\vspace{2mm}
\noindent
\emph{(2)(iii)}: Similar to \emph{(2)(ii)} above.

\vspace{2mm}
\noindent
\emph{(2)(iv)}: Noting that $f_{\bm{\lambda}}(s_{i})=f_{\bm{\lambda}}(\tau_{2i}\tau_{2i+1}+e_{2i})=f_{\bm{\lambda}}(\tau_{2i})f_{\bm{\lambda}}(\tau_{2i+1})=s_{i}\otimes s_{i}$, then
\[
f_{\bm{\lambda}}(s_{i}s_{i+1}s_{i}) = s_{i}s_{i+1}s_{i}\otimes s_{i}s_{i+1}s_{i}
= s_{i+1}s_{i}s_{i+1}\otimes s_{i+1}s_{i}s_{i+1}
= f_{\bm{\lambda}}(s_{i+1}s_{i}s_{i+1}).
\]

\vspace{2mm}
\noindent
\emph{(6)(i)}: Follows since $y_{1},\dots,y_{k}$ pariwise commute.

\vspace{2mm}
\noindent
\emph{(6)(ii)}: Follows since $s_{i}y_{j}=y_{j}s_{i}$ whenever $j\neq i,i+1$.

\vspace{2mm}
\noindent
\emph{(7)(i)}: 
\[
f_{\bm{\lambda}}(\tau_{2i-2}\tau_{2i}\tau_{2i-2}) = 1\otimes s_{i-1}s_{i}s_{i-1}
= 1\otimes s_{i}s_{i-1}s_{i}
= f_{\bm{\lambda}}(\tau_{2i}\tau_{2i-2}\tau_{2i}) = f_{\bm{\lambda}}(\tau_{2i}\tau_{2i-2}\tau_{2i}(1-e_{2i-2}))
\]

\vspace{2mm}
\noindent
\emph{(7)(ii)}: Similar to \emph{(7)(i)}.

\vspace{2mm}
\noindent
\emph{(7)(iii)}: $f_{\bm{\lambda}}(\tau_{2i-1}\tau_{2i}\tau_{2i-1}) = s_{i-1}^{2}\otimes s_{i} = 1\otimes s_{i}= f_{\bm{\lambda}}(\tau_{2i}) = f_{\bm{\lambda}}(\tau_{2i}-e_{2i-2}\tau_{2i}-\tau_{2i}e_{2i-2})$.

\vspace{2mm}
\noindent
\emph{(7)(iv)}: Similar to \emph{(7)(iii)}.

\vspace{2mm}
\noindent
\emph{(8)(i)}:
\begin{align*}
f_{\bm{\lambda}}(\tau_{2i}x_{2i-1}\tau_{2i}+e_{2i}e_{2i-1}\tau_{2i}+\tau_{2i}e_{2i-1}e_{2i}-\tau_{2i}) &= f_{\bm{\lambda}}(\tau_{2i}x_{2i-1}\tau_{2i})-f_{\bm{\lambda}}(\tau_{2i}) \\
&= (1\otimes s_{i})(-1\otimes y_{i})(1\otimes s_{i})-1\otimes s_{i} \\
&= -1\otimes s_{i}y_{i}s_{i} - 1\otimes s_{i} \\
&= -1\otimes (y_{i+1}-s_{i})- 1\otimes s_{i} \\
&= -1\otimes y_{i+1} \\
&= f_{\bm{\lambda}}(x_{2i+1})
\end{align*}
where the forth equality follows since $s_{i}y_{i}s_{i}=y_{i+1}-s_{i}$ in $\mathcal{H}_{k}$.

\vspace{2mm}
\noindent
\emph{(8)(ii)}:
\begin{align*}
f_{\bm{\lambda}}(\tau_{2i+1}x_{2i}\tau_{2i+1}+e_{2i}e_{2i+1}\tau_{2i+1}e_{2i+1}e_{2i}+\tau_{2i+1}) &= f_{\bm{\lambda}}(\tau_{2i+1}x_{2i}\tau_{2i+1})+f_{\bm{\lambda}}(\tau_{2i+1}) \\
&= (s_{i}\otimes 1)(y_{i}\otimes 1)(s_{i}\otimes 1)+s_{i}\otimes 1 \\
&= (s_{i}y_{i}s_{i}+s_{i})\otimes 1 \\
&= y_{i+1}\otimes 1 \\
&= f_{\bm{\lambda}}(x_{2i+2})
\end{align*}
where the forth equality follows since $y_{i+1}=s_{i}y_{i}s_{i}+s_{i}$ in $\mathcal{H}_{k}$.

\vspace{2mm}
\noindent
\emph{(8)(iii)}:
\begin{align*}
f_{\bm{\lambda}}(\tau_{2i}x_{2i}\tau_{2i}+e_{2i}e_{2i-1}\tau_{2i-1} + \tau_{2i-1}e_{2i-1}e_{2i}) &= f_{\bm{\lambda}}(\tau_{2i}x_{2i}\tau_{2i}), \\
&= (1\otimes s_{i})(y_{i}\otimes 1)(1\otimes s_{i}), \\
&= y_{i}\otimes 1, \\
&= f_{\bm{\lambda}}(x_{2i}).
\end{align*}

\vspace{2mm}
\noindent
\emph{(8)(iv)}:
\begin{align*}
f_{\bm{\lambda}}(\tau_{2i+1}x_{2i+1}\tau_{2i+1}+e_{2i}e_{2i+1}\tau_{2i+1} + \tau_{2i+1}e_{2i+1}e_{2i}) &= f_{\bm{\lambda}}(\tau_{2i+1}x_{2i+1}\tau_{2i+1}), \\
&= (s_{i}\otimes 1)(-1\otimes y_{i})(s_{i}\otimes 1), \\
&= -1\otimes y_{i}, \\
&= f_{\bm{\lambda}}(x_{2i+1}).
\end{align*}

\vspace{2mm}
\noindent
\emph{(10)(i)} and \emph{(10)(ii)}: Immediate.

\vspace{2mm}
\noindent
Thus $f_{\bm{\lambda}}$ is a homomorphism. Surjectivity follows as $\langle f_{\bm{\lambda}}(\tau_{i}), f_{\bm{\lambda}}(x_{j})\rangle_{i,j}=\mathcal{H}_{k}\otimes\mathcal{H}_{k}$.

\end{proof}

\begin{cor}
The polynomial algebra $\mathbb{C}[x_{1},\dots,x_{2k}]$ is a subalgebra of $\mathcal{A}_{2k}^{\emph{aff}}$.
\end{cor}

\begin{proof}
This is the same as asking that all monomials in the generators of the subalgebra $\langle x_{1},\dots,x_{2k} \rangle$ of $\mathcal{A}_{2k}^{\text{aff}}$ are linearly independent, which follows since their images under $f_{\bm{\lambda}}$ are.

\end{proof}

To end this section we establish a counterpart to the recursive relations of Enyang's generators. To do so, we collect the more technical relations needed into the following lemma:

\begin{lem}\label{TauRecPrepRels}
The following relations hold in $\mathcal{A}_{2k}^{\emph{aff}}$:
\begin{itemize}
\item[(i)] $e_{2i}x_{2i}e_{2i}=0$
\item[(ii)] $e_{2i}\tau_{2i-1}e_{2i}=0$
\item[(iii)] $e_{2i-2}\tau_{2i}e_{2i-2}=0$
\item[(iv)] $e_{2i-2}\tau_{2i}=e_{2i-2}x_{2i-2}s_{i}e_{2i-2}s_{i}$
\item[(v)] $\tau_{2i}e_{2i-2}=s_{i}e_{2i-2}s_{i}x_{2i-2}e_{2i-2}$
\item[(vi)] $\tau_{2i}\tau_{2i-2}\tau_{2i}e_{2i-2}=e_{2i-2}x_{2i-2}s_{i-1}e_{2i}e_{2i-1}e_{2i-2}$
\item[(vii)] $\tau_{2i-1}e_{2i}s_{i-1}=s_{i}e_{2i-2}e_{2i-1}e_{2i}s_{i-1}x_{2i-2}e_{2i-2}s_{i}$
\item[(viii)] $\tau_{2i}\tau_{2i-2}\tau_{2i}e_{2i-2}=e_{2i-2}\tau_{2i}\tau_{2i-2}\tau_{2i}$
\end{itemize}
\end{lem}

\begin{proof}
\emph{(i)}: We have $e_{2i}x_{2i}e_{2i} = e_{2i}e_{2i-1}\tau_{2i}e_{2i}=0$, by employing \emph{\Cref{AffMixRels} (i)} and \emph{\Cref{APADefn} (3)(iv)}.

\vspace{2mm}
\noindent
\emph{(ii)}: By rearranging \emph{(7)(iv)} of \emph{\Cref{APADefn}} in terms of $\tau_{2i-1}$, we have that
\[
e_{2i}\tau_{2i-1}e_{2i} = e_{2i}(\tau_{2i}\tau_{2i-1}\tau_{2i}+e_{2i}\tau_{2i-1}+\tau_{2i-1}e_{2i})e_{2i}
= e_{2i}\tau_{2i-1}e_{2i} +e_{2i}\tau_{2i-1}e_{2i},
\]
where we used relation $e_{2i}\tau_{2i}=0$. Rearranging gives $e_{2i}\tau_{2i-1}e_{2i}=0$. Item \emph{(iii)} follows in a similar manner.

\vspace{2mm}
\noindent
\emph{(iv)}: We have
\begin{align*}
e_{2i-2}x_{2i-2}s_{i}e_{2i-2}s_{i} &= e_{2i-2}x_{2i-1}s_{i}e_{2i-2}s_{i} \\
&= e_{2i-2}(s_{i}x_{2i+1}-s_{i}x_{2i}e_{2i}-e_{2i}x_{2i}+x_{2i-1}e_{2i}+\tau_{2i+1})e_{2i-2}s_{i} \\
&= e_{2i-2}s_{i}x_{2i+1}e_{2i-2}s_{i}-e_{2i-2}s_{i}x_{2i}e_{2i}e_{2i-2}s_{i}-e_{2i-2}e_{2i}x_{2i}e_{2i-2}s_{i} \\
& \hspace{10mm}+e_{2i-2}x_{2i-1}e_{2i}e_{2i-2}s_{i}+e_{2i-2}\tau_{2i+1}e_{2i-2}s_{i}
\end{align*}
where the first equality follows from \emph{(9)(i)} of \emph{\Cref{APADefn}}, and the second from \emph{\Cref{AffRecRels} (i)}. We examine the five terms above:
\begin{itemize}
\item[(1)] $e_{2i-2}s_{i}x_{2i+1}e_{2i-2}s_{i} = e_{2i-2}s_{i}e_{2i-2}x_{2i+1}s_{i} = e_{2i-2}e_{2i}x_{2i+1}s_{i}$,
\item[(2)] $-e_{2i-2}s_{i}x_{2i}e_{2i}e_{2i-2}s_{i} = -e_{2i-2}s_{i}e_{2i-2}x_{2i}e_{2i} = -e_{2i-2}e_{2i}x_{2i}e_{2i}=0$,
\item[(3)] $-e_{2i-2}e_{2i}x_{2i}e_{2i-2}s_{i} = -e_{2i-2}e_{2i}x_{2i}s_{i} = -e_{2i-2}e_{2i}x_{2i+1}s_{i}$,
\item[(4)] $e_{2i-2}x_{2i-1}e_{2i}e_{2i-2}s_{i}=e_{2i-2}x_{2i-1}e_{2i-2}e_{2i}s_{i} = 0$,
\item[(5)] $e_{2i-2}\tau_{2i+1}e_{2i-2}s_{i}=e_{2i-2}\tau_{2i+1}s_{i}=e_{2i-2}\tau_{2i}$.
\end{itemize}
Substituting back into the above equation gives $e_{2i-2}x_{2i-2}s_{i}e_{2i-2}s_{i} =e_{2i-2}\tau_{2i}$ as desired.

\vspace{2mm}
\noindent
\emph{(v)}: This follows by applying the anti-automorphism $*$ to \emph{(iv)}.

\vspace{2mm}
\noindent
\emph{(vi)}:
\begin{align*}
\tau_{2i}\tau_{2i-2}\tau_{2i}e_{2i-2} &= \tau_{2i}\tau_{2i-2}(s_{i}e_{2i-2}s_{i}x_{2i-2}e_{2i-2}), \\
&= \tau_{2i}\tau_{2i-2}s_{i-1}e_{2i}s_{i-1}x_{2i-2}e_{2i-2}, \\
&= \tau_{2i}\tau_{2i-1}e_{2i}s_{i-1}x_{2i-2}e_{2i-2}, \\
&= \tau_{2i}(\tau_{2i}e_{2i-2}\tau_{2i}\tau_{2i-1})s_{i-1}x_{2i-2}e_{2i-2}, \\
&= (1-e_{2i})e_{2i-2}\tau_{2i}\tau_{2i-2}x_{2i-2}e_{2i-2}, \\
&= e_{2i-2}\tau_{2i}\tau_{2i-2}x_{2i-2}e_{2i-2}, \\
&= e_{2i-2}\tau_{2i}(x_{2i-2}\tau_{2i-2}+e_{2i-3}e_{2i-2}-e_{2i-2}e_{2i-3})e_{2i-2}, \\
&= e_{2i-2}\tau_{2i}e_{2i-3}e_{2i-2}, \\
&= (e_{2i-2}x_{2i-2}s_{i}e_{2i-2}s_{i})e_{2i-3}e_{2i-2}, \\
&= e_{2i-2}x_{2i-2}s_{i-1}e_{2i}e_{2i-1}e_{2i-2}.
\end{align*}
The first equality follows by \emph{(v)}, the forth from \emph{(5)(iii)} of \emph{\Cref{APADefn}}, the sixth since $e_{2i}\tau_{2i}=0$, the seventh from \emph{(8)(iii)} of \emph{\Cref{APADefn}}, the ninth from $\tau_{2i-2}e_{2i-2}=0$ and \emph{(iii)}, and the tenth from \emph{(iv)}.

\vspace{2mm}
\noindent
\emph{(vii)}:
\begin{align*}
s_{i}e_{2i-2}e_{2i-1}e_{2i}s_{i-1}x_{2i-2}e_{2i-2}s_{i} &= s_{i}e_{2i-2}e_{2i-1}s_{i-1}s_{i}e_{2i-2}s_{i}x_{2i-2}e_{2i-2}s_{i} \\
&= s_{i}e_{2i-2}s_{i}e_{2i-3}e_{2i-2}x_{2i-2}s_{i}e_{2i-2}s_{i} \\
&= s_{i-1}e_{2i}e_{2i-1}e_{2i-2}x_{2i-2}s_{i-1}e_{2i}s_{i-1} \\
&= s_{i-1}e_{2i}e_{2i-1}e_{2i-2}e_{2i-1}\tau_{2i-1}s_{i-1}e_{2i}s_{i-1} \\
&= s_{i-1}e_{2i}e_{2i-1}\tau_{2i-2}e_{2i}s_{i-1} \\
&= s_{i-1}e_{2i}e_{2i-1}e_{2i}\tau_{2i-1} \\
&= s_{i-1}e_{2i}\tau_{2i-1} \\
&= s_{i-1}e_{2i}\tau_{2i-2}s_{i-1} \\
&= s_{i-1}\tau_{2i-2}e_{2i}s_{i-1} \\
&= \tau_{2i-1}e_{2i}s_{i-1} \\
\end{align*}
where the forth equality follows from \emph{\Cref{AffMixRels} (i)}.

\vspace{2mm}
\noindent
\emph{(viii)}:
\begin{align*}
\tau_{2i}\tau_{2i-2}\tau_{2i}e_{2i-2} &= \tau_{2i}\tau_{2i-2}(\tau_{2i-1}e_{2i}\tau_{2i-1}\tau_{2i}) \\
&= \tau_{2i}(s_{i-1}-e_{2i-2})e_{2i}\tau_{2i-1}\tau_{2i} \\
&= \tau_{2i}s_{i-1}e_{2i}\tau_{2i-1}\tau_{2i} \\
&= \tau_{2i}s_{i}e_{2i-2}s_{i}s_{i-1}\tau_{2i-1}\tau_{2i} \\
&= \tau_{2i+1}e_{2i-2}s_{i}\tau_{2i-2}\tau_{2i} \\
&= e_{2i-2}\tau_{2i}\tau_{2i-2}\tau_{2i}
\end{align*}
where the first equality follows from \emph{(5)(iii)} of \emph{\Cref{APADefn}}, the second since $s_{i-1}=\tau_{2i-1}\tau_{2i-2}+e_{2i-2}$, the third since $e_{2i-2}\tau_{2i-1}=0$, and the sixth since $\tau_{2i+1}$ and $e_{2i-2}$ commute.

\end{proof}

\begin{lem}\label{enyangrecursion}
The following relations hold in $\mathcal{A}_{2k}^{\emph{aff}}$:
\begin{align*}
\tau_{2i} = s_{i-1}s_{i}\tau_{2i-2}s_{i}&s_{i-1}+e_{2i-2}x_{2i-2}s_{i}e_{2i-2}s_{i} + s_{i}e_{2i-2}x_{2i-2}s_{i}e_{2i-2} \\
&-e_{2i-2}x_{2i-2}s_{i-1}e_{2i}e_{2i-1}e_{2i-2} - s_{i}e_{2i-2}e_{2i-1}e_{2i}s_{i-1}x_{2i-2}e_{2i-2}s_{i}.
\end{align*}
and
\begin{align*}
\tau_{2i+1} = s_{i-1}s_{i}\tau_{2i-1}s_{i}&s_{i-1}+s_{i}e_{2i-2}x_{2i-2}s_{i}e_{2i-2}s_{i} + e_{2i-2}x_{2i-2}s_{i}e_{2i-2} \\
&- s_{i}e_{2i-2}x_{2i-2}s_{i-1}e_{2i}e_{2i-1}e_{2i-2} - e_{2i-2}e_{2i-1}e_{2i}s_{i-1}x_{2i-2}e_{2i-2}s_{i}.
\end{align*}
\end{lem}

\begin{proof}
We prove the first relation, the second follows from by multiplying on the left by $s_{i}$. We have that
\begin{align*}
s_{i}\tau_{2i-2}s_{i} &= (\tau_{2i}\tau_{2i+1}+e_{2i})\tau_{2i-2}(\tau_{2i+1}\tau_{2i}+e_{2i}) \\
&= \tau_{2i}\tau_{2i+1}^{2}\tau_{2i-2}\tau_{2i}+\tau_{2i-2}e_{2i} \\
&= \tau_{2i}\tau_{2i-2}\tau_{2i}+\tau_{2i-2}e_{2i} 
\end{align*}
where the second equality follows since $\tau_{2i-2}$ commutes with $\tau_{2i+1}$ and $e_{2i}\tau_{2i}=\tau_{2i}e_{2i}=0$. Substituting the above we get
\begin{equation}
s_{i-1}s_{i}\tau_{2i-2}s_{i}s_{i-1} = s_{i-1}\tau_{2i}\tau_{2i-2}\tau_{2i}s_{i-1}+\tau_{2i-1}e_{2i}s_{i-1}.
\end{equation}
For the first term in equation \emph{(10)} we have
\begin{align*}
s_{i-1}\tau_{2i}\tau_{2i-2}\tau_{2i}s_{i-1} &= s_{i-1}(\tau_{2i-2}\tau_{2i}\tau_{2i-2} + \tau_{2i}\tau_{2i-2}\tau_{2i}e_{2i-2})s_{i-1} \\
&= \tau_{2i-1}\tau_{2i}\tau_{2i-1} + s_{i-1}\tau_{2i}\tau_{2i-2}\tau_{2i}e_{2i-2} \\
&= \tau_{2i-1}\tau_{2i}\tau_{2i-1} + \tau_{2i}\tau_{2i-2}\tau_{2i}e_{2i-2} \\
&= \tau_{2i}-e_{2i-2}\tau_{2i}-\tau_{2i}e_{2i-2}+\tau_{2i}\tau_{2i-2}\tau_{2i}e_{2i-2}
\end{align*}
where the first equality follows by \emph{(7)(i)} of \emph{\Cref{APADefn}}, the second from $s_{i-1}\tau_{2i-2}=\tau_{2i-2}s_{i-1}=\tau_{2i-1}$, the third from \emph{\Cref{TauRecPrepRels} (viii)}, and the forth from \emph{(7)(iii)} of \emph{\Cref{APADefn}}. Substituting this back into equation \emph{(10)}, and rearranging yields
\[ \tau_{2i} = s_{i-1}s_{i}\tau_{2i-2}s_{i}s_{i-1}+e_{2i-2}\tau_{2i}+\tau_{2i}e_{2i-2}-\tau_{2i}\tau_{2i-2}\tau_{2i}e_{2i-2}-\tau_{2i-1}e_{2i}s_{i-1}. \]
The desired relation is obtained by applying relations \emph{(iv)} to \emph{(vii)} of \emph{\Cref{TauRecPrepRels}}.

\end{proof}


\subsection{Central elements in $\mathcal{A}_{2k}^{\text{aff}}$}


In this section we describe a central subalgebra of $\mathcal{A}_{2k}^{\text{aff}}$ consisting of certain polynomials in the affine generators. 
We end the section with a conjecture describing the center of $\mathcal{A}_{2k}^{\text{aff}}$.

\begin{lem}\label{AffCostComRels}
The following relations hold:
\begin{itemize}
\item[(i)] $\tau_{2i}x_{2i+1}=x_{2i-1}\tau_{2i}+e_{2i-1}e_{2i}-1$.
\item[(ii)] $\tau_{2i+1}x_{2i+2}=x_{2i}\tau_{2i+1}-e_{2i}e_{2i+1}+1$.
\item[(iii)] $\tau_{2i}x_{2i}=x_{2i}\tau_{2i}+e_{2i-1}e_{2i}-e_{2i}e_{2i-1}$.
\item[(iv)] $\tau_{2i+1}x_{2i+1}=x_{2i+1}\tau_{2i+1}-e_{2i}e_{2i+1}+e_{2i+1}e_{2i}$.
\end{itemize}
\end{lem}

\begin{proof}
\emph{(i)}: Multiplying \emph{(8)(i)} of \emph{\Cref{APADefn}} on the left by $\tau_{2i}$ gives
\begin{align*}
\tau_{2i}x_{2i+1} &= \tau_{2i}^{2}x_{2i-1}\tau_{2i}+\tau_{2i}e_{2i}e_{2i-1}\tau_{2i}+\tau_{2i}^{2}e_{2i-1}e_{2i}-\tau_{2i}^{2} \\
&= (1-e_{2i})x_{2i-1}\tau_{2i}+(1-e_{2i})e_{2i-1}e_{2i}-(1-e_{2i}) \\
&= x_{2i-1}\tau_{2i}+x_{2i-1}e_{2i}\tau_{2i}+e_{2i-1}e_{2i}-e_{2i}-1+e_{2i} \\
&= x_{2i-1}\tau_{2i}+e_{2i-1}e_{2i}-1
\end{align*}
where the second equality follows as $\tau_{2i}^{2}=1-e_{2i}$ and $t_{2i}e_{2i}=0$, and the third since $x_{2i-1}$ and $e_{2i}$ commute.

\vspace{2mm}
\noindent
\emph{(ii)}: Multiplying \emph{(8)(ii)} of \emph{\Cref{APADefn}} on the left by $\tau_{2i+1}$ gives
\begin{align*}
\tau_{2i+1}x_{2i+2} &= \tau_{2i+1}^{2}x_{2i}\tau_{2i+1}+\tau_{2i+1}e_{2i}e_{2i+1}\tau_{2i+1}e_{2i+1}e_{2i}+\tau_{2i+1}^{2} \\
&= (1-e_{2i})x_{2i}\tau_{2i+1}+1-e_{2i} \\
&= x_{2i}\tau_{2i+1}-e_{2i}x_{2i}\tau_{2i+1}+1-e_{2i} \\
&=  x_{2i}\tau_{2i+1}-e_{2i}x_{2i+1}\tau_{2i+1}+1-e_{2i} \\
&= x_{2i}\tau_{2i+1}-e_{2i}e_{2i-1}\tau_{2i+1}^{2}+1-e_{2i} \\
&= x_{2i}\tau_{2i+1}-e_{2i}e_{2i-1}+e_{2i}+1-e_{2i} \\
&= x_{2i}\tau_{2i+1}-e_{2i}e_{2i-1}+1
\end{align*}
where the second equality follows since $\tau_{2i+1}e_{2i}=0$ and $\tau_{2i+1}^{2}=1-e_{2i}$, the forth equality follows since $e_{2i}x_{2i}=e_{2i}x_{2i+1}$, and the fifth equality follows since $e_{2i}x_{2i}=e_{2i}e_{2i-1}\tau_{2i+1}$ (by \emph{\Cref{AffMixRels} (ii)} and \emph{(iii)}).

\vspace{2mm}
\noindent
\emph{(iii)}: Multiplying \emph{(8)(iii)} of \emph{\Cref{APADefn}} on the left by $\tau_{2i}$ gives
\begin{align*}
\tau_{2i}x_{2i} &= \tau_{2i}^{2}x_{2i}\tau_{2i}+\tau_{2i}e_{2i}e_{2i-1}\tau_{2i}+\tau_{2i}^{2}e_{2i-1}e_{2i} \\
&= (1-e_{2i})x_{2i}\tau_{2i}+(1-e_{2i})e_{2i-1}e_{2i} \\
&= x_{2i}\tau_{2i}-e_{2i}x_{2i}\tau_{2i}+e_{2i-1}e_{2i}-e_{2i} \\
&= x_{2i}\tau_{2i}-e_{2i}e_{2i-1}\tau_{2i}^{2}+e_{2i-1}e_{2i}-e_{2i} \\
&= x_{2i}\tau_{2i}-e_{2i}e_{2i-1}+e_{2i}+e_{2i-1}e_{2i}-e_{2i} \\
&= x_{2i}\tau_{2i}-e_{2i}e_{2i-1}+e_{2i-1}e_{2i}
\end{align*}
where the second equality follows since $\tau_{2i}e_{2i}=0$ and $\tau_{2i+1}^{2}=1-e_{2i}$, and the forth equality follows since $e_{2i}x_{2i}=e_{2i}e_{2i-1}\tau_{2i}$ (by \emph{\Cref{AffMixRels} (i)}).

\vspace{2mm}
\noindent
\emph{(iv)}: Multiplying \emph{(8)(iv)} of \emph{\Cref{APADefn}} on the left by $\tau_{2i+1}$ gives
\begin{align*}
\tau_{2i+1}x_{2i+1} &= \tau_{2i+1}^{2}x_{2i+1}\tau_{2i+1}+\tau_{2i+1}e_{2i}e_{2i+1}\tau_{2i+1}+\tau_{2i+1}^{2}e_{2i+1}e_{2i} \\
&= (1-e_{2i})x_{2i+1}\tau_{2i+1}+(1-e_{2i})e_{2i+1}e_{2i} \\
&= x_{2i+1}\tau_{2i+1}-e_{2i}x_{2i+1}\tau_{2i+1}+e_{2i+1}e_{2i}-e_{2i} \\
&= x_{2i+1}\tau_{2i+1}-e_{2i}e_{2i+1}\tau_{2i+1}^{2}+e_{2i+1}e_{2i}-e_{2i}  \\
&= x_{2i}\tau_{2i}-e_{2i}e_{2i+1}+e_{2i}+e_{2i+1}e_{2i}-e_{2i} \\
&= x_{2i}\tau_{2i}-e_{2i}e_{2i+1}+e_{2i+1}e_{2i}
\end{align*}
where the second equality follows since $\tau_{2i+1}e_{2i}=0$ and $\tau_{2i+1}^{2}=1-e_{2i}$, and the forth equality follows since $e_{2i}x_{2i+1}=e_{2i}e_{2i+1}\tau_{2i+1}$ (by \emph{\Cref{AffMixRels} (ii)}).

\end{proof}

\begin{lem}\label{AffIndCostComRels}
For any $n\geq 1$, the following relations hold:
\begin{itemize}
\item[(i)] $\tau_{2i}x_{2i+1}^{n}=x_{2i-1}^{n}\tau_{2i}+\sum\limits_{\substack{a+b=n-1 \\ a,b\geq 0}}x_{2i-1}^{a}(e_{2i-1}e_{2i}-1)x_{2i+1}^{b}$.
\item[(ii)] $\tau_{2i}x_{2i}^{n}=x_{2i}^{n}\tau_{2i}+\sum\limits_{\substack{a+b=n-1 \\ a,b\geq 0}}x_{2i}^{a}(e_{2i-1}e_{2i}-e_{2i}e_{2i-1})x_{2i}^{b}$.
\item[(iii)] $\tau_{2i+1}x_{2i+2}^{n}=x_{2i}^{n}\tau_{2i+1}+\sum\limits_{\substack{a+b=n-1 \\ a,b\geq 0}}x_{2i}^{a}(-e_{2i}e_{2i+1}+1)x_{2i+2}^{b}$.
\item[(iv)] $\tau_{2i+1}x_{2i+1}^{n}=x_{2i+1}^{n}\tau_{2i+1}+\sum\limits_{\substack{a+b=n-1 \\ a,b\geq 0}}x_{2i+1}^{a}(-e_{2i}e_{2i+1}+e_{2i+1}e_{2i})x_{2i+1}^{b}$.
\end{itemize}

\end{lem}

\begin{proof}
This follows from \emph{\Cref{AffCostComRels}} by induction on $n$.

\end{proof}

Let $y_{1},\dots,y_{2k}$ be commuting variables. We let $\mathsf{SSym}[y_{1},\dots,y_{2k}]$ denote the subalgebra of the polynomial algebra $\mathbb{C}[y_{1},\dots,y_{2k}]$ generated by the supersymmetric power-sum polynomials
\[ p_{n}(y_{1},\dots,y_{2k}) := y_{1}^{n}+y_{3}^{n}+\dots+y_{2k-1}^{n}-(y_{2}^{n}+y_{4}^{n}+\dots+y_{2k}^{n}) \]
for all $n\geq 1$.
We have an injective algebra homomorphism $\mathsf{SSym}[y_{1},\dots,y_{2k}]\rightarrow \mathcal{A}_{2k}^{\text{aff}}$ via $y_{i}\mapsto x_{i}$. We denote the image by $\mathsf{SSym}[x_{1}\dots,x_{2k}]$, and let $p_{n}$ denote $p_{n}(x_{1},\dots,x_{2k})$. Let $Z(\mathcal{A}_{2k}^{\text{aff}})$ denote the center of $\mathcal{A}_{2k}^{\text{aff}}$.

\begin{prop}\label{AffSSymPolyCent}
We have that $\mathsf{SSym}[x_{1},\dots,x_{2k}] \subset Z(\mathcal{A}_{2k}^{\emph{aff}})$.
\end{prop}

\begin{proof}
We simply show that each generator of $\mathcal{A}_{2k}^{\text{aff}}$ commutes with each polynomial $p_{n}$. It is immediate that the generators $z_{l}$ and $x_{i}$ commute with $p_{n}$ for any $n\geq 1$ by \emph{(10)(ii)} and \emph{(6)(i)} of \emph{\Cref{APADefn}}. Let $[-,-]$ denote the commutator bracket.

\vspace{2mm}
\noindent
For the generators $e_{2i}$ we have
\[
[p_{n},e_{2i}] = (-x_{2i}^{n}+x_{2i+1}^{n})e_{2i} - e_{2i}(-x_{2i}^{n}+x_{2i+1}^{n})
= (-x_{2i}^{n}+x_{2i}^{n})e_{2i} - e_{2i}(-x_{2i}^{n}+x_{2i}^{n}) = 0,
\]
where the first equality follows from the commuting relation \emph{(6)(iii)} of \emph{\Cref{APADefn}}, and the second equality follows since $x_{2i+1}e_{2i}=x_{2i}e_{2i}$ and $e_{2i}x_{2i+1}=e_{2i}x_{2i}$ by \emph{(9)(ii)} and \emph{(9)(i)} of \emph{\Cref{APADefn}}. Similarly we have $[p_{n},e_{2i-1}]=0$.

\vspace{2mm}
\noindent
For the generator $\tau_{2i}$, the commuting relation \emph{(6)(ii)} of \emph{\Cref{APADefn}} tells us that
\[ [\tau_{2i},p_{n}] = \tau_{2i}(x_{2i-1}^{n}-x_{2i}^{n}+x_{2i+1}^{n}) - (x_{2i-1}^{n}-x_{2i}^{n}+x_{2i+1}^{n})\tau_{2i}. \]
By acting on relation \emph{(i)} of \emph{\Cref{AffIndCostComRels}} by the anti-automorphism $*$, and rearranging, we obtain
\[ \tau_{2i}x_{2i-1}^{n}= x_{2i+1}^{n}\tau_{2i}-\sum\limits_{\substack{a+b=n-1 \\ a,b\geq 0}}x_{2i+1}^{a}(e_{2i}e_{2i-1}-1)x_{2i-1}^{b}. \]
Employing this and relations \emph{(i)} and \emph{(ii)} of \emph{\Cref{AffIndCostComRels}}, we have
\begin{align*}
\tau_{2i}(x_{2i-1}^{n}&-x_{2i}^{n}+x_{2i+1}^{n}) = (x_{2i-1}^{n}-x_{2i}^{n}+x_{2i+1}^{n})\tau_{2i} + \sum\limits_{\substack{a+b=n-1 \\ a,b\geq 0}}x_{2i-1}^{a}(e_{2i-1}e_{2i}-1)x_{2i+1}^{b} \\
&-\sum\limits_{\substack{a+b=n-1 \\ a,b\geq 0}}x_{2i}^{a}(e_{2i-1}e_{2i}-e_{2i}e_{2i-1})x_{2i}^{b} -\sum\limits_{\substack{a+b=n-1 \\ a,b\geq 0}}x_{2i+1}^{a}(e_{2i}e_{2i-1}-1)x_{2i-1}^{b}
\end{align*}
Hence showing that $[\tau_{2i},p_{n}]=0$ is equivalent to showing that the three summations above sum to zero. This follows by changing the second summation accordingly:
\[ -\sum\limits_{\substack{a+b=n-1 \\ a,b\geq 0}}x_{2i}^{a}(e_{2i-1}e_{2i}-e_{2i}e_{2i-1})x_{2i}^{b} =  -\sum\limits_{\substack{a+b=n-1 \\ a,b\geq 0}}x_{2i}^{a}e_{2i-1}e_{2i}x_{2i}^{b}-x_{2i}^{a}e_{2i}e_{2i-1}x_{2i}^{b}\]
\[ \hspace{4mm} =-\sum\limits_{\substack{a+b=n-1 \\ a,b\geq 0}}x_{2i-1}^{a}e_{2i-1}e_{2i}x_{2i+1}^{b}-x_{2i+1}^{a}e_{2i}e_{2i-1}x_{2i-1}^{b} \]
by repeat application of relations \emph{(9)(i)} and \emph{(9)(ii)} of \emph{\Cref{APADefn}}. One shows $[\tau_{2i+1},p_{n}]=0$ analogously.

\end{proof}

Under the projection $\mathsf{pr}:\mathcal{A}_{2k}^{\text{aff}}\rightarrow \mathcal{A}_{2k}$ the subalgebra $\mathsf{SSym}[x_{1},\dots,x_{2k}]$ gets sent to $\mathsf{SSym}[X_{1},\dots,X_{2k}]$, showing that such a subalgebra is central in $\mathcal{A}_{2k}$. It was shown in \cite[Thm 4.2.6]{Cre21} that  this is in fact the whole center of $\mathcal{A}_{2k}$.  Note that in \cite{Cre21}, the centre is given as $\mathsf{SSym}[N_{1},\dots,N_{2k}]$ where
\[ N_{i} := \begin{cases}
                 \frac{z}{2}-1-X_{i}, & \text{ if $i$ odd}, \\
                 X_{i}-\frac{z}{2}+1, & \text{ if $i$ even}. 
               \end{cases} \]
But one  can easily see  that $\mathsf{SSym}[X_{1},\dots,X_{2k}]=\mathsf{SSym}[N_{1},\dots,N_{2k}]$. Based on this and comparing with the centers of other affine diagram algebras, see \cite[Corollary 4.10]{Naz96} and \cite[Theorem 4.2]{DVR11}, leads to a natural conjecture for the center of $\mathcal{A}_{2k}^{\text{aff}}$:

\begin{conj}
$Z(\mathcal{A}_{2k}^{\text{aff}})=\langle z_{l}, p_{n} \ | \ l,n\in\mathbb{Z}_{\geq 0} \rangle$.
\end{conj}


\subsection{Extending the Action on Tensor Spaces}


\vspace{2mm}
We now seek to extend the action of $\mathcal{A}_{2k}$ on $V^{\otimes k}$ to one of $\mathcal{A}_{2k}^{\text{aff}}$ on $M\otimes V^{\otimes k}$, where $M$ is any $\mathbb{C}S(n)$-module. The tensor space $M\otimes V^{\otimes k}$ is also viewed as an $\mathbb{C}S(n)$-module by the diagonal action. Before extending the action, we briefly define some central elements in $\mathbb{C}S(n)$. For each $b\in[n]$ and $l\in \mathbb{N}$, we let
\[ T_{n,b} := \sum_{a\in[n]\backslash\{b\}}(a,b), \hspace{3mm} \text{and} \hspace{3mm} Z_{n,l}:=\sum_{b\in[n]}T_{n,b}^{l}. \]
So $T_{n,b}$ is the sum of all transposition containing $b$, and $Z_{n,l}$ is the $l$-power sum in $T_{n,b}$ as $b$ runs from 1 to $n$.

\begin{lem}
For each $l\in\mathbb{N}$, we have that $Z_{n,l}$ belongs to the center of $\mathbb{C}S(n)$.
\end{lem}

\begin{proof}
This follows since $\pi T_{n,b} = T_{n,\pi(b)}\pi$ for any $\pi\in S(n)$. 

\end{proof}

\begin{thm}
Given any $\mathbb{C}S(n)$-module $M=\mathsf{Span}_{\mathbb{C}}\{m_{1},\dots,m_{d}\}$, we have a $\mathbb{C}$-algebra homomorphism
\[ \psi_{n,k}^{(M)}:\mathcal{A}_{2k}^{\emph{aff}}\rightarrow \mathsf{End}_{S(n)}(M\otimes V^{\otimes k}) \]
defined on the generators by
\begin{align*}
\psi_{n,2k}^{(M)}(e_{2i-1})(m_{a_{0}}\otimes v_{\bm{a}}) &= \sum_{b=1}^{n}m\otimes v_{a_{1}}\otimes \dots \otimes v_{a_{i-1}} \otimes v_{b} \otimes v_{a_{i+1}}\otimes \dots \otimes v_{a_{k}}, \\
\psi_{n,2k}^{(M)}(e_{2i})(m_{a_{0}}\otimes v_{\bm{a}}) &= \delta_{a_{i}, a_{i+1}}m_{a_{0}}\otimes v_{\bm{a}}, \\
\psi_{n,2k}^{(M)}(\tau_{2i})(m_{a_{0}}\otimes v_{\bm{a}}) &= \varepsilon_{a_{i},a_{i+1}}(a_{i}, a_{i+1})(m_{a_{0}}\otimes v_{a_{1}}\otimes \dots \otimes v_{a_{i-1}})\otimes v_{a_{i}}\otimes \dots \otimes v_{a_{k}}, \\
\psi_{n,2k}^{(M)}(\tau_{2i+1})(m_{a_{0}}\otimes v_{\bm{a}}) &= \varepsilon_{a_{i},a_{i+1}}(a_{i}, a_{i+1})(m_{a_{0}}\otimes v_{a_{1}}\otimes \dots \otimes v_{a_{i+1}})\otimes v_{a_{i+2}}\otimes \dots \otimes v_{a_{k}}, \\
\psi_{n,2k}^{(M)}(x_{2i-1})(m_{a_{0}}\otimes v_{\bm{a}}) &= \sum\limits_{\substack{b=1 \\ b\neq a_{i}}}^{n}(b, a_{i})(m\otimes v_{a_{1}}\otimes \dots \otimes v_{a_{i-1}})\otimes v_{a_{i}}\otimes \dots \otimes v_{a_{k}}, \\
\psi_{n,2k}^{(M)}(x_{2i})(m_{a_{0}}\otimes v_{\bm{a}}) &= \sum\limits_{\substack{b=1 \\ b\neq a_{i}}}^{n}(b, a_{i})(m\otimes v_{a_{1}}\otimes \dots \otimes v_{a_{i}})\otimes v_{a_{i+1}}\otimes \dots \otimes v_{a_{k}},  \\  
\psi_{n,2k}^{(M)}(z_{l})(m_{a_{0}}\otimes v_{\bm{a}}) &= \left(Z_{n,l}m_{a_{0}}\right)\otimes v_{\bm{a}},
\end{align*}
for all $(a_{0},\bm{a})\in[d]\times [n]^{k}$, extended $\mathbb{C}$-linearly across $M\otimes V^{\otimes k}$.
\end{thm}

\begin{proof}
This can been shown by direct computations, much of which are fairly simple but lengthy. To ease notation, for any tuple $\bm{a}=(a_{0},a_{1},\dots,a_{k})\in[d]\times [n]^{k}$, we represent a simple tensor in $M\otimes V^{\otimes k}$ by a word in the entries of $\bm{a}$, that is $a_{0}a_{1}\dots a_{k} := m_{a_{0}}\otimes v_{a_{1}}\otimes \dots \otimes v_{a_{k}}$. We begin by showing that $\psi_{n,k}^{(M)}$ is well-defined, that is to confirm that these endomorphisms do indeed commute with the diagonal action of $S(n)$. We do this by showing for any $\pi\in S(n)$, that $\pi\psi_{n,k}^{(M)}(g)\pi^{-1}=\psi_{n,k}^{(M)}(g)$ for each generator $g$ of $\mathcal{A}_{2k}^{\text{aff}}$.

\vspace{2mm}
\noindent
One can deduce that $\pi\psi_{n,k}^{(M)}(e_{i})\pi^{-1}=\psi_{n,k}^{(M)}(e_{i})$ since the action of the generators $e_{i}$ ignores the $M$ component, and hence this follows from \emph{\Cref{SWD}}.

\vspace{2mm}
\noindent
For the generators $\tau_{2i}$,
\begin{align*}
\pi&\psi_{n,k}^{(M)}(\tau_{2i})\pi^{-1}(\bm{a}) = \pi\psi_{n,k}^{(M)}(\tau_{2i})\Big(\pi^{-1}(a_{0}a_{1}\dots a_{k})\Big) \\
&= \varepsilon_{a_{i},a_{i+1}}\pi\Big((\pi^{-1}(a_{i}),\pi^{-1}(a_{i+1}))\pi^{-1}(a_{0}a_{1}\dots a_{i-1})\pi^{-1}(a_{i}\dots a_{k})\Big) \\
&= \varepsilon_{a_{i},a_{i+1}}\pi(\pi^{-1}(a_{i}),\pi^{-1}(a_{i+1}))\pi^{-1}(a_{0}a_{1}\dots a_{i-1})a_{i}\dots a_{k} \\
&= \varepsilon_{a_{i},a_{i+1}}(a_{i},a_{i+1})(a_{0}a_{1}\dots a_{i-1})a_{i}\dots a_{k} \\
&= \psi_{n,k}^{(M)}(\tau_{2i})(\bm{a})
\end{align*}
noting $\varepsilon_{\pi^{-1}(a_{i}),\pi^{-1}(a_{i+1})}=\varepsilon_{a_{i},a_{i+1}}$. One can show $\pi\psi_{n,k}^{(M)}(\tau_{2i+1})\pi^{-1}=\psi_{n,k}(\tau_{2i+1})$ in a similar manner.

\vspace{2mm}
\noindent
For the generators $x_{2i-1}$,
\begin{align*}
\pi&\psi_{n,k}^{(M)}(x_{2i-1})\pi^{-1}(\bm{a}) = \pi\psi_{n,k}^{(M)}(x_{2i-1})\Big(\pi^{-1}(a_{0}a_{1}\dots a_{k})\Big) \\
&= \pi\left(\sum\limits_{\substack{b\in[n] \\ b\neq \pi^{-1}(a_{i})}}(b,\pi^{-1}(a_{i}))\pi^{-1}(a_{0}a_{1}\dots a_{i-1})\pi^{-1}(a_{i}\dots a_{k})\right) \\
&= \sum\limits_{\substack{b\in[n] \\ b\neq \pi^{-1}(a_{i})}}\pi(b,\pi^{-1}(a_{i}))\pi^{-1}(a_{0}a_{1}\dots a_{i-1})a_{i}\dots a_{k} \\
&= \sum\limits_{\substack{b\in[n] \\ b\neq \pi^{-1}(a_{i})}}(\pi(b),a_{i})(a_{0}a_{1}\dots a_{i-1})a_{i}\dots a_{k} \\
&= \sum\limits_{\substack{b'\in[n] \\ b'\neq a_{i}}}(b',a_{i})(a_{0}a_{1}\dots a_{i-1})a_{i}\dots a_{k} \\
&= \psi_{n,2k}^{(M)}(x_{2i-1})(\bm{a}) 
\end{align*}
by the substitution $b'=\pi(b)$. One can show $\pi\psi_{n,k}^{(M)}(x_{2i})\pi^{-1}=\psi_{n,k}(x_{2i})$ in a similar manner. Lastly $\pi\psi_{n,k}^{(M)}(z_{l})\pi^{-1}=\psi_{n,k}(z_{l})$ can be seen since $Z_{n,l}$ are central in $\mathbb{C}S(n)$.

One now needs to confirm that the defining relations of $\mathcal{A}_{2k}^{\text{aff}}$ in \emph{\Cref{APADefn}} are upheld under $\psi_{n,k}^{(M)}$. As mentioned, these can be shown by direct, but lengthy computations. With this in mind, we will only give details of some of the more difficult relations, namely relations \emph{(8)} through \emph{(10)}. Note that the Braid-like relations \emph{(7)} follow in a analogous manner to the proof of \emph{\Cref{BraidLikeRels}}.

\vspace{2mm}
\noindent
\emph{(8)(i)}: We seek to show that
\[ \psi_{n,k}^{(M)}(x_{2i+1})=\psi_{n,k}^{(M)}(\tau_{2i}x_{2i-1}\tau_{2i}+e_{2i}e_{2i-1}\tau_{2i}+\tau_{2i}e_{2i-1}e_{2i}-\tau_{2i}). \]
To show this we examine how each term on the hand right side acts on the simple tensor $\bm{a}$, and show that the sum recovers the action of $x_{2i+1}$. It proves easier to do this by tackling two cases, when $a_{i}\neq a_{i+1}$ and when $a_{i}=a_{i+1}$.

(Case 1): Assume $a_{i}\neq a_{i+1}$, then for the first term we have
\begin{align*}
&\psi_{n,k}(\tau_{2i}x_{2i-1}\tau_{2i})(\bm{a}) = \psi_{n,k}(\tau_{2i}x_{2i-1})\Big((a_{i},a_{i+1})(a_{0}a_{1}\dots a_{i-1})a_{i}\dots a_{k}\Big) \\
&= \psi_{n,k}(\tau_{2i})\left(\sum\limits_{\substack{b\in[n] \\ b\neq a_{i}}}(b,a_{i})(a_{i},a_{i+1})(a_{0}a_{1}\dots a_{i-1})a_{i}\dots a_{k}\right) \\
&= \sum\limits_{\substack{b\in[n] \\ b\neq a_{i}}}(a_{i},a_{i+1})(b,a_{i})(a_{i},a_{i+1})(a_{0}a_{1}\dots a_{i-1})a_{i}\dots a_{k} \\
&= \sum\limits_{\substack{b\in[n] \\ b\neq a_{i}}}((a_{i},a_{i+1})(b),a_{i+1})(a_{0}a_{1}\dots a_{i-1})a_{i}\dots a_{k} \\
&= \sum\limits_{\substack{c\in[n] \\ c\neq a_{i+1}}}(c,a_{i+1})(a_{0}a_{1}\dots a_{i-1})a_{i}\dots a_{k} \\
&= \sum\limits_{\substack{c\in[n] \\ c\neq a_{i+1}}}(c,a_{i+1})(a_{0}a_{1}\dots a_{i})a_{i+1}\dots a_{k} + (a_{i},a_{i+1})(a_{0}a_{1}\dots a_{i-1})a_{i}\dots a_{k} \\
& \hspace{45mm}-(a_{i},a_{i+1})(a_{0}a_{1}\dots a_{i})a_{i+1}\dots a_{k} \\
\\
&= \psi_{n,k}^{(M)}(x_{2i+1})(\bm{a}) + \psi_{n,k}^{(M)}(\tau_{2i})(\bm{a}) - (a_{i},a_{i+1})(a_{0}a_{1}\dots a_{i})a_{i+1}\dots a_{k}
\end{align*}
where we employed the substitution $c=(a_{i},a_{i+1})(b)$. For the second term,
\begin{align*}
\psi_{n,k}^{(M)}&(e_{2i}e_{2i-1}\tau_{2i})(\bm{a}) = \psi_{n,k}^{(M)}(e_{2i}e_{2i-1})\Big((a_{i},a_{i+1})(a_{0}a_{1}\dots a_{i-1})a_{i}\dots a_{k}\Big) \\
&= \psi_{n,k}^{(M)}(e_{2i})\left(\sum_{b=1}^{n}(a_{i},a_{i+1})(a_{0}a_{1}\dots a_{i-1})ba_{i+1}\dots a_{k}\right) \\
&= (a_{i},a_{i+1})(a_{0}a_{1}\dots a_{i-1})a_{i+1}a_{i+1}\dots a_{k} \\
\\
&= (a_{i},a_{i+1})(a_{0}a_{1}\dots a_{i})a_{i+1}\dots a_{k}. 
\end{align*}
For the third term $\psi_{n,k}^{(M)}(\tau_{2i}e_{2i-1}e_{2i})(\bm{a})=0$ since $a_{i}\neq a_{i+1}$. Thus collectively,
\begin{align*}
\psi_{n,k}^{(M)}&(\tau_{2i}x_{2i-1}\tau_{2i}+e_{2i}e_{2i-1}\tau_{2i}+\tau_{2i}e_{2i-1}e_{2i}-\tau_{2i}) \\
&=\psi_{n,k}^{(M)}(\tau_{2i}x_{2i-1}\tau_{2i})+\psi_{n,k}^{(M)}(e_{2i}e_{2i-1}\tau_{2i})+\psi_{n,k}^{(M)}(\tau_{2i}e_{2i-1}e_{2i})-\psi_{n,k}^{(M)}(\tau_{2i}) \\
&=\psi_{n,k}^{(M)}(x_{2i+1})(\bm{a}) + \psi_{n,k}^{(M)}(\tau_{2i})(\bm{a}) - (a_{i},a_{i+1})(a_{0}a_{1}\dots a_{i})a_{i+1}\dots a_{k} \\
&\hspace{25mm}+(a_{i},a_{i+1})(a_{0}a_{1}\dots a_{i})a_{i+1}\dots a_{k} -\psi_{n,k}^{(M)}(\tau_{2i}) \\
&=\psi_{n,k}^{(M)}(x_{2i+1})(\bm{a}).
\end{align*}

(Case 2): Assume $a_{i}=a_{i+1}$. Then $\psi_{n,k}^{(M)}(\tau_{2i})(\bm{a})=0$, and so
\[ \psi_{n,k}^{(M)}(\tau_{2i}x_{2i-1}\tau_{2i}+e_{2i}e_{2i-1}\tau_{2i}+\tau_{2i}e_{2i-1}e_{2i}-\tau_{2i}) = \psi_{n,k}^{(M)}(\tau_{2i}e_{2i-1}e_{2i}). \]
Hence we just need to confirm that $\psi_{n,k}^{(M)}(x_{2i+1})=\psi_{n,k}^{(M)}(\tau_{2i}e_{2i-1}e_{2i})$. Well,
\begin{align*}
\psi_{n,k}^{(M)}(\tau_{2i}e_{2i-1}e_{2i})(\bm{a}) &= \psi_{n,k}^{(M)}(\tau_{2i})\left(\sum_{b=1}^{n}a_{0}a_{1}\dots a_{i-1}ba_{i+1}\dots a_{k}\right) = \sum_{b=1}^{n}(b,a_{i+1})(a_{0}a_{1}\dots a_{i-1})ba_{i+1}\dots a_{k} \\
&= \sum_{b=1}^{n}(b,a_{i+1})(a_{0}a_{1}\dots a_{i})a_{i+1}\dots a_{k} = \psi_{n,k}^{(M)}(x_{2i+1}).
\end{align*}
The remaining \emph{Skein-like} relations follow by employing similar arguments.

\vspace{2mm}
\noindent
\emph{(9)(i)}: We seek to show $\psi_{n,k}^{(M)}(e_{i}x_{i})=\psi_{n,k}^{(M)}(e_{i}x_{i+1})$. We show this first when working with $e_{2i}$, then with $e_{2i-1}$. Assume $a_{i}\neq a_{i+1}$, then
\begin{align*}
\psi_{n,k}^{(M)}(e_{2i}x_{2i})(\bm{a}) &= \psi_{n,k}^{(M)}(e_{2i})\left(\sum\limits_{\substack{b=1 \\ b\neq a_{i}}}(b,a_{i})(a_{0}a_{1}\dots a_{i})a_{i+1}\dots a_{k}\right) = (a_{i},a_{i+1})(a_{0}a_{1}\dots a_{i})a_{i+1}\dots a_{k},
\end{align*}
\begin{align*}
\psi_{n,k}^{(M)}(e_{2i}x_{2i+1})(\bm{a}) &= \psi_{n,k}^{(M)}(e_{2i})\left(\sum\limits_{\substack{b=1 \\ b\neq a_{i+1}}}(b,a_{i+1})(a_{0}a_{1}\dots a_{i})a_{i+1}\dots a_{k}\right) = (a_{i},a_{i+1})(a_{0}a_{1}\dots a_{i})a_{i+1}\dots a_{k}.
\end{align*}
When $a_{i}=a_{i+1}$ one can check that $\psi_{n,k}^{(M)}(e_{2i}x_{2i})=\psi_{n,k}^{(M)}(e_{2i}x_{2i+1})=0$, thus $\psi_{n,k}^{(M)}(e_{2i}x_{2i})=\psi_{n,k}^{(M)}(e_{2i}x_{2i+1})$. For odd indices we have
\begin{align*}
\psi_{n,k}^{(M)}(e_{2i-1}x_{2i-1})(\bm{a}) &= \psi_{n,k}^{(M)}(e_{2i-1})\left(\sum\limits_{\substack{b=1 \\ b\neq a_{i}}}(b,a_{i})(a_{0}a_{1}\dots a_{i-1})a_{i}\dots a_{k}\right) = \sum_{c=1}^{n}\sum\limits_{\substack{b=1 \\ b\neq a_{i}}}(b,a_{i})(a_{0}a_{1}\dots a_{i-1})ca_{i+1}\dots a_{k},
\end{align*}
\begin{align*}
\psi_{n,k}^{(M)}(e_{2i-1}x_{2i})(\bm{a}) &= \psi_{n,k}^{(M)}(e_{2i-1})\left(\sum\limits_{\substack{b=1 \\ b\neq a_{i}}}(b,a_{i})(a_{0}a_{1}\dots a_{i})a_{i+1}\dots a_{k}\right) = \sum_{c=1}^{n}\sum\limits_{\substack{b=1 \\ b\neq a_{i}}}(b,a_{i})(a_{0}a_{1}\dots a_{i-1})ca_{i+1}\dots a_{k}.
\end{align*}
Thus $\psi_{n,k}^{(M)}(e_{i}x_{i})=\psi_{n,k}^{(M)}(e_{i}x_{i+1})$. Relation \emph{(9)(ii)} may be shown in a similar manner.

\vspace{2mm}
\noindent
\emph{(10)(i)}:
\begin{align*}
\psi_{n,k}^{(M)}(e_{1}x_{1}^{l}e_{1})(\bm{a}) &= \psi_{n,k}^{(M)}(e_{1}x_{1}^{l})\left(\sum_{b=1}^{n}a_{0}ba_{2}\dots a_{k}\right) = \psi_{n,k}^{(M)}(e_{1})\left(\sum_{b=1}^{n}(T_{n,b}^{l}a_{0})ba_{2}\dots a_{k}\right) \\
&= \sum_{c=1}^{n}\left(\sum_{b=1}^{n}T_{n,b}^{l}a_{0}\right)ca_{2}\dots a_{k} = \sum_{c=1}^{n}\left(Z_{n,l}a_{0}\right)ca_{2}\dots a_{k} \\
&=\psi_{n,k}^{(M)}(z_{l})\left(\sum_{c=1}^{n}a_{0}ca_{2}\dots a_{k}\right) = \psi_{n,k}^{(M)}(z_{l})\left(\psi_{n,k}^{(M)}(e_{1})(a_{0}a_{1}a_{2}\dots a_{k})\right) = \psi_{n,k}^{(M)}(z_{l}e_{1})(\bm{a}).
\end{align*}
Lastly relation \emph{(10)(ii)} is simple to check since $Z_{n,l}$ belongs to the center of $\mathbb{C}S(n)$. 

\end{proof}

\begin{cor}\label{infinitesub}
The subalgebra $\langle \tau_{i}, e_{j} \rangle_{i,j}$ of $\mathcal{A}_{2k}^{\emph{aff}}$ is infinite dimensional over $\mathbb{C}[z_{0}]$. 
\end{cor}

\begin{proof}
Set $d_{m}:=x_{1}^{m}e_{2}e_{1}$ for all $m\in\mathbb{N}$. We first show that $d_{m}\in\langle \tau_{i}, e_{j} \rangle$ by induction on $m$. By \emph{\Cref{AffMixRels}} (i) we see that $e_{2}x_{2}\in\langle \tau_{i}, e_{j} \rangle$. Then multiplying on the right by $e_{1}$ yeilds $e_{2}x_{2}e_{1} = e_{2}x_{1}e_{1}=x_{1}e_{2}e_{1} = d_{1}$, where the first equality follows from (9)(ii) of \emph{\Cref{APADefn}}, and the second from (6)(iii). Thus we have the base case $d_{1}\in\langle \tau_{i}, e_{j} \rangle$. Assume $d_{m'}\in\langle \tau_{i}, e_{j} \rangle$ for all $m'<m$ with $m\geq 2$, we seek to show that $d_{m}\in\langle \tau_{i}, e_{j} \rangle$. Well
\begin{align*}
d_{m-1}\tau_{2}e_{1}&=x_{1}^{m-1}e_{2}e_{1}\tau_{2}e_{1}=x_{1}^{m-1}e_{2}x_{2}e_{1} = x_{1}^{m-1}e_{2}x_{1}e_{1} = x_{1}^{m}e_{2}e_{1}=d_{n},
\end{align*}
where the second equality follows from \emph{\Cref{AffMixRels}} (i), and the remaining equalities follow in the same manner as the base case. Hence $d_{m}\in\langle \tau_{i}, e_{j} \rangle$ completing induction. We now seek to show that the set $\{d_{m} \ | \ m\in\mathbb{N}\}$ is $\mathbb{C}[z_{0}]$-linearly independent in $\mathcal{A}_{2k}^{\text{aff}}$, which will complete the proof. Let $I\subset\mathbb{N}$ be finite and assume
\[ \sum_{m\in I}h_{m}(z_{0})d_{m} = 0, \]
where $h_{m}(z_{0})$ are polynomials in $\mathbb{C}[z_{0}]$. We seek to show that $h_{m}(z_{0})=0$ for each $m\in I$. Let $M\in I$ be the maximal element, and let $R$ be the set of roots for each $h_{m}(z_{0})$. Pick an $n\in\mathbb{N}$ such that $n>M+1$ and $n\notin R$. Let $F$ be any free $\mathbb{C}S(n)$-module. For any $f\in F$ and $(a_{1},\dots,a_{k})\in [n]^{k}$, we have
\[ \psi_{n,k}^{(F)}(d_{m})(f\otimes v_{a_{1}}\otimes v_{a_{2}}\otimes \dots \otimes v_{a_{k}}) = \left(T_{n,a_{2}}^{m}f\right)\otimes v_{a_{2}}\otimes v_{a_{2}}\otimes v_{a_{3}}\otimes \dots \otimes v_{a_{k}}. \]
Since $F$ is free, it will follow that the set $\{\psi_{n,k}^{(F)}(d_{m}) \ | \ m\in I\}$ is linear independent in $\mathsf{End}_{S(n)}(F\otimes V^{\otimes k})$ if the set $\{T_{n,a_{2}}^{m} \ | \ m \in I\}$ is linearly independent in $\mathbb{C}S(n)$. This follows since $n>M+1$, and hence $T_{n,a_{2}}^{m}$ contains a permutation consisting of a single cycle of size $m+1$, while all permutations in $T_{n,a_{2}}^{m'}$ must have smaller support whenever $m'<m$. now consider the equation
\[ \psi_{n,k}^{(F)}\left(\sum_{m\in I}h_{m}(z_{0})d_{m}\right) = \sum_{m\in I}h_{m}(n)\psi_{n,k}^{(F)}(d_{m}) = 0. \]
Since $n$ is not a root of any $h_{m}(z_{0})$, and the set $\{\psi_{n,k}^{(F)}(d_{m}) \ | \ m\in I\}$ is linear independent, we must have that $h_{m}(z_{0})=0$ for each $m\in I$.

\end{proof}


\section{Connections with the Heisenberg category}

J. Brundan and M. Vargas recently defined in \cite{BV21} an affine partition category $\mathsf{APar}$ as a monoidal subcategory of the Heisenberg category introduced by Khovanov in \cite{Kho14} generated by certain objects and morphisms. This was based on the observation made by S. Likeng and A. Savage in \cite{LS21} that the partition category can be realised inside the Heisenberg category.
This affine partition category naturally gives rise to another definition of an affine partition algebra, which they denote by $AP_k$ by taking the endomorphism algebra ${\rm End}_{\mathsf{APar}}((\uparrow \downarrow)^k)$ for the object $(\uparrow \downarrow)^k$ in $\mathsf{APar}$ (see section 4.1. below).

 Inspired by the work of Brundan and Vargas, we construct a surjective homomorphism $\varphi$ from $\mathcal{A}^{\text{aff}}_{2k}$ to ${\rm End}_{\mathsf{Heis}}((\uparrow \downarrow)^k)$. In fact, our argument generalises to show that Brundan and Vargas' affine partition category $\mathsf{APar}$ is the \emph{full} monoidal subcategory in $\mathsf{Heis}$ generated by the object $\uparrow \downarrow$. As a corollary we obtain that $AP_k$ is a quotient of $\mathcal{A}^{\text{aff}}_{2k}$. 
 
 We start by recalling the definition of the Heisenberg category.


\subsection{Heisenberg Category}


The Heisenberg Category $\mathsf{Heis}$ is a $\mathbb{C}$-linear monoidal category originally defined by M. Khovanov in \cite{Kho14}. The objects of $\mathsf{Heis}$ are generated, as a monoidal category, by the two objects $\uparrow$ and $\downarrow$. We use juxtaposition to denote the tensor product of objects, and the monoidal identity object is the empty word $\emptyset$. Hence we view the free monoid $\langle \uparrow, \downarrow\rangle$ as the set of objects in $\mathsf{Heis}$. Consider two objects $\bm{a}=a_{1}\cdots a_{n}$ and $\bm{b}=b_{1}\cdots b_{m}$ for $a_{i},b_{i}\in\{\uparrow,\downarrow\}$. The space of morphisms $\mathsf{Hom}_{\mathsf{Heis}}(\bm{a},\bm{b})$ is the $\mathbb{C}$-vector space generated by certain diagrams modulo local relations. We call such diagrams $(\bm{a},\bm{b})$-diagrams and define them as follows: Firstly, we work in the strip $\mathbb{R}\times [0,1]$ with boundary $\mathbb{B}:=\mathbb{R}\times\{1\} \cup \mathbb{R}\times \{0\}$. We call an orientated immersion of the interval $[0,1]$ and circle $S^{1}$ a \emph{string} and \emph{loop} respectively. We denote orientations by drawing an arrow on the curve. Now consider the set of points $E=[n]\times \{1\}\cup[m]\times \{0\}$, and colour $(i,1)\in [n]\times\{1\}$ and $(j,0)\in [m]\times\{0\}$ with the symbols $a_{i}$ and $b_{j}$ respectively. We say that a set partition of $E$ into pairs is an $(\bm{a},\bm{b})$-matching if pairs of points in the same row are coloured by opposite arrows, while pairs of points in different rows are coloured by the same arrow. Then an $(\bm{a},\bm{b})$-diagram is a finite collection of strings and loops, modulo rel boundary isotopies, such that: 
\begin{itemize}
\item[(D1)] The endpoints of the strings induce an $(\bm{a},\bm{b})$-matching on $E$
\item[(D2)] There are only finitely many points of intersection, and no triple or tangential intersections occur
\item[(D3)] The boundary $\mathbb{B}$ doesn't intersect any loops, and only intersects strings at the endpoints $E$
\end{itemize}
For example let $\bm{a}=\downarrow\downarrow\uparrow$ and $\bm{b}=\uparrow\downarrow\downarrow\downarrow\uparrow$, then
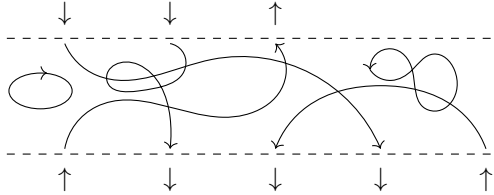
\[ \begin{tikzpicture}[scale=1.4]
    \node[label=above:{$\scriptsize \downarrow$}] at (1,1){};
    \node[label=above:{$\scriptsize \downarrow$}] at (2,1){};
    \node[label=above:{$\scriptsize \uparrow$}] at (3,1){};

    \node[label=below:{$\scriptsize \uparrow$}] at (1,0){};
    \node[label=below:{$\scriptsize \downarrow$}] at (2,0){};
    \node[label=below:{$\scriptsize \downarrow$}] at (3,0){};
    \node[label=below:{$\scriptsize \downarrow$}] at (4,0){};
    \node[label=below:{$\scriptsize \uparrow$}] at (5,0){};

    \draw[->] (1,0) to [curve through={(1.3,0.4) .. (2.5,0.3) .. (3.1,0.7)}] (3,1);
    \draw[<-] (2,0) to [curve through={(2,0.3) .. (1.7,0.8) .. (1.4,0.65) .. (1.7,0.55) .. (2.15,0.8)}] (2,1);
    \draw[<-] (3,0) to [curve through={(3.3,0.4) .. (4,0.6) .. (4.7,0.4)}] (5,0);
    \draw[<-] (4,0) to [curve through={(2.4,0.85) .. (1.3,0.7)}] (1,1);

    \draw[<-] (3.9,0.75) to [curve through={(4.1,0.95) .. (4.3,0.75) .. (4.45,0.37) .. (4.7,0.75) .. (4.5,0.9)  (4.3,0.75) .. (4.1,0.65)}] (3.9,0.75);

\begin{scope}[decoration={markings, mark=at position 0.25 with {\arrow{<}}}]
           \draw[postaction={decorate}] (0.77,0.55) ellipse (3mm and 1.7mm);
\end{scope}

\begin{scope}[dashed]
          \draw (0.45,1.05) to (5.1,1.05);
          \draw (0.45,-0.05) to (5.1,-0.05);
\end{scope}

\end{tikzpicture} \]
is a $(\bm{a},\bm{b})$-diagram. Isotopic deformation of the interior of $\mathbb{R}\times [0,1]$ is allowed, and will preserve the relative structure of the points of intersection. If a loop contains no intersections we call it a \emph{bubble}. Bubbles can have clockwise or anticlockwise orientation. If the endpoints of a string occur in different rows we call it a \emph{vertical} string, and it has either a \emph{down} or \emph{up} orientation. If the endpoints belong to the same row then we call it an \emph{arc}. Non self-intersecting arcs have either  a \emph{clockwise} or \emph{anti-clockwise} orientation. In the above example there are two loops, one of which is a bubble, and four strings, three of which are vertical and one an arc. We call an endpoint of a string a \emph{source} if the arrow of orientation points away from it, and a \emph{target} otherwise. We consider $(\bm{a},\bm{b})$-diagrams modulo the following local relations:
\begin{itemize}
\item[(H1)] \hspace{52mm}
\begin{tikzpicture}[baseline=-1ex, scale=0.5]
\begin{scope}[dotted]
    \draw (0,0) circle (1.414cm);
\end{scope}
    \draw (-1,-1) to (1,1);
    \draw (1,-1) to (-1,1);
    \draw (0,1.414) to [curve through={((0,1.3) .. (-0.7,0) .. (0,-1.3)}] (0,-1.414);
\end{tikzpicture} \hspace{1mm} = \hspace{1mm} 
\begin{tikzpicture}[baseline=-1ex, scale=0.5]
\begin{scope}[dotted]
    \draw (0,0) circle (1.414cm);
\end{scope}
    \draw (-1,-1) to (1,1);
    \draw (1,-1) to (-1,1);
    \draw (0,1.414) to [curve through={((0,1.3) .. (0.7,0) .. (0,-1.3)}] (0,-1.414);
\end{tikzpicture}
\item[(H2)] \hspace{25mm}
\begin{tikzpicture}[baseline=-1ex, scale=0.5]
\begin{scope}[dotted]
    \draw (0,0) circle (1.414cm);
\end{scope}
    \draw[->] (-1,-1) to [curve through={(-0.8,-0.9) .. (0.45,0) .. (-0.8,0.9)}] (-1,1);
    \draw[->] (1,-1) to [curve through={(0.8,-0.9) .. (-0.45,0) .. (0.8,0.9)}] (1,1);
\end{tikzpicture} \hspace{1mm} = \hspace{1mm}
\begin{tikzpicture}[baseline=-1ex, scale=0.5]
\begin{scope}[dotted]
    \draw (0,0) circle (1.414cm);
\end{scope}
    \draw[->] (-1,-1) to [curve through={(-0.8,0)}] (-1,1);
    \draw[->] (1,-1) to [curve through={(0.8,0)}] (1,1);
\end{tikzpicture} , \hspace{10mm}
\begin{tikzpicture}[baseline=-1ex, scale=0.5]
\begin{scope}[dotted]
    \draw (0,0) circle (1.414cm);
\end{scope}
    \draw[->] (-1,-1) to [curve through={(-0.8,-0.9) .. (0.45,0) .. (-0.8,0.9)}] (-1,1);
    \draw[<-] (1,-1) to [curve through={(0.8,-0.9) .. (-0.45,0) .. (0.8,0.9)}] (1,1);
\end{tikzpicture} \hspace{1mm} = \hspace{1mm}
\begin{tikzpicture}[baseline=-1ex, scale=0.5]
\begin{scope}[dotted]
    \draw (0,0) circle (1.414cm);
\end{scope}
    \draw[->] (-1,-1) to [curve through={(-0.8,0)}] (-1,1);
    \draw[<-] (1,-1) to [curve through={(0.8,0)}] (1,1);
\end{tikzpicture}
\item[(H3)] \hspace{40mm}
\begin{tikzpicture}[baseline=-1ex, scale=0.5]
\begin{scope}[dotted]
    \draw (0,0) circle (1.414cm);
\end{scope}
    \draw[<-] (-1,-1) to [curve through={(-0.8,-0.9) .. (0.45,0) .. (-0.8,0.9)}] (-1,1);
    \draw[->] (1,-1) to [curve through={(0.8,-0.9) .. (-0.45,0) .. (0.8,0.9)}] (1,1);
\end{tikzpicture} \hspace{1mm} = \hspace{1mm}
\begin{tikzpicture}[baseline=-1ex, scale=0.5]
\begin{scope}[dotted]
    \draw (0,0) circle (1.414cm);
\end{scope}
    \draw[<-] (-1,-1) to [curve through={(-0.8,0)}] (-1,1);
    \draw[->] (1,-1) to [curve through={(0.8,0)}] (1,1);
\end{tikzpicture} \hspace{1mm} - \hspace{1mm}
\begin{tikzpicture}[baseline=-1ex, scale=0.5]
\begin{scope}[dotted]
    \draw (0,0) circle (1.414cm);
\end{scope}
    \draw[<-] (-1,-1) to [curve through={(0,-0.8)}] (1,-1);
    \draw[->] (-1,1) to [curve through={(0,0.8)}] (1,1);
\end{tikzpicture}
\item[(H4)] \hspace{40mm}
\begin{tikzpicture}[baseline=-1ex, scale=0.6]
\begin{scope}[dotted]
         \draw (0,0) circle (1.15cm);
\end{scope}
         \draw[->] (0,-1.15) to [curve through={(0.1,-0.9) .. (-0.45,0.5) .. (-0.7,0) .. (-0.45,-0.5) .. (0.1,0.9)}] (0,1.15);
\end{tikzpicture} \hspace{1mm} = 0, \hspace{10mm}
\begin{tikzpicture}[baseline=-1ex, scale=0.5]
\begin{scope}[dotted]
    \draw (0,0) circle (1.414cm);
\end{scope}
    \draw[->] (0,1) to [curve through={(-1,0) .. (0,-1) .. (1,0)}] (0,1);
\end{tikzpicture} \hspace{1mm} = 1
\end{itemize}
Relation $(H1)$ holds regardless of orientations. To apply such a local relation to an $(\bm{a},\bm{b})$-diagram one locates a disk which is isotopic to one of the disks above, then replace such a disk according to the corresponding equation. Note that any of the local relations may be rotated in any way to give an equivalent relation. Relation $(H1)$ tells us that any curve may past over a crossing, and relations $(H2)$ and $(H3)$ tells us how to pull part orientated curves, where $(H3)$ shows that this can not always be done for free. Relation $(H4)$ tells us that left curls kill $(\bm{a},\bm{b})$-diagrams, and that any anti-clockwise bubble may be removed for free. 

The composition of morphisms is given by vertical concatenation of diagrams, and rescaling (and extending $\mathbb{C}$-linearly). We denote composition by juxtaposition of symbols. When $\bm{a}=\bm{b}$ we write $\bm{a}$-diagram instead of $(\bm{a},\bm{a})$-diagram. The morphism space $\mathsf{End}_{\mathsf{Heis}}(\bm{a})$ is a $\mathbb{C}$-algebra with identity given by the diagram of non-intersecting vertical strings. Now for later use, we collect some relations regarding arbitrary $(\bm{a},\bm{b})$-diagrams. The following local relation follows from (H2) and (H3), see also \cite[(3.5)]{LS21}:
\begin{lem}\label{BubComTwoStrings}
Clockwise bubbles satisfy the commuting relation
\[ 
\begin{tikzpicture}[baseline=2ex, scale=0.7]
    \draw[->] (1,0) to (1,1);
    \draw[<-] (2,0) to (2,1);
    \draw[<-] (3,1) to [curve through={(2.5,0.5) .. (3,0) .. (3.5,0.5)}] (3,1);
\end{tikzpicture} \hspace{2mm} = \hspace{2mm}
\begin{tikzpicture}[baseline=2ex, scale=0.7]
    \draw[->] (1,0) to (1,1);
    \draw[<-] (2,0) to (2,1);
    \draw[<-] (0,1) to [curve through={(-0.5,0.5) .. (0,0) .. (0.5,0.5)}] (0,1);
\end{tikzpicture} \]
\begin{flushright} $\square$ \end{flushright}
\end{lem}
Although left curls annihilate diagrams, right curls do not, and they play an important role. We will represent right curls by a decoration, and label such decorations with weights to denote multiplicity:
\[ \begin{tikzpicture}[baseline=3ex, scale=1.4]
    \draw[->] (1,0) to (1,1);
    \filldraw (1,0.5) circle (2pt);
\end{tikzpicture} \hspace{1mm} := \hspace{1mm} 
\begin{tikzpicture}[baseline=3ex, scale=1.4]
    \draw[->] (1,0) to [curve through={(1.1,0.5) .. (1.25,0.65) .. (1.4,0.5) .. (1.25,0.35) .. (1.1,0.5)}] (1,1);
\end{tikzpicture}, \hspace{15mm}
\begin{tikzpicture}[baseline=3ex, scale=1.4]
    \node[label=right:{$l$}] (l) at (1,0.5){};
    \draw[->] (1,0) to (1,1);
    \filldraw (1,0.5) circle (2pt);
\end{tikzpicture} \hspace{1mm} := \hspace{1mm} 
\begin{tikzpicture}[baseline=3ex, scale=1.4]
    \node[label=right:{$\Big\} \hspace{1mm} l$ times}] (l) at (1.05,0.5){};
    \node[ ] (vd) at (1.1,0.55){$\vdots$};
    \draw[->] (1,0) to (1,1);
    \filldraw (1,0.7) circle (2pt);
    \filldraw (1,0.3) circle (2pt);
\end{tikzpicture} \]

The following result is a simple application of the local relations. 

\begin{lem}\label{DecRels}
The following two local relations hold:
\[ \begin{tikzpicture}[baseline=2.5ex, scale=0.9]
\begin{scope}[dotted]
    \draw (1.5,0.5) circle (0.7071cm);
\end{scope}
    \draw[->] (1,0) to (2,1);
    \draw[->] (2,0) to (1,1);
    \filldraw (1.25,0.25) circle (2pt);
\end{tikzpicture} \hspace{1mm} = \hspace{1mm} 
\begin{tikzpicture}[baseline=2.5ex, scale=0.9]
\begin{scope}[dotted]
    \draw (1.5,0.5) circle (0.7071cm);
\end{scope}
    \draw[->] (1,0) to (2,1);
    \draw[->] (2,0) to (1,1);
    \filldraw (1.75,0.75) circle (2pt);
\end{tikzpicture} \hspace{1mm} + \hspace{1mm}
\begin{tikzpicture}[baseline=2.5ex, scale=0.9]
\begin{scope}[dotted]
    \draw (1.5,0.5) circle (0.7071cm);
\end{scope}
    \draw[->] (1,0) to (1,1);
    \draw[->] (2,0) to (2,1);
\end{tikzpicture},  \hspace{10mm}
\begin{tikzpicture}[baseline=2.5ex, scale=0.9]
\begin{scope}[dotted]
    \draw (1.5,0.5) circle (0.7071cm);
\end{scope}
    \draw[->] (1,0) to (2,1);
    \draw[->] (2,0) to (1,1);
    \filldraw (1.75,0.25) circle (2pt);
\end{tikzpicture} \hspace{1mm} = \hspace{1mm} 
\begin{tikzpicture}[baseline=2.5ex, scale=0.9]
\begin{scope}[dotted]
    \draw (1.5,0.5) circle (0.7071cm);
\end{scope}
    \draw[->] (1,0) to (2,1);
    \draw[->] (2,0) to (1,1);
    \filldraw (1.25,0.75) circle (2pt);
\end{tikzpicture} \hspace{1mm} - \hspace{1mm}
\begin{tikzpicture}[baseline=2.5ex, scale=0.9]
\begin{scope}[dotted]
    \draw (1.5,0.5) circle (0.7071cm);
\end{scope}
    \draw[->] (1,0) to (1,1);
    \draw[->] (2,0) to (2,1);
\end{tikzpicture},  \]
\begin{flushright} $\square$ \end{flushright}
\end{lem}

We now recall a basis for the morphism spaces $\mathsf{Hom}_{\mathsf{Heis}}(\bm{a},\bm{b})$ presented in \cite{Kho14}. We first introduce a few definitions to help us describe this basis in a manner which will lend itself better for later results. 
\begin{defn} \label{SimDia}
For $\bm{a},\bm{b}\in\langle\uparrow,\downarrow\rangle$, we say an $(\bm{a},\bm{b})$-diagram is \emph{simple} if it contains no loops, no self-intersections, and two strings intersect at most once. Let $\mathsf{Sim}(\bm{a},\bm{b})$ denote the set of simple $(\bm{a},\bm{b})$-diagrams, and write $\mathsf{Sim}(\bm{a})$ for $\mathsf{Sim}(\bm{a},\bm{a})$.
\end{defn} 
Given words $\bm{a}=a_{1}\cdots a_{n}, \bm{b}=b_{1}\cdots b_{m}\in\langle\uparrow,\downarrow\rangle$ with $a_{i},b_{j}\in\{\uparrow,\downarrow\}$, let $\bm{b}^{*}$ denote the word obtained from $\bm{b}$ by replacing up arrows with down arrows, and down arrows with up arrows. Let $u$ equal the number of up arrows appearing in $\bm{a}$ and $\bm{b}^{*}$, and $d$ the number of a down arrows. Then by $(D1)$, one can deduce that $\mathsf{Hom}_{\mathsf{Heis}}(\bm{a},\bm{b})$ is non-empty if and only if $u=d$. In this situation we have that $|\mathsf{Sim}(\bm{a},\bm{b})|=u!$, since there is one simple $(\bm{a},\bm{b})$-diagram for every $(\bm{a},\bm{b})$-matching. Such a correspondence is given by reading the pairings of endpoints formed from the strings of a simple diagram.

\begin{ex} \label{SimpleDiaEx}
Consider the words $\bm{a}=\uparrow\downarrow$ and $\bm{b}=\downarrow\uparrow\uparrow\downarrow$. Then the $6=3!$ simple $(\bm{a},\bm{b})$-diagrams are
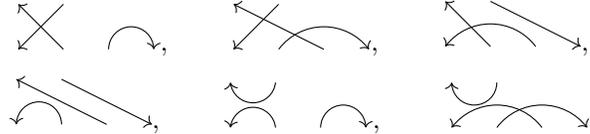
\[ 
\begin{tikzpicture}[scale=0.6]
    \draw[->] (1,1) to (0,0);
    \draw[->] (1,0) to (0,1);
    \draw[->] (2,0) to [curve through={(2.5,0.5)}] (3,0);
\end{tikzpicture}, \hspace{8mm}
\begin{tikzpicture}[scale=0.6]
    \draw[->] (1,1) to (0,0);
    \draw[<-] (0,1) to (2,0);
    \draw[->] (1,0) to [curve through={(2,0.5)}] (3,0);
\end{tikzpicture}, \hspace{8mm}
\begin{tikzpicture}[scale=0.6]
    \draw[->] (1,1) to (3,0);
    \draw[->] (1,0) to (0,1);
    \draw[<-] (0,0) to [curve through={(1,0.5)}] (2,0);
\end{tikzpicture},
\]
\[
\begin{tikzpicture}[scale=0.6]
    \draw[<-] (0,1) to (2,0);
    \draw[->] (1,1) to (3,0);
    \draw[<-] (0,0) to [curve through={(0.5,0.5)}] (1,0);
\end{tikzpicture}, \hspace{8mm}
\begin{tikzpicture}[scale=0.6]
    \draw[<-] (0,1) to [curve through={(0.5,0.55)}] (1,1);
    \draw[<-] (0,0) to [curve through={(0.5,0.45)}] (1,0);
    \draw[->] (2,0) to [curve through={(2.5,0.5)}] (3,0);
\end{tikzpicture}, \hspace{8mm}
\begin{tikzpicture}[scale=0.6]
    \draw[<-] (0,1) to [curve through={(0.5,0.5)}] (1,1);
    \draw[<-] (0,0) to [curve through={(1,0.5)}] (2,0);
    \draw[->] (1,0) to [curve through={(2,0.5)}] (3,0);
\end{tikzpicture}.
\]
These diagrams are in a one-to-one correspondence with the $(\bm{a},\bm{b})$-matchings of the set of endpoints $E=\{(i,1),(j,0) \ | \ i\in[2], j\in[4]\}$. As an example, for the first diagram above we have
\[ \begin{tikzpicture}[baseline=1.5ex, scale=0.6]
    \draw[->] (1,1) to (0,0);
    \draw[->] (1,0) to (0,1);
    \draw[->] (2,0) to [curve through={(2.5,0.5)}] (3,0);
\end{tikzpicture} \hspace{4mm} \leftrightarrow \hspace{4mm} \Big\{ \{(1,1),(2,0)\}, \{(2,1),(1,0)\}, \{(3,0),(4,0)\} \Big\}. \] 
\end{ex}
The basis for $\mathsf{Hom}_{\mathsf{Heis}}(\bm{a},\bm{b})$ we describe below is obtained by adding decorations (right curls) and decorated clockwise loops to all the simple $(\bm{a},\bm{b})$-diagrams in a partiuclar manner. We describe this by introducing some basic diagrams and using the composition of diagrams.
\begin{defn}
Let $\bm{a}=a_{1}\cdots a_{n}\in\langle\uparrow,\downarrow\rangle$ for $a_{i}\in\{\uparrow,\downarrow\}$. For $i\in[n], w\in\mathbb{Z}_{\geq 0}$, define the $\bm{a}$-diagrams
\[ r_{i} \ = \
\begin{tikzpicture}[baseline=1.5ex, scale=0.7]
    \draw (1,0) to (1,1);
    \node[label=above:{$\scriptsize \cdots$}] at (2,0.2){};
    \draw (3,0) to (3,1);
    \draw (4,0) to (4,1);
    \draw (5,0) to (5,1);
    \node[label=above:{$\scriptsize \cdots$}] at (6,0.2){};
    \draw (7,0) to (7,1);
    \node[label=above:{$\scriptsize a_{1}$}] at (1,1){};
    \node[label=above:{$\scriptsize a_{i-1}$}] at (3,1){};
    \node[label=above:{$\scriptsize a_{i}$}] at (4,1){};
    \node[label=above:{$\scriptsize a_{i+1}$}] at (5,1){};
    \node[label=above:{$\scriptsize a_{n}$}] at (7,1){};
    \filldraw (4,0.5) circle (1.8pt);
\end{tikzpicture}, \hspace{10mm}
c_{w} \ = \hspace{2mm}
\begin{tikzpicture}[baseline=1.5ex, scale=0.7]
    \draw (1,0) to (1,1);
    \node[label=above:{$\scriptsize \cdots$}] at (2,0.2){};
    \draw (3,0) to (3,1);
    \node[label=above:{$\scriptsize a_{1}$}] at (1,1){};
    \node[label=above:{$\scriptsize a_{n}$}] at (3,1){};
\begin{scope}[decoration={markings, mark=at position 0.25 with {\arrow{<}}}]
           \draw[postaction={decorate}] (0,0.6) circle (3.5mm);
\end{scope}
    \filldraw (0,0.25) circle (1.8pt);
    \node[label=below:{$w$}] (w) at (0,0.25){};
\end{tikzpicture}
\]
The orientation of strings is taken to match $\bm{a}$. Although both $r_{i}$ and $c_{w}$ depend on $\bm{a}$, we surpress this fact as it should be clear from context. 
\end{defn}

\begin{defn} \label{DiaBasis}
Given any $\bm{a}=a_{1}\cdots a_{n},\bm{b}=b_{1}\cdots b_{m}\in\langle\uparrow,\downarrow\rangle$ for $a_{i},b_{j}\in\{\uparrow,\downarrow\}$, let $\mathsf{B}(\bm{a},\bm{b})$ be the set of $(\bm{a},\bm{b})$-diagrams of the form
\[ c_{w}^{k_{w}}\cdots c_{1}^{k_{1}}c_{0}^{k_{0}}r_{1}^{s_{1}}\cdots r_{n}^{s_{n}}\alpha r_{1}^{t_{1}}\cdots r_{m}^{t_{m}} \]
where $\alpha\in\mathsf{Sim}(\bm{a},\bm{b})$, $w,k_{l},s_{i},t_{j}\in\mathbb{Z}_{\geq 0}$, and $s_{i}=t_{j}=0$ whenever $(i,1)$ and $(j,0)$ are sources respectively. We write $\mathsf{B}(\bm{a})$ for $\mathsf{B}(\bm{a},\bm{a})$.
\end{defn}

\begin{ex}
Given $\bm{a}=\uparrow\downarrow$ and $\bm{b}=\downarrow\uparrow\uparrow\downarrow$, an example of an element of $\mathsf{B}(\bm{a},\bm{b})$ is
\[ c_{5}c_{0}^{2}r_{1}^{3}\alpha r_{4}^{2} \hspace{2mm} = \hspace{2mm}
\begin{tikzpicture}[baseline=4ex, scale=1.2]
    \draw[->] (1,1) to (3,0);
    \draw[->] (1,0) to (0,1);
    \draw[<-] (0,0) to [curve through={(1,0.5)}] (2,0);
\begin{scope}[decoration={markings, mark=at position 0.25 with {\arrow{<}}}]
           \draw[postaction={decorate}] (-0.5,0.7) circle (2.5mm);
\end{scope}
\begin{scope}[decoration={markings, mark=at position 0.25 with {\arrow{<}}}]
           \draw[postaction={decorate}] (-0.95,0.3) circle (2.4mm);
\end{scope}
\begin{scope}[decoration={markings, mark=at position 0.25 with {\arrow{<}}}]
           \draw[postaction={decorate}] (-1.45,0.5) circle (2.5mm);
\end{scope}
    \filldraw (0.2,0.8) circle (1.8pt);
    \filldraw (2.7,0.15) circle (1.8pt);
    \filldraw (-1.45,0.25) circle (1.8pt);
    \node[label=right:{$3$}] (s_{1}) at (0.2,0.8){};
    \node[label=above:{$2$}] (t_{4}) at (2.7,0.15){};
    \node[label=below:{$5$}] (k_{3}) at (-1.45,0.25){};
\end{tikzpicture}
\]
where $\alpha$ is the third simple $(\bm{a},\bm{b})$-diagram in the list given in \emph{\Cref{SimpleDiaEx}}.
\end{ex}
The following result is \emph{Proposition 5} of \cite{Kho14}.
\begin{thm} \label{BasisThm}
The set $\mathsf{B}(\bm{a},\bm{b})$ is a basis for $\mathsf{Hom}_{\mathsf{Heis}}(\bm{a},\bm{b})$.
\end{thm}

\begin{rmk}
The description of this basis is analogous to the basis given by regular monomials presented in \cite[Theorem 4.6]{Naz96}. Note that no decorated anti-clockwise bubbles appear. This is due to the fact that any decorated anti-clockwise bubble may be expressed as a linear combination of decorated clockwise bubbles, see for example \cite[Proposition 2]{Kho14}.
\end{rmk}


\subsection{A surjective homomorphism $\varphi:\mathcal{A}_{2k}^{\text{aff}}\rightarrow\mathsf{End}_{\mathsf{Heis}}((\uparrow\downarrow)^{k})$}


 When drawing a $(\uparrow\downarrow)^{k}$-diagram, instead of labelling the endpoints with arrows, we instead will label the points $(i,1)$ with $i$ for each $1\leq i\leq 2k$, since the parity of the label recovers the orientation of the arrow. Also to ease notation we employ the following diagrammatic shorthand for elements of $\mathsf{B}((\uparrow\downarrow)^{k})$:
\[
\begin{tikzpicture}[baseline=2.5ex, scale=0.8]
    \draw[->] (1,0) to (1,1);
    \draw[<-] (2,0) to (2,1);
    \node (d1) at (2.5,0.5){\dots};
    \draw (3,0) to (3,1);
    
    \node (a) at (5,0.5){$\beta$};
\begin{scope}[dashed]
          \draw (4,0) to (4,1);
          \draw (4,0) to (6,0);
          \draw (6,0) to (6,1);
          \draw (4,1) to (6,1);
\end{scope}

    \node (r) at (0.2,0.5){$\rho$};
\begin{scope}[dashed]
    \draw (0.2,0.5) circle (5mm);
\end{scope}

    \draw (7,0) to (7,1);
    \node (d2) at (7.5,0.5){\dots};
    \draw[->] (8,0) to (8,1);
    \draw[<-] (9,0) to (9,1);

    \node[label=above:{\scriptsize $1$}] (1) at (1,1){};
    \node[label=above:{\scriptsize $u$}] (u) at (4,1){};
    \node[label=above:{\scriptsize $v$}] (v) at (6,1){};
    \node[label=above:{\scriptsize $2k$}] (2k) at (9,1){};
\end{tikzpicture} \leftrightsquigarrow \hspace{3mm}
\begin{tikzpicture}[baseline=2.5ex, scale=0.8]
    \node (r) at (0.2,0.5){$\rho$};
\begin{scope}[dashed]
    \draw (0.2,0.5) circle (5mm);
\end{scope}
    \node (a) at (2,0.5){$\beta$};
\begin{scope}[dashed]
          \draw (1,0) to (1,1);
          \draw (1,0) to (3,0);
          \draw (3,0) to (3,1);
          \draw (1,1) to (3,1);
\end{scope}
    \node[label=above:{\scriptsize $u$}] (u) at (1,1){};
    \node[label=above:{\scriptsize $v$}] (v) at (3,1){};
\end{tikzpicture},  \]
where $\beta$ is loopless, $\rho$ is a collection of (possibly decorated) clockwise bubbles, and $u,v\in[2k]$. Hence we drop the trivial vertical strings but retain the labels $u$ through $v$, allowing one to recover the original diagram.  

\begin{prop} \label{APADiaHom}
We have a $\mathbb{C}$-algebra homomorphism
\[ \varphi:\mathcal{A}_{2k}^{\emph{aff}}\rightarrow\mathsf{End}_{\mathsf{Heis}}((\uparrow\downarrow)^{k}) \]
given on the generators by
\begin{align*}
\varphi(e_{2i-1}) \hspace{1mm} &= 
\begin{tikzpicture}[baseline=2ex, scale=0.9]
    \node[label=above:{\scriptsize $2i-1$}] (2i-1) at (1,1){};
    \node[label=above:{\scriptsize $2i$}] (2i) at (2,1){};
    \draw[<-] (1,1) to [bend right] (2,1);
    \draw[->] (1,0) to [bend left] (2,0);
\end{tikzpicture}, \hspace{5mm}
\varphi(e_{2i}) \hspace{1mm} = 
\begin{tikzpicture}[baseline=2ex, scale=0.9]
    \node[label=above:{\scriptsize $2i$}] (2i) at (1,1){};
    \node[label=above:{\scriptsize $2i+1$}] (2i+1) at (2,1){};
    \draw[->] (1,1) to [bend right] (2,1);
    \draw[<-] (1,0) to [bend left] (2,0);
\end{tikzpicture}, \hspace{5mm}
\varphi(x_{2i-1}) \hspace{1mm} = 
\begin{tikzpicture}[baseline=2ex, scale=0.9]
    \node[label=above:{\scriptsize $2i-1$}] (2i-1) at (1,1){};
    \draw[->] (1,0) to (1,1);
    \filldraw (1,0.5) circle (2pt);
\end{tikzpicture}, \hspace{5mm}
\varphi(x_{2i}) \hspace{1mm} = 
\begin{tikzpicture}[baseline=2ex, scale=0.9]
    \node[label=above:{\scriptsize $2i$}] (2i) at (1,1){};
    \draw[<-] (1,0) to (1,1);
    \filldraw (1,0.5) circle (2pt);
\end{tikzpicture}, \\
\\
\varphi(\tau_{2i}) \hspace{1mm} &= 
\begin{tikzpicture}[baseline=2ex, scale=0.9]
    \node[label=above:{\scriptsize $2i-1$}] (2i-1) at (1,1){};
    \node[label=above:{\scriptsize $2i$}] (2i) at (2,1){};
    \node[label=above:{\scriptsize $2i+1$}] (2i+1) at (3,1){};
    \draw[->] (1,0) to (3,1);
    \draw[<-] (2,0) to [curve through={(1.9,0.2) .. (1.6,0.5) .. (1.9,0.8)}] (2,1);
    \draw[->] (3,0) to (1,1);
\end{tikzpicture}, \hspace{5mm}
\varphi(\tau_{2i+1}) \hspace{1mm} = 
\begin{tikzpicture}[baseline=2ex, scale=0.9]
    \node[label=above:{\scriptsize $2i$}] (2i) at (1,1){};
    \node[label=above:{\scriptsize $2i+1$}] (2i+1) at (2,1){};
    \node[label=above:{\scriptsize $2i+2$}] (2i+2) at (3,1){};
    \draw[<-] (1,0) to (3,1);
    \draw[->] (2,0) to [curve through={(1.9,0.2) .. (1.6,0.5) .. (1.9,0.8)}] (2,1);
    \draw[<-] (3,0) to (1,1);
\end{tikzpicture}, \hspace{5mm}
\varphi(z_{l}) \hspace{1mm} = \hspace{1mm}
\begin{tikzpicture}[baseline=2.2ex, scale=0.9]
    \node[label=right:{$\scriptsize $l$$}] (l) at (0.9,0.5){};
\begin{scope}[decoration={markings, mark=at position 0.25 with {\arrow{<}}}]
           \draw[postaction={decorate}] (0.5,0.5) ellipse (4mm and 3.3mm);
\end{scope}
    \filldraw (0.9,0.5) circle (2pt); 
\end{tikzpicture}
\end{align*}
\end{prop}

\begin{proof}
This will be shown by directly checking that each of the defining relations in \emph{\Cref{APADefn}} is satisfied under the map $\varphi$. Most of these are simple to check but lengthy, hence for such relations we do not give full details.

\vspace{2mm}
\noindent
\emph{(1)(i)}:
\begin{align*}
\varphi(\tau_{2i}^{2}) \hspace{1mm} &= 
\begin{tikzpicture}[baseline=5ex, scale=0.8]
    \node[label=above:{\scriptsize $2i-1$}] (2i-1) at (1,2){};
    \node[label=above:{\scriptsize $2i$}] (2i) at (2,2){};
    \node[label=above:{\scriptsize $2i+1$}] (2i+1) at (3,2){};
    \draw[->] (1,1) to (3,2);
    \draw[<-] (2,1) to [curve through={(1.9,1.2) .. (1.6,1.5) .. (1.9,1.8)}] (2,2);
    \draw[->] (3,1) to (1,2);
    \draw[->] (1,0) to (3,1);
    \draw[<-] (2,0) to [curve through={(1.9,0.2) .. (1.6,0.5) .. (1.9,0.8)}] (2,1);
    \draw[->] (3,0) to (1,1);
\end{tikzpicture} \hspace{1mm} = \hspace{1mm}
\begin{tikzpicture}[baseline=2.5ex, scale=0.8]
    \node[label=above:{\scriptsize $2i-1$}] (2i-1) at (1,1){};
    \node[label=above:{\scriptsize $2i$}] (2i) at (2,1){};
    \node[label=above:{\scriptsize $2i+1$}] (2i+1) at (3,1){};
    \draw[->] (1,0) to [curve through={(1.3,0.15) .. (2.5,0.35) .. (2.7,0.5) .. (2.5,0.65) .. (1.3,0.85)}] (1,1);
    \draw[->] (3,0) to [curve through={(2.7,0.15) .. (1.5,0.35) .. (1.3,0.5) .. (1.5,0.65) .. (2.7,0.85)}] (3,1);
    \draw[<-] (2,0) to [curve through={(2.15,0.15) .. (2.5,0.5) .. (2.15,0.85)}] (2,1);
\end{tikzpicture} \hspace{4mm} \text{by (H1)} \\
&= 
\begin{tikzpicture}[baseline=2.5ex, scale=0.8]
    \node[label=above:{\scriptsize $2i-1$}] (2i-1) at (1,1){};
    \node[label=above:{\scriptsize $2i$}] (2i) at (2,1){};
    \node[label=above:{\scriptsize $2i+1$}] (2i+1) at (3,1){};
    \draw[->] (1,0) to (1,1);
    \draw[->] (3,0) to [curve through={(2.7,0.15) .. (2.2,0.35) .. (2,0.5) .. (2.2,0.65) .. (2.7,0.85)}] (3,1);
    \draw[<-] (2,0) to [curve through={(2.15,0.15) .. (2.5,0.5) .. (2.15,0.85)}] (2,1);
\end{tikzpicture} \hspace{37mm} \text{by (H2)} \\
&= 
\begin{tikzpicture}[baseline=2.5ex, scale=0.8]
    \node[label=above:{\scriptsize $2i-1$}] (2i-1) at (1,1){};
    \node[label=above:{\scriptsize $2i$}] (2i) at (2,1){};
    \node[label=above:{\scriptsize $2i+1$}] (2i+1) at (3,1){};
    \draw[->] (1,0) to (1,1);
    \draw[->] (3,0) to (3,1);
    \draw[<-] (2,0) to (2,1);
\end{tikzpicture} \hspace{1mm} - \hspace{1mm}
\begin{tikzpicture}[baseline=2.5ex, scale=0.8]
    \node[label=above:{\scriptsize $2i-1$}] (2i-1) at (1,1){};
    \node[label=above:{\scriptsize $2i$}] (2i) at (2,1){};
    \node[label=above:{\scriptsize $2i+1$}] (2i+1) at (3,1){};
    \draw[->] (1,0) to (1,1);
    \draw[->] (3,0) to [bend right] (2,0);
    \draw[->] (2,1) to [bend right] (3,1);
\end{tikzpicture} \hspace{5mm} \text{by (H3)}
\end{align*}
which equals $\varphi(1-e_{2i})$. One can show that relation \emph{(1)(ii)} is upheld in a similar manner.

\vspace{2mm}
\noindent
\emph{(2)}: Relation \emph{(2)(i)} is $\tau_{2i+1}\tau_{2j}=\tau_{2j}\tau_{2i+1}$ for all $j\neq i+1$. When $j\neq i$, it is clear to see diagrammatically that this relation is upheld under $\varphi$. For case $j=i$, one applies (H1) and then (H2) to see that
\[ \varphi(\tau_{2i+1}\tau_{2i}) = 
\begin{tikzpicture}[baseline=2ex, scale=0.8]
    \node[label=above:{\scriptsize $2i-1$}] (2i-1) at (1,1){};
    \node[label=above:{\scriptsize $2i$}] (2i) at (2,1){};
    \node[label=above:{\scriptsize $2i+1$}] (2i+1) at (3,1){};
    \node[label=above:{\scriptsize $2i+2$}] (2i+2) at (4,1){};
    \draw[->] (1,0) to (3,1);
    \draw[<-] (2,0) to (4,1);
    \draw[->] (3,0) to (1,1);
    \draw[<-] (4,0) to (2,1);
\end{tikzpicture}
= \varphi(\tau_{2i}\tau_{2i+1}). \]
Both relations \emph{(2)(ii)} and \emph{(2)(iii)} can be seen to hold under $\varphi$ diagrammatically. For relation \emph{(2)(iv)}, we have that
\[ \varphi(s_{i}) = \varphi(\tau_{2i+1}\tau_{2i} +e_{2i}) = 
\begin{tikzpicture}[baseline=2ex, scale=0.8]
    \node[label=above:{\scriptsize $2i-1$}] (2i-1) at (1,1){};
    \node[label=above:{\scriptsize $2i$}] (2i) at (2,1){};
    \node[label=above:{\scriptsize $2i+1$}] (2i+1) at (3,1){};
    \node[label=above:{\scriptsize $2i+2$}] (2i+2) at (4,1){};
    \draw[->] (1,0) to (3,1);
    \draw[<-] (2,0) to (4,1);
    \draw[->] (3,0) to (1,1);
    \draw[<-] (4,0) to (2,1);
\end{tikzpicture} \hspace{1mm} + \hspace{1mm}
\begin{tikzpicture}[baseline=2ex, scale=0.8]
    \node[label=above:{\scriptsize $2i-1$}] (2i-1) at (1,1){};
    \node[label=above:{\scriptsize $2i$}] (2i) at (2,1){};
    \node[label=above:{\scriptsize $2i+1$}] (2i+1) at (3,1){};
    \node[label=above:{\scriptsize $2i+2$}] (2i+2) at (4,1){};
    \draw[->] (1,0) to (1,1);
    \draw[<-] (2,0) to [bend left] (3,0);
    \draw[->] (2,1) to [bend right] (3,1);
    \draw[<-] (4,0) to (4,1);
\end{tikzpicture}. \]
Such elements satisfy the braid relation $s_{i}s_{i+1}s_{i}=s_{i+1}s_{i}s_{i+1}$ by \cite[Theorem 4.1]{LS21}.

\vspace{2mm}
\noindent
\emph{(3)}: Relation \emph{(3)(i)} is upheld under $\varphi$ by applying \emph{\Cref{BubComTwoStrings}}, and \emph{(3)(ii)} is upheld by (H4). Relations \emph{(3)(iii)} and \emph{(3)(iv)} are upheld by the fact that left curls are annihilated. For relation \emph{(3)(v)}, it is clear that applying (H1) allows one to go from the diagram $\varphi(\tau_{2i}e_{2i-1}e_{2i+1})$ to $\varphi(\tau_{2i+1}e_{2i-1}e_{2i+1})$, and similarly for relation \emph{(3)(vi)}.

\vspace{2mm}
\noindent
\emph{(4)}: All of these relations follow diagrammatically and from \emph{(3)(iii)} and \emph{(3)(iv)}.

\vspace{2mm}
\noindent
\emph{(5)}: Relation \emph{(5)(i)} is simple to check since the diagrams contain no points of intersection. For \emph{(5)(ii)}, applying (H1) and (H2) we see that
\[ \varphi(\tau_{2i}e_{2i-1}\tau_{2i}) \hspace{2mm} =
\begin{tikzpicture}[baseline=8.7ex, scale=0.8]
    \node[label=above:{\scriptsize $2i-1$}] (2i-1) at (1,3){};
    \node[label=above:{\scriptsize $2i$}] (2i) at (2,3){};
    \node[label=above:{\scriptsize $2i+1$}] (2i+1) at (3,3){};
    \node[label=above:{\scriptsize $2i+2$}] (2i+2) at (4,3){};
    \draw[->] (1,2) to (3,3);
    \draw[<-] (2,2) to [curve through={(1.9,2.2) .. (1.6,2.5) .. (1.9,2.8)}] (2,3);
    \draw[->] (3,2) to (1,3);
    \draw[<-] (4,2) to (4,3);

    \draw[<-] (1,2) to [bend right] (2,2);
    \draw[->] (1,1) to [bend left] (2,1);
    \draw[->] (3,1) to (3,2);
    \draw[<-] (4,1) to (4,2);

    \draw[->] (1,0) to (3,1);
    \draw[<-] (2,0) to [curve through={(1.9,0.2) .. (1.6,0.5) .. (1.9,0.8)}] (2,1);
    \draw[->] (3,0) to (1,1);
    \draw[<-] (4,0) to (4,1);
\end{tikzpicture} = 
\begin{tikzpicture}[baseline=8.7ex, scale=0.8]
    \node[label=above:{\scriptsize $2i-1$}] (2i-1) at (1,3){};
    \node[label=above:{\scriptsize $2i$}] (2i) at (2,3){};
    \node[label=above:{\scriptsize $2i+1$}] (2i+1) at (3,3){};
    \node[label=above:{\scriptsize $2i+2$}] (2i+2) at (4,3){};
    \draw[->] (1,2) to (1,3);
    \draw[->] (2,3) to [curve through={(2.5,2.5) .. (2.65,2.3) .. (2.5,2.1) .. (2.35,2.3) .. (2.5,2.5)}] (3,3);
    \draw (4,2) to (4,3);

    \draw (1,1) to (1,2);
    \draw (4,1) to (4,2);

    \draw (1,0) to (1,1);
    \draw[<-] (2,0) to [curve through={(2.5,0.5) .. (2.65,0.7) .. (2.5,0.9) .. (2.35,0.7) .. (2.5,0.5)}] (3,0);
    \draw[<-] (4,0) to (4,1);
\end{tikzpicture}, \] 
\[ \varphi(\tau_{2i+1}e_{2i+1}\tau_{2i+1}) \hspace{2mm} =
\begin{tikzpicture}[baseline=8.7ex, scale=0.8]
    \node[label=above:{\scriptsize $2i-1$}] (2i-1) at (1,3){};
    \node[label=above:{\scriptsize $2i$}] (2i) at (2,3){};
    \node[label=above:{\scriptsize $2i+1$}] (2i+1) at (3,3){};
    \node[label=above:{\scriptsize $2i+2$}] (2i+2) at (4,3){};
    \draw[<-] (2,2) to (4,3);
    \draw[->] (3,2) to [curve through={(2.9,2.2) .. (2.6,2.5) .. (2.9,2.8)}] (3,3);
    \draw[<-] (4,2) to (2,3);
    \draw[->] (1,2) to (1,3);

    \draw[<-] (3,2) to [bend right] (4,2);
    \draw[->] (3,1) to [bend left] (4,1);
    \draw[->] (1,1) to (1,2);
    \draw[<-] (2,1) to (2,2);

    \draw[<-] (2,0) to (4,1);
    \draw[->] (3,0) to [curve through={(2.9,0.2) .. (2.6,0.5) .. (2.9,0.8)}] (3,1);
    \draw[<-] (4,0) to (2,1);
    \draw[->] (1,0) to (1,1);
\end{tikzpicture} = 
\begin{tikzpicture}[baseline=8.7ex, scale=0.8]
    \node[label=above:{\scriptsize $2i-1$}] (2i-1) at (1,3){};
    \node[label=above:{\scriptsize $2i$}] (2i) at (2,3){};
    \node[label=above:{\scriptsize $2i+1$}] (2i+1) at (3,3){};
    \node[label=above:{\scriptsize $2i+2$}] (2i+2) at (4,3){};
    \draw[->] (1,2) to (1,3);
    \draw[->] (2,3) to [curve through={(2.5,2.5) .. (2.65,2.3) .. (2.5,2.1) .. (2.35,2.3) .. (2.5,2.5)}] (3,3);
    \draw (4,2) to (4,3);

    \draw (1,1) to (1,2);
    \draw (4,1) to (4,2);

    \draw (1,0) to (1,1);
    \draw[<-] (2,0) to [curve through={(2.5,0.5) .. (2.65,0.7) .. (2.5,0.9) .. (2.35,0.7) .. (2.5,0.5)}] (3,0);
    \draw[<-] (4,0) to (4,1);
\end{tikzpicture}, \]
thus $\varphi(\tau_{2i}e_{2i-1}\tau_{2i})=\varphi(\tau_{2i+1}e_{2i+1}\tau_{2i+1})$. In a similar manner, for relation \emph{(5)(iii)} one can show,
\[ \varphi(\tau_{2i}e_{2i-2}\tau_{2i}) \hspace{2mm} = 
\begin{tikzpicture}[baseline=2ex, scale=0.8]
    \node[label=above:{\scriptsize $2i-2$}] (2i-2) at (1,1){};
    \node[label=above:{\scriptsize $2i+1$}] (2i+1) at (4,1){};
    \draw[->] (2,0) to (2,1);
    \draw[<-] (3,0) to (3,1);
    \draw[<-] (1,0) to [curve through={(1.8,0.35) .. (2.5,0.45) .. (3.2,0.35)}] (4,0);
    \draw[->] (1,1) to [curve through={(1.8,0.65) .. (2.5,0.55) .. (3.2,0.65)}] (4,1);
\end{tikzpicture} = \hspace{2mm} \varphi(\tau_{2i-1}e_{2i}\tau_{2i-1}). \]

\vspace{2mm}
\noindent
\emph{(6)}: These relations are immediately seen to be upheld diagrammatically.

\vspace{2mm}
\noindent
\emph{(7)(i)}: We seek to show $\varphi(\tau_{2i-2}\tau_{2i}\tau_{2i-2})=\varphi(\tau_{2i}\tau_{2i-2}\tau_{2i}(1-e_{2i-2}))$. The left hand side gives
\begin{align*}
\varphi(&\tau_{2i-2}\tau_{2i}\tau_{2i-2}) \hspace{1mm} = \hspace{1mm} 
\begin{tikzpicture}[baseline=9ex, scale=0.8]
    \node[label=above:{\scriptsize $2i-3$}] (2i-3) at (1,3){};
    \node[label=above:{\scriptsize $2i-2$}] (2i-2) at (2,3){};
    \node[label=above:{\scriptsize $2i-1$}] (2i-1) at (3,3){};
    \node[label=above:{\scriptsize $2i$}] (2i) at (4,3){};
    \node[label=above:{\scriptsize $2i+1$}] (2i+1) at (5,3){};
    \draw[->] (1,0) to (3,1);
    \draw[<-] (2,0) to [curve through={(1.9,0.2) .. (1.6,0.5) .. (1.9,0.8)}] (2,1);
    \draw[->] (3,0) to (1,1);
    \draw[<-] (4,0) to (4,1);
    \draw[->] (5,0) to (5,1);
    \draw[->] (1,1) to (1,2);
    \draw[<-] (2,1) to (2,2);
    \draw[->] (3,1) to (5,2);
    \draw[<-] (4,1) to [curve through={(3.9,1.2) .. (3.6,1.5) .. (3.9,1.8)}] (4,2);
    \draw[->] (5,1) to (3,2);
    \draw[->] (1,2) to (3,3);
    \draw[<-] (2,2) to [curve through={(1.9,2.2) .. (1.6,2.5) .. (1.9,2.8)}] (2,3);
    \draw[->] (3,2) to (1,3);
    \draw[<-] (4,2) to (4,3);
    \draw[->] (5,2) to (5,3);
\end{tikzpicture} \\
&= \hspace{1mm} 
\begin{tikzpicture}[baseline=2ex, scale=0.8]
    \node[label=above:{\scriptsize $2i-3$}] (2i-3) at (1,1){};
    \node[label=above:{\scriptsize $2i-2$}] (2i-2) at (2,1){};
    \node[label=above:{\scriptsize $2i-1$}] (2i-1) at (3,1){};
    \node[label=above:{\scriptsize $2i$}] (2i) at (4,1){};
    \node[label=above:{\scriptsize $2i+1$}] (2i+1) at (5,1){};
    \draw[->] (2,1) to (2,0);
    \draw[->] (3,0) to (3,1);
    \draw[->] (4,1) to (4,0);
    \draw[->] (1,0) to [curve through={(1.3,0.1) .. (4.4,0.5)}] (5,1);
    \draw[<-] (1,1) to [curve through={(1.3,0.9) .. (4.4,0.5)}] (5,0);
\end{tikzpicture} \hspace{1mm} - \hspace{1mm}
\begin{tikzpicture}[baseline=2ex, scale=0.8]
    \node[label=above:{\scriptsize $2i-3$}] (2i-3) at (1,1){};
    \node[label=above:{\scriptsize $2i-2$}] (2i-2) at (2,1){};
    \node[label=above:{\scriptsize $2i-1$}] (2i-1) at (3,1){};
    \node[label=above:{\scriptsize $2i$}] (2i) at (4,1){};
    \node[label=above:{\scriptsize $2i+1$}] (2i+1) at (5,1){};
    \draw[->] (2,1) to [curve through={(2.5,0.6)}] (3,1);
    \draw[<-] (2,0) to [curve through={(2.5,0.4)}] (3,0);
    \draw[->] (4,1) to (4,0);
    \draw[->] (1,0) to [curve through={(1.3,0.1) .. (4.4,0.5)}] (5,1);
    \draw[<-] (1,1) to [curve through={(1.3,0.9) .. (4.4,0.5)}] (5,0);
\end{tikzpicture} \\
&= \hspace{1mm} 
\begin{tikzpicture}[baseline=2ex, scale=0.8]
    \node[label=above:{\scriptsize $2i-3$}] (2i-3) at (1,1){};
    \node[label=above:{\scriptsize $2i-2$}] (2i-2) at (2,1){};
    \node[label=above:{\scriptsize $2i-1$}] (2i-1) at (3,1){};
    \node[label=above:{\scriptsize $2i$}] (2i) at (4,1){};
    \node[label=above:{\scriptsize $2i+1$}] (2i+1) at (5,1){};
    \draw[->] (2,1) to (2,0);
    \draw[->] (3,0) to (3,1);
    \draw[->] (4,1) to (4,0);
    \draw[->] (1,0) to [curve through={(1.3,0.1) .. (4.4,0.5)}] (5,1);
    \draw[<-] (1,1) to [curve through={(1.3,0.9) .. (4.4,0.5)}] (5,0);
\end{tikzpicture} \hspace{1mm} - \hspace{1mm}
\begin{tikzpicture}[baseline=2ex, scale=0.8]
    \node[label=above:{\scriptsize $2i-3$}] (2i-3) at (1,1){};
    \node[label=above:{\scriptsize $2i-2$}] (2i-2) at (2,1){};
    \node[label=above:{\scriptsize $2i-1$}] (2i-1) at (3,1){};
    \node[label=above:{\scriptsize $2i$}] (2i) at (4,1){};
    \node[label=above:{\scriptsize $2i+1$}] (2i+1) at (5,1){};
    \draw[->] (2,1) to [curve through={(2.5,0.7)}] (3,1);
    \draw[<-] (2,0) to [curve through={(2.5,0.3)}] (3,0);
    \draw[->] (4,1) to (4,0);
    \draw[->] (1,0) to [curve through={(2.5,0.4) .. (4.4,0.5)}] (5,1);
    \draw[<-] (1,1) to [curve through={(2.5,0.6) .. (4.4,0.5)}] (5,0);
\end{tikzpicture}
\end{align*}
where the second equality follows by applying (H3), and the third equality follows from (H2). By applying (H1) and (H2), one can check that the first term above is $\varphi(\tau_{2i}\tau_{2i-2}\tau_{2i})$ and the second term above is $\varphi(\tau_{2i}\tau_{2i-2}\tau_{2i}e_{2i-2})$, hence \emph{(7)(i)} holds. Relation \emph{(7)(ii)} can be shown in an analogous manner.

\vspace{2mm}
\noindent
\emph{(7)(iii)}: We seek to show that $\varphi(\tau_{2i-1}\tau_{2i}\tau_{2i-1})=\varphi(\tau_{2i}-e_{2i-2}\tau_{2i}-\tau_{2i}e_{2i-2})$. The left hand side gives
\begin{align*}
\varphi(&\tau_{2i-1}\tau_{2i}\tau_{2i-1}) \hspace{1mm} = \hspace{1mm}
\begin{tikzpicture}[baseline=9ex, scale=0.8]
    \node[label=above:{\scriptsize $2i-2$}] (2i-2) at (1,3){};
    \node[label=above:{\scriptsize $2i-1$}] (2i-1) at (2,3){};
    \node[label=above:{\scriptsize $2i$}] (2i) at (3,3){};
    \node[label=above:{\scriptsize $2i+1$}] (2i+1) at (4,3){};
    \draw[<-] (1,0) to (3,1);
    \draw[->] (2,0) to [curve through={(2.1,0.2) .. (2.4,0.5) .. (2.1,0.8)}] (2,1);
    \draw[<-] (3,0) to (1,1);
    \draw[->] (4,0) to (4,1);
    \draw[->] (2,1) to (4,2);
    \draw[<-] (3,1) to [curve through={(2.9,1.2) .. (2.6,1.5) .. (2.9,1.8)}] (3,2);
    \draw[->] (4,1) to (2,2);
    \draw[<-] (1,1) to (1,2);
    \draw[<-] (1,2) to (3,3);
    \draw[->] (2,2) to [curve through={(2.1,2.2) .. (2.4,2.5) .. (2.1,2.8)}] (2,3);
    \draw[<-] (3,2) to (1,3);
    \draw[->] (4,2) to (4,3);
\end{tikzpicture} \hspace{1mm} = \hspace{1mm} 
\begin{tikzpicture}[baseline=6ex, scale=0.8]
    \node[label=above:{\scriptsize $2i-2$}] (2i-2) at (1,2){};
    \node[label=above:{\scriptsize $2i-1$}] (2i-1) at (2,2){};
    \node[label=above:{\scriptsize $2i$}] (2i) at (3,2){};
    \node[label=above:{\scriptsize $2i+1$}] (2i+1) at (4,2){};
    \draw[->] (2,0) to [curve through={(2.5,0.4) .. (3.8,1)}] (4,2);
    \draw[->] (3,2) to [curve through={(2.7,1.85) .. (1.4,1.5) .. (1.1,1) .. (1.4,0.5) .. (2.7,0.15)}] (3,0);
    \draw[->] (1,2) to [curve through={(1.9,1.4) .. (2.5,1.15) .. (1.3,0.25)}] (1,0);
    \draw[->] (4,0) to [curve through={(3.8,0.2) .. (1.9,1.15)}] (2,2);
\end{tikzpicture} \hspace{1mm} - \hspace{1mm}
\begin{tikzpicture}[baseline=6ex, scale=0.8]
    \node[label=above:{\scriptsize $2i-2$}] (2i-2) at (1,2){};
    \node[label=above:{\scriptsize $2i-1$}] (2i-1) at (2,2){};
    \node[label=above:{\scriptsize $2i$}] (2i) at (3,2){};
    \node[label=above:{\scriptsize $2i+1$}] (2i+1) at (4,2){};
    \draw[->] (2,0) to [curve through={(1.9,0.55) .. (1.2,0.15)}] (1,0);
    \draw[->] (3,2) to [curve through={(2.7,1.85) .. (1.4,1.5) .. (1.1,1) .. (1.4,0.5) .. (2.7,0.15)}] (3,0);
    \draw[->] (1,2) to [curve through={(1.9,1.4) .. (2.5,1.15) .. (2.2,0.65) .. (2.7,0.75) .. (3.8,1.1)}] (4,2);
    \draw[->] (4,0) to [curve through={(3.8,0.2) .. (1.9,1.15)}] (2,2);
\end{tikzpicture}.
\end{align*}
By applying (H2) twice and (H1), the second term above straightens out to
\[
\begin{tikzpicture}[baseline=5ex, scale=0.8]
    \node[label=above:{\scriptsize $2i-2$}] (2i-2) at (1,2){};
    \node[label=above:{\scriptsize $2i-1$}] (2i-1) at (2,2){};
    \node[label=above:{\scriptsize $2i$}] (2i) at (3,2){};
    \node[label=above:{\scriptsize $2i+1$}] (2i+1) at (4,2){};
    \draw[->] (2,0) to [curve through={(1.9,0.55) .. (1.2,0.15)}] (1,0);
    \draw[->] (3,2) to [curve through={(2.7,1.85) .. (1.4,1.5) .. (1.1,1) .. (1.4,0.5) .. (2.7,0.15)}] (3,0);
    \draw[->] (1,2) to [curve through={(1.9,1.4) .. (2.5,1.15) .. (2.2,0.65) .. (2.7,0.75) .. (3.8,1.1)}] (4,2);
    \draw[->] (4,0) to [curve through={(3.8,0.2) .. (1.9,1.15)}] (2,2);
\end{tikzpicture} \hspace{1mm} = \hspace{1mm}
\begin{tikzpicture}[baseline=2ex, scale=0.8]
    \node[label=above:{\scriptsize $2i-2$}] (2i-2) at (1,1){};
    \node[label=above:{\scriptsize $2i-1$}] (2i-1) at (2,1){};
    \node[label=above:{\scriptsize $2i$}] (2i) at (3,1){};
    \node[label=above:{\scriptsize $2i+1$}] (2i+1) at (4,1){};
    \draw[->] (3,1) to (3,0);
    \draw[->] (2,0) to [curve through={(1.5,0.35)}] (1,0);
    \draw[->] (1,1) to [curve through={(2.5,0.65)}] (4,1);
    \draw[->] (4,0) to [curve through={(3.5,0.4) .. (2.5,0.45)}] (2,1);
\end{tikzpicture} \hspace{1mm} = \hspace{1mm} \varphi(\tau_{2i}e_{2i-2}).
\] 
For the first term we get
\begin{align*}
\begin{tikzpicture}[baseline=6ex, scale=0.8]
    \node[label=above:{\scriptsize $2i-2$}] (2i-2) at (1,2){};
    \node[label=above:{\scriptsize $2i-1$}] (2i-1) at (2,2){};
    \node[label=above:{\scriptsize $2i$}] (2i) at (3,2){};
    \node[label=above:{\scriptsize $2i+1$}] (2i+1) at (4,2){};
    \draw[->] (2,0) to [curve through={(2.5,0.4) .. (3.8,1)}] (4,2);
    \draw[->] (3,2) to [curve through={(2.7,1.85) .. (1.4,1.5) .. (1.1,1) .. (1.4,0.5) .. (2.7,0.15)}] (3,0);
    \draw[->] (1,2) to [curve through={(1.9,1.4) .. (2.5,1.15) .. (1.3,0.25)}] (1,0);
    \draw[->] (4,0) to [curve through={(3.8,0.2) .. (1.9,1.15)}] (2,2);
\end{tikzpicture} \hspace{1mm} &= \hspace{1mm} 
\begin{tikzpicture}[baseline=6ex, scale=0.8]
    \node[label=above:{\scriptsize $2i-2$}] (2i-2) at (1,2){};
    \node[label=above:{\scriptsize $2i-1$}] (2i-1) at (2,2){};
    \node[label=above:{\scriptsize $2i$}] (2i) at (3,2){};
    \node[label=above:{\scriptsize $2i+1$}] (2i+1) at (4,2){};
    \draw[->] (2,0) to [curve through={(2.5,0.4) .. (3.8,1)}] (4,2);
    \draw[->] (3,2) to [curve through={(2.7,1.85) .. (1.4,1.5) .. (1.1,1) .. (1.4,0.5) .. (2.7,0.15)}] (3,0);
    \draw[->] (1,2) to [curve through={(1.4,1.4) .. (1.6,1.1) .. (1.3,0.25)}] (1,0);
    \draw[->] (4,0) to [curve through={(3.8,0.2) .. (2,1.15)}] (2,2);
\end{tikzpicture} \hspace{1mm} - \hspace{1mm}
\begin{tikzpicture}[baseline=6ex, scale=0.8]
    \node[label=above:{\scriptsize $2i-2$}] (2i-2) at (1,2){};
    \node[label=above:{\scriptsize $2i-1$}] (2i-1) at (2,2){};
    \node[label=above:{\scriptsize $2i$}] (2i) at (3,2){};
    \node[label=above:{\scriptsize $2i+1$}] (2i+1) at (4,2){};
    \draw[->] (2,0) to [curve through={(2.5,0.4) .. (3.8,1)}] (4,2);
    \draw[->] (3,2) to [curve through={(2.7,1.85) .. (1.4,1.5) .. (1.1,1) .. (1.4,0.5) .. (2.7,0.15)}] (3,0);
    \draw[->] (1,2) to [curve through={(1.2,1.8) .. (1.85,1.3)}] (2,2);
    \draw[->] (4,0) to [curve through={(3.8,0.2) .. (2.1,0.85)}] (1,0);
\end{tikzpicture} \\
&= \hspace{1mm}
\begin{tikzpicture}[baseline=2ex, scale=0.8]
    \node[label=above:{\scriptsize $2i-2$}] (2i-2) at (1,1){};
    \node[label=above:{\scriptsize $2i-1$}] (2i-1) at (2,1){};
    \node[label=above:{\scriptsize $2i$}] (2i) at (3,1){};
    \node[label=above:{\scriptsize $2i+1$}] (2i+1) at (4,1){};
    \draw[->] (2,0) to (4,1);
    \draw[<-] (3,0) to [curve through={(2.9,0.2) .. (2.6,0.5) .. (2.9,0.8)}] (3,1);
    \draw[->] (4,0) to (2,1);
    \draw[<-] (1,0) to (1,1);
\end{tikzpicture} \hspace{1mm} - \hspace{1mm}
 \begin{tikzpicture}[baseline=2ex, scale=0.8]
    \node[label=above:{\scriptsize $2i-2$}] (2i-2) at (1,1){};
    \node[label=above:{\scriptsize $2i-1$}] (2i-1) at (2,1){};
    \node[label=above:{\scriptsize $2i$}] (2i) at (3,1){};
    \node[label=above:{\scriptsize $2i+1$}] (2i+1) at (4,1){};
    \draw[<-] (3,0) to (3,1);
    \draw[<-] (2,1) to [curve through={(1.5,0.65)}] (1,1);
    \draw[<-] (1,0) to [curve through={(2.5,0.35)}] (4,0);
    \draw[<-] (4,1) to [curve through={(3.5,0.6) .. (2.5,0.5)}] (2,0);
\end{tikzpicture} \hspace{1mm} = \hspace{1mm} \varphi(\tau_{2i}) + \varphi(e_{2i-2}\tau_{2i})
\end{align*}
where the first equality follows by applying (H3), and the second equality by (H1) and (H2). Therefore collectively we have show \emph{(7)(iii)}. Relation \emph{(7)(iv)} follows in an analogous manner.

\vspace{2mm}
\noindent
\emph{(8)(i)}: We seek to show that 
\begin{equation} \label{(Eq10)}
\varphi(x_{2i+1}) = \varphi(\tau_{2i}x_{2i-1}\tau_{2i})+\varphi(e_{2i}e_{i-1}\tau_{2i})+\varphi(\tau_{2i}e_{2i-1}e_{2i})-\varphi(\tau_{2i}).
\end{equation}
One can check that
\[ \varphi(e_{2i}e_{2i-1}\tau_{2i}) = 
\begin{tikzpicture}[baseline=2ex, scale=0.8]
    \node[label=above:{\scriptsize $2i-1$}] (2i-1) at (1,1){};
    \node[label=above:{\scriptsize $2i$}] (2i) at (2,1){};
    \node[label=above:{\scriptsize $2i+1$}] (2i+1) at (3,1){};
    \draw[->] (1,0) to (1,1);
    \draw[<-] (2,0) to [curve through={(2.5,0.2)}] (3,0);
    \draw[->] (2,1) to [curve through={(2.5,0.8)}] (3,1);
    \filldraw (2.5,0.2) circle (2pt);
\end{tikzpicture}, \hspace{5mm}
\varphi(\tau_{2i}e_{2i-1}e_{2i}) = 
\begin{tikzpicture}[baseline=2ex, scale=0.8]
    \node[label=above:{\scriptsize $2i-1$}] (2i-1) at (1,1){};
    \node[label=above:{\scriptsize $2i$}] (2i) at (2,1){};
    \node[label=above:{\scriptsize $2i+1$}] (2i+1) at (3,1){};
    \draw[->] (1,0) to (1,1);
    \draw[<-] (2,0) to [curve through={(2.5,0.2)}] (3,0);
    \draw[->] (2,1) to [curve through={(2.5,0.8)}] (3,1);
    \filldraw (2.5,0.8) circle (2pt);
\end{tikzpicture}. \]
By (H1), (H2), (H3), and applying \emph{\Cref{DecRels}} (and a $90^{\circ}$ clockwise rotation of  \emph{\Cref{DecRels}}), we have
\begin{align*}
\varphi(\tau_{2i}x_{2i-1}\tau_{2i}) \hspace{2mm} &= 
\begin{tikzpicture}[baseline=7.5ex, scale=0.8]
    \node[label=above:{\scriptsize $2i-1$}] (2i-1) at (1,3){};
    \node[label=above:{\scriptsize $2i$}] (2i) at (2,3){};
    \node[label=above:{\scriptsize $2i+1$}] (2i+1) at (3,3){};
    \draw[->] (1,2) to (3,3);
    \draw[<-] (2,2) to [curve through={(1.9,2.2) .. (1.6,2.5) .. (1.9,2.8)}] (2,3);
    \draw[->] (3,2) to (1,3);
    \draw[->] (1,1) to (1,2);
    \draw[<-] (2,1) to (2,2);
    \draw[->] (3,1) to (3,2);
    \draw[->] (1,0) to (3,1);
    \draw[<-] (2,0) to [curve through={(1.9,0.2) .. (1.6,0.5) .. (1.9,0.8)}] (2,1);
    \draw[->] (3,0) to (1,1);
    \filldraw (1,1.5) circle (2pt);
\end{tikzpicture} = \hspace{2mm}
\begin{tikzpicture}[baseline=3ex, scale=0.8]
    \node[label=above:{\scriptsize $2i-1$}] (2i-1) at (1,1){};
    \node[label=above:{\scriptsize $2i$}] (2i) at (2,1){};
    \node[label=above:{\scriptsize $2i+1$}] (2i+1) at (3,1){};
    \draw[->] (1,0) to [curve through={(1.3,0.25) .. (1.4,0.5) .. (1.3,0.75)}] (1,1);
    \draw[<-] (2,0) to [curve through={(2.3,0.25) .. (2.4,0.5) .. (2.3,0.75)}] (2,1);
    \draw[->] (3,0) to [curve through={(2.3,0.25) .. (1.3,0.3) .. (1.1,0.5) .. (1.3,0.7) .. (2.3,0.8)}] (3,1);
    \filldraw (1.1,0.5) circle (1.5pt);
\end{tikzpicture} = 
\begin{tikzpicture}[baseline=3ex, scale=0.8]
    \node[label=above:{\scriptsize $2i-1$}] (2i-1) at (1,1){};
    \node[label=above:{\scriptsize $2i$}] (2i) at (2,1){};
    \node[label=above:{\scriptsize $2i+1$}] (2i+1) at (3,1){};
    \draw[->] (1,0) to [curve through={(1.3,0.25) .. (1.4,0.5) .. (1.3,0.75)}] (1,1);
    \draw[<-] (2,0) to [curve through={(2.3,0.25) .. (2.4,0.5) .. (2.3,0.75)}] (2,1);
    \draw[->] (3,0) to [curve through={(2.3,0.25) .. (1.3,0.3) .. (1.1,0.5) .. (1.3,0.7) .. (2.3,0.8)}] (3,1);
    \filldraw (1.55,0.75) circle (1.5pt);
\end{tikzpicture} \hspace{2mm} + \hspace{2mm}
\begin{tikzpicture}[baseline=3ex, scale=0.8]
    \node[label=above:{\scriptsize $2i-1$}] (2i-1) at (1,1){};
    \node[label=above:{\scriptsize $2i$}] (2i) at (2,1){};
    \node[label=above:{\scriptsize $2i+1$}] (2i+1) at (3,1){};
    \draw[->] (1,0) to [curve through={(1.3,0.25) .. (1.32,0.67) .. (1.9,0.75) .. (2.3,0.75)}] (3,1);
    \draw[<-] (2,0) to [curve through={(2.3,0.25) .. (2.4,0.5) .. (2.3,0.75)}] (2,1);
    \draw[->] (3,0) to [curve through={(2.3,0.25) .. (1.3,0.3) .. (1.1,0.5) .. (1.27,0.75) .. (1.1,0.9)}] (1,1);
\end{tikzpicture} \\
&= 
\begin{tikzpicture}[baseline=3ex, scale=0.8]
    \node[label=above:{\scriptsize $2i-1$}] (2i-1) at (1,1){};
    \node[label=above:{\scriptsize $2i$}] (2i) at (2,1){};
    \node[label=above:{\scriptsize $2i+1$}] (2i+1) at (3,1){};
    \draw[->] (1,0) to (1,1);
    \draw[<-] (2,0) to [curve through={(2.3,0.25) .. (2.4,0.5) .. (2.3,0.75)}] (2,1);
    \draw[->] (3,0) to [curve through={(2.3,0.25) .. (1.9,0.5) .. (2.3,0.8)}] (3,1);
    \filldraw (1.9,0.5) circle (1.5pt);
\end{tikzpicture} \hspace{2mm} + \hspace{2mm}
\begin{tikzpicture}[baseline=3ex, scale=0.8]
    \node[label=above:{\scriptsize $2i-1$}] (2i-1) at (1,1){};
    \node[label=above:{\scriptsize $2i$}] (2i) at (2,1){};
    \node[label=above:{\scriptsize $2i+1$}] (2i+1) at (3,1){};
    \draw[->] (1,0) to (3,1);
    \draw[<-] (2,0) to [curve through={(1.9,0.2) .. (1.6,0.5) .. (1.9,0.8)}] (2,1);
    \draw[->] (3,0) to (1,1);
\end{tikzpicture} \\
&= 
\begin{tikzpicture}[baseline=3ex, scale=0.8]
    \node[label=above:{\scriptsize $2i-1$}] (2i-1) at (1,1){};
    \node[label=above:{\scriptsize $2i$}] (2i) at (2,1){};
    \node[label=above:{\scriptsize $2i+1$}] (2i+1) at (3,1){};
    \draw[->] (1,0) to (1,1);
    \draw[<-] (2,0) to [curve through={(2.3,0.25) .. (2.4,0.5) .. (2.3,0.75)}] (2,1);
    \draw[->] (3,0) to [curve through={(2.3,0.25) .. (1.9,0.5) .. (2.3,0.8)}] (3,1);
    \filldraw (2.5,0.83) circle (1.5pt);
\end{tikzpicture} \hspace{2mm} - \hspace{2mm}
\begin{tikzpicture}[baseline=3ex, scale=0.8]
    \node[label=above:{\scriptsize $2i-1$}] (2i-1) at (1,1){};
    \node[label=above:{\scriptsize $2i$}] (2i) at (2,1){};
    \node[label=above:{\scriptsize $2i+1$}] (2i+1) at (3,1){};
    \draw[->] (1,0) to (1,1);
    \draw[<-] (2,0) to [curve through={(2.3,0.25) .. (2.4,0.5) .. (2.3,0.7) .. (1.9,0.5) .. (2.3,0.25)}] (3,0);
    \draw[->] (2,1) to [curve through={(2.3,0.8)}] (3,1);
\end{tikzpicture} \hspace{2mm} + \hspace{2mm}
\begin{tikzpicture}[baseline=3ex, scale=0.8]
    \node[label=above:{\scriptsize $2i-1$}] (2i-1) at (1,1){};
    \node[label=above:{\scriptsize $2i$}] (2i) at (2,1){};
    \node[label=above:{\scriptsize $2i+1$}] (2i+1) at (3,1){};
    \draw[->] (1,0) to (3,1);
    \draw[<-] (2,0) to [curve through={(1.9,0.2) .. (1.6,0.5) .. (1.9,0.8)}] (2,1);
    \draw[->] (3,0) to (1,1);
\end{tikzpicture} \\
&= 
\begin{tikzpicture}[baseline=3ex, scale=0.8]
    \node[label=above:{\scriptsize $2i-1$}] (2i-1) at (1,1){};
    \node[label=above:{\scriptsize $2i$}] (2i) at (2,1){};
    \node[label=above:{\scriptsize $2i+1$}] (2i+1) at (3,1){};
    \draw[->] (1,0) to (1,1);
    \draw[<-] (2,0) to (2,1);
    \draw[->] (3,0) to (3,1);
    \filldraw (3,0.5) circle (1.5pt);
\end{tikzpicture} \hspace{2mm} - \hspace{2mm}
\begin{tikzpicture}[baseline=3ex, scale=0.8]
    \node[label=above:{\scriptsize $2i-1$}] (2i-1) at (1,1){};
    \node[label=above:{\scriptsize $2i$}] (2i) at (2,1){};
    \node[label=above:{\scriptsize $2i+1$}] (2i+1) at (3,1){};
    \draw[->] (1,0) to (1,1);
    \draw[<-] (2,0) to [bend left] (3,0);
    \draw[->] (2,1) to [bend right] (3,1);
    \filldraw (2.5,0.85) circle (1.5pt);
\end{tikzpicture} \hspace{2mm} - \hspace{2mm}
\begin{tikzpicture}[baseline=3ex, scale=0.8]
    \node[label=above:{\scriptsize $2i-1$}] (2i-1) at (1,1){};
    \node[label=above:{\scriptsize $2i$}] (2i) at (2,1){};
    \node[label=above:{\scriptsize $2i+1$}] (2i+1) at (3,1){};
    \draw[->] (1,0) to (1,1);
    \draw[<-] (2,0) to [bend left] (3,0);
    \draw[->] (2,1) to [bend right] (3,1);
    \filldraw (2.5,0.15) circle (1.5pt);
\end{tikzpicture} \hspace{2mm} + \hspace{2mm}
\begin{tikzpicture}[baseline=3ex, scale=0.8]
    \node[label=above:{\scriptsize $2i-1$}] (2i-1) at (1,1){};
    \node[label=above:{\scriptsize $2i$}] (2i) at (2,1){};
    \node[label=above:{\scriptsize $2i+1$}] (2i+1) at (3,1){};
    \draw[->] (1,0) to (3,1);
    \draw[<-] (2,0) to [curve through={(1.9,0.2) .. (1.6,0.5) .. (1.9,0.8)}] (2,1);
    \draw[->] (3,0) to (1,1);
\end{tikzpicture} \\
\\
&= \varphi(x_{2i+1})-\varphi(\tau_{2i}e_{2i-1}e_{2i})-\varphi(e_{2i}e_{2i-1}\tau_{2i})+\varphi(\tau_{2i}).
\end{align*}
Rearranging yields \emph{\Cref{(Eq10)}}. The remaining Skein-like relations \emph{(8)(ii)}, \emph{(8)(iii)}, and \emph{(8)(iv)}, following in a similar manner where we employ \emph{\Cref{DecRels}} to pull the decoration over various oriented crossings.

\vspace{2mm}
\noindent
\emph{(9)} and \emph{(10)}: These relations are immediately seen to be upheld diagrammatically.

\end{proof}

The remainder of this section seeks to show that the algebra homomorphism $\varphi$ in the above proposition is surjective. Firstly, from \emph{\Cref{BasisThm}} we know that $\mathsf{End}_{\mathsf{Heis}}((\uparrow\downarrow)^{k})$ has a basis given by
\[ c_w^{k_w} \ldots c_1^{k_1} c_0^{k_0} r_1^{s_1} r_3^{s_3} \ldots r_{2k-1}^{s_{2k-1} } \alpha r_2^{t_2} r_4^{t_4} \ldots r_{2k}^{t_{2k}} \]
where $\alpha\in\mathsf{Sim}((\uparrow\downarrow)^{k})$. Since $\varphi(z_{l})=c_{l}$ and $\varphi(x_{i})=r_{i}$, to prove that $\varphi$ is surjective it is enough to show that $\mathsf{Sim}((\uparrow\downarrow)^{k})\subset\mathsf{Im}(\varphi)$. We will prove that $\mathsf{Sim}((\uparrow\downarrow)^{k})\subset\langle \varphi(e_{i}), \varphi(\tau_{j})\rangle_{i,j} \subset \mathsf{Im}(\varphi)$. We say that a simple diagram is \emph{planar} if no intersections occur among its strings, for example the diagrams $\varphi(e_{i})$ are all planar for each $i\in[2k-1]$. The total number of planar diagrams in $\mathsf{Sim}((\uparrow\downarrow)^{k})$ is $C_{2k}$, the $2k$-th Catalan number. These diagrams are precisely oriented versions of the Temperley-Lieb diagrams. The Jones normal form gives a way of writing the Temperley-Lieb diagrams as a product of generators  (see \cite{J83}, and also \cite[Theorem 4.3 and Figure 16]{Kau90}) which does not involve bubbles, and so may be applied here for the elements $\varphi(e_{i})$ to show that any planar diagram belongs to $\langle \varphi(e_{i}) \rangle_{i}$ and hence to $\mathsf{Im}(\varphi)$.

\begin{defn} \label{UpDownPerms}
Let $\pi\in S(k)$. Then we define the following simple $(\uparrow\downarrow)^{k}$-diagrams:
\begin{itemize}
\item[(i)] $\pi^{\uparrow}$ by pairings of endpoints $\{(2i-1,0), (2\pi(i)-1,1)\}$ and $\{(2i,0),(2i,1)\}$ for each $1\leq i \leq k$.
\item[(ii)] $\pi^{\downarrow}$ by pairings of endpoints $\{(2i-1,0), (2i-1,1)\}$ and $\{(2\pi(i),0),(2i,1)\}$ for each $1\leq i \leq k$.
\end{itemize}
\end{defn}

\begin{ex}\label{permutation}
For $k=3$ and $\pi=(1,2,3)\in S(3)$, we have
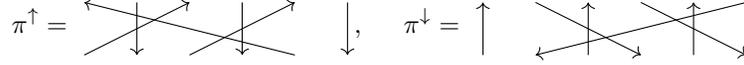
\[ \pi^{\uparrow} = \
\begin{tikzpicture}[baseline=2ex, scale=0.7]
   \draw[->] (1,0) to (3,1);
   \draw[<-] (2,0) to (2,1);
   \draw[->] (3,0) to (5,1);
   \draw[<-] (4,0) to (4,1);
   \draw[->] (5,0) to (1,1);
   \draw[<-] (6,0) to (6,1);
\end{tikzpicture}, \hspace{5mm} 
\pi^{\downarrow} = \
\begin{tikzpicture}[baseline=2ex, scale=0.7]
   \draw[->] (1,0) to (1,1);
   \draw[<-] (2,0) to (6,1);
   \draw[->] (3,0) to (3,1);
   \draw[<-] (4,0) to (2,1);
   \draw[->] (5,0) to (5,1);
   \draw[<-] (6,0) to (4,1);
\end{tikzpicture}  \]
\end{ex}

\noindent
For any $\pi\in S(k)$, it is shown in \cite{Stem97} that we have a reduced expression of the form
\[ \pi = (s_{m_{1}}s_{m_{1}+1}\cdots s_{n_{1}})(s_{m_{2}}s_{m_{2}+1}\cdots s_{n_{2}})\cdots (s_{m_{l}}s_{m_{l}+1}\cdots s_{n_{l}}), \]
where $k>n_{1}>n_{2}>\cdots >n_{l}$ and $n_{i}\geq m_{i}$. Noting that $s_{i}^{\downarrow}=\varphi(\tau_{2i+1})$, consider
\[ \alpha^{\downarrow}(w):=(s_{m_{1}}^{\downarrow}s_{m_{1}+1}^{\downarrow}\cdots s_{n_{1}}^{\downarrow})(s_{m_{2}}^{\downarrow}s_{m_{2}+1}^{\downarrow}\cdots s_{n_{2}}^{\downarrow})\cdots (s_{m_{l}}^{\downarrow}s_{m_{l}+1}^{\downarrow}\cdots s_{n_{l}}^{\downarrow})\in\mathsf{Im}(\varphi). \]
Strings in $\alpha^{\downarrow}(w)$ may intersect one another more than once, but we can resolve such double crossings by pulling strings apart via the local relations. The  descending condition on the indices in this reduced expression means we will never need to employ (H3) to pull strings apart, and thus we must have that $\alpha^{\downarrow}(w)=\pi^{\downarrow}$. Hence $\pi^{\downarrow}\in\mathsf{Im}(\varphi)$. Rotating $\pi^{\downarrow}$ by $180^{\circ}$ yields $(\rho\pi\rho^{-1})^{\uparrow}$  where $\rho$ is the product of transposition $(i,k-i+1)$ for each $i\in[k]$. Thus we also have that $\pi^{\uparrow}\in\mathsf{Im}(\varphi)$ for all $\pi\in S(k)$.

To aid upcoming proofs we define a collection of diagrams which loosen the conditions on simple diagrams.

\begin{defn}
We call an $(\bm{a},\bm{b})$-diagram \emph{semisimple} if the following hold:
\begin{itemize}
\item[(1)] It contains no loops or self intersections.
\item[(2)] No top arc intersects a bottom arc.
\end{itemize}
Let $\mathsf{SSim}(\bm{a},\bm{b})$ denote the set of semisimple $(\bm{a},\bm{b})$-diagrams, and write $\mathsf{SSim}(\bm{a})$ for $\mathsf{SSim}(\bm{a},\bm{a})$.
\end{defn} 

From definitions we have that $\mathsf{Sim}(\bm{a},\bm{b})\subset \mathsf{SSim}(\bm{a},\bm{b})$. Any diagram $\alpha\in\mathsf{SSim}((\uparrow\downarrow)^{k},(\uparrow\downarrow)^{l})$ contains precisely $k+l$ strings, and the endpoints of these strings induce an $((\uparrow\downarrow)^{k},(\uparrow\downarrow)^{l})$-matching of the endpoints $E$. We let $\overline{\alpha}$ denote the unique simple diagram corresponding to such a matching (recalling the discussion after \emph{\Cref{SimDia}}). 

\begin{lem} \label{SimBarDecomp}
Given any simple diagram $\alpha\in \mathsf{Sim}((\uparrow\downarrow)^{k},(\uparrow\downarrow)^{l})$, there exists $\pi\in S(k)$, $\sigma\in S(l)$, and a planar diagram $\beta \in \mathsf{Sim}((\uparrow\downarrow)^{k},(\uparrow\downarrow)^{l})$ such that $\pi^{\uparrow}\beta\sigma^{\downarrow}$ is semisimple and
\[ \alpha = \overline{\pi^{\uparrow}\beta\sigma^{\downarrow}}. \]
\end{lem}

\begin{proof}
Given any simple diagram $\gamma\in\mathsf{Sim}((\uparrow\downarrow)^{k},(\uparrow\downarrow)^{l})$ let $(2i,0),(2j,0)\in E$ (respectively $(2i-1,1), (2i-j,1)$) be two $\downarrow$ (respectively $\uparrow$) endpoints in the bottom row (respectively top row) of $\gamma$. Let $\gamma'$ be the simple diagram obtained from $\gamma$ by permuting these two endpoints around. It can be seen that $\gamma (i,j)^{\downarrow}$ (respectively $(i,j)^{\uparrow}\gamma$) is semisimple as long as the permutation doesn't swap the orianetation of an arc around, since that is the only way a self intersection can occur. In this situation, one can see that
\[ \gamma' = \overline{\gamma (i,j)^{\downarrow}} \hspace{5mm} \Big(\text{respectively} =\overline{(i,j)^{\uparrow}\gamma}\Big). \]
Hence to prove this lemma it is enough to show that we can reach a planar diagram $\beta$ from $\alpha$ by repeatably permuting the endpoints in the bottom row coloured by $\downarrow$, and top row coloured by $\uparrow$, in such a way that the orientations of arcs are preserved. We focus on the bottom row, where the top row will follow in the same manner by a $180^{\circ}$ rotation of the diagrammatics. Starting with $\alpha$ we remove intersections one at a time by employing a suitable permutation of endpoints. There are a few cases to consider, and in each such case the endpoints of the strings in the following diagrams will be arbitrary:

\vspace{2mm}
\noindent
(Case 1): Crossing of two down strings:
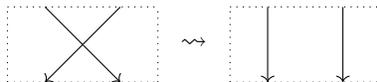
\[ 
\begin{tikzpicture}[scale=1, baseline=3ex]
\begin{scope}[dotted]
          \draw (1,0) to (1,1);
          \draw (1,0) to (3,0);
          \draw (3,0) to (3,1);
          \draw (1,1) to (3,1);
\end{scope}
    \draw[->] (1.5,1) to (2.5,0);
    \draw[->] (2.5,1) to (1.5,0);
\end{tikzpicture} \hspace{2mm} \rightsquigarrow \hspace{2mm}
\begin{tikzpicture}[scale=1, baseline=3ex]
\begin{scope}[dotted]
          \draw (1,0) to (1,1);
          \draw (1,0) to (3,0);
          \draw (3,0) to (3,1);
          \draw (1,1) to (3,1);
\end{scope}
    \draw[->] (1.5,1) to (1.5,0);
    \draw[->] (2.5,1) to (2.5,0);
\end{tikzpicture}
\]

\vspace{2mm}
\noindent
(Case 2): Crossing of a down string with an clockwise/anti-clockwise arc:
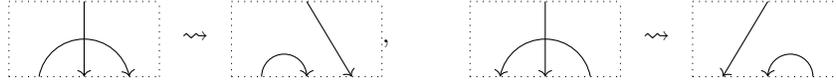
\[ 
\begin{tikzpicture}[scale=1, baseline=3ex]
\begin{scope}[dotted]
          \draw (1,0) to (1,1);
          \draw (1,0) to (3,0);
          \draw (3,0) to (3,1);
          \draw (1,1) to (3,1);
\end{scope}
    \draw[->] (1.4,0) to [curve through={(2,0.5)}] (2.6,0);
    \draw[->] (2,1) to (2,0);
\end{tikzpicture} \hspace{2mm} \rightsquigarrow \hspace{2mm}
\begin{tikzpicture}[scale=1, baseline=3ex]
\begin{scope}[dotted]
          \draw (1,0) to (1,1);
          \draw (1,0) to (3,0);
          \draw (3,0) to (3,1);
          \draw (1,1) to (3,1);
\end{scope}
    \draw[->] (1.4,0) to [curve through={(1.7,0.3)}] (2,0);
    \draw[->] (2,1) to (2.6,0);
\end{tikzpicture}, \hspace{10mm}
\begin{tikzpicture}[scale=1, baseline=3ex]
\begin{scope}[dotted]
          \draw (1,0) to (1,1);
          \draw (1,0) to (3,0);
          \draw (3,0) to (3,1);
          \draw (1,1) to (3,1);
\end{scope}
    \draw[<-] (1.4,0) to [curve through={(2,0.5)}] (2.6,0);
    \draw[->] (2,1) to (2,0);
\end{tikzpicture} \hspace{2mm} \rightsquigarrow \hspace{2mm}
\begin{tikzpicture}[scale=1, baseline=3ex]
\begin{scope}[dotted]
          \draw (1,0) to (1,1);
          \draw (1,0) to (3,0);
          \draw (3,0) to (3,1);
          \draw (1,1) to (3,1);
\end{scope}
    \draw[<-] (2,0) to [curve through={(2.3,0.3)}] (2.6,0);
    \draw[->] (2,1) to (1.4,0);
\end{tikzpicture}
\]
Note in either situation the orientation of the arc is preserved by the permutation.

\vspace{2mm}
\noindent
(Case 3): Crossing of two arcs: There are four cases based on the orientations of the two arcs given by
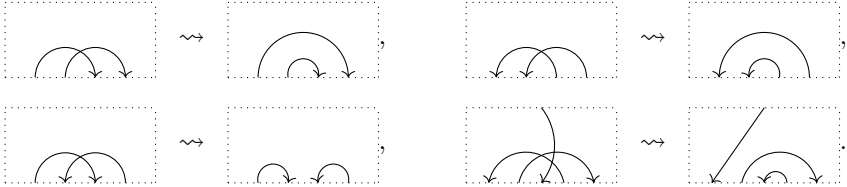
\[ 
\begin{tikzpicture}[scale=1, baseline=3ex]
\begin{scope}[dotted]
          \draw (1,0) to (1,1);
          \draw (1,0) to (3,0);
          \draw (3,0) to (3,1);
          \draw (1,1) to (3,1);
\end{scope}
    \draw[->] (1.4,0) to [curve through={(1.8,0.4)}] (2.2,0);
    \draw[->] (1.8,0) to [curve through={(2.2,0.4)}] (2.6,0);
\end{tikzpicture} \hspace{2mm} \rightsquigarrow \hspace{2mm}
\begin{tikzpicture}[scale=1, baseline=3ex]
\begin{scope}[dotted]
          \draw (1,0) to (1,1);
          \draw (1,0) to (3,0);
          \draw (3,0) to (3,1);
          \draw (1,1) to (3,1);
\end{scope}
    \draw[->] (1.4,0) to [curve through={(2,0.6)}] (2.6,0);
    \draw[->] (1.8,0) to [curve through={(2.0,0.25)}] (2.2,0);
\end{tikzpicture}, \hspace{10mm}
\begin{tikzpicture}[scale=1, baseline=3ex]
\begin{scope}[dotted]
          \draw (1,0) to (1,1);
          \draw (1,0) to (3,0);
          \draw (3,0) to (3,1);
          \draw (1,1) to (3,1);
\end{scope}
    \draw[<-] (1.4,0) to [curve through={(1.8,0.4)}] (2.2,0);
    \draw[<-] (1.8,0) to [curve through={(2.2,0.4)}] (2.6,0);
\end{tikzpicture} \hspace{2mm} \rightsquigarrow \hspace{2mm}
\begin{tikzpicture}[scale=1, baseline=3ex]
\begin{scope}[dotted]
          \draw (1,0) to (1,1);
          \draw (1,0) to (3,0);
          \draw (3,0) to (3,1);
          \draw (1,1) to (3,1);
\end{scope}
    \draw[<-] (1.4,0) to [curve through={(2,0.6)}] (2.6,0);
    \draw[<-] (1.8,0) to [curve through={(2.0,0.25)}] (2.2,0);
\end{tikzpicture},
\]
\[ 
\begin{tikzpicture}[scale=1, baseline=3ex]
\begin{scope}[dotted]
          \draw (1,0) to (1,1);
          \draw (1,0) to (3,0);
          \draw (3,0) to (3,1);
          \draw (1,1) to (3,1);
\end{scope}
    \draw[->] (1.4,0) to [curve through={(1.8,0.4)}] (2.2,0);
    \draw[<-] (1.8,0) to [curve through={(2.2,0.4)}] (2.6,0);
\end{tikzpicture} \hspace{2mm} \rightsquigarrow \hspace{2mm}
\begin{tikzpicture}[scale=1, baseline=3ex]
\begin{scope}[dotted]
          \draw (1,0) to (1,1);
          \draw (1,0) to (3,0);
          \draw (3,0) to (3,1);
          \draw (1,1) to (3,1);
\end{scope}
    \draw[->] (1.4,0) to [curve through={(1.6,0.25)}] (1.8,0);
    \draw[<-] (2.2,0) to [curve through={(2.4,0.25)}] (2.6,0);
\end{tikzpicture}, \hspace{10mm}
\begin{tikzpicture}[scale=1, baseline=3ex]
\begin{scope}[dotted]
          \draw (1,0) to (1,1);
          \draw (1,0) to (3,0);
          \draw (3,0) to (3,1);
          \draw (1,1) to (3,1);
\end{scope}
    \draw[<-] (1.3,0) to [curve through={(1.7,0.4)}] (2.3,0);
    \draw[->] (1.7,0) to [curve through={(2.3,0.4)}] (2.7,0);
    \draw[<-] (2,0) to [curve through={(2.15,0.3)}] (2,1);
\end{tikzpicture} \hspace{2mm} \rightsquigarrow \hspace{2mm}
\begin{tikzpicture}[scale=1, baseline=3ex]
\begin{scope}[dotted]
          \draw (1,0) to (1,1);
          \draw (1,0) to (3,0);
          \draw (3,0) to (3,1);
          \draw (1,1) to (3,1);
\end{scope}
    \draw[<-] (2,0) to [curve through={(2.15,0.15)}] (2.3,0);
    \draw[->] (1.7,0) to [curve through={(2.3,0.4)}] (2.7,0);
    \draw[<-] (1.3,0) to (2,1);
\end{tikzpicture}.
\]
noting that in the last case such a down string must exist. Again, the orientations of the arcs are preserved under the permutation of endpoints in all four of the above situations.

In all three of the cases above, it can be seen that the new simple diagram we obtain after the permutation of endpoints has strictly less number of intersections. We claim that applying the moves above on the bottom row, and their $180^{\circ}$ counterparts on the top row, until all such intersections are removed will yield a planar diagram. For contradiction, suppose this is not the case. Thus even after removing all such types of intersections, the diagram still contains some other type of intersection. The other such intersections are either between an up string and arc on the bottom row, a down string and arc on the top row, or an up string and a down string. The former two are $180^{\circ}$ counterparts to one another, hence we only need to consider one such type. Firstly, if an up string intersects a clockwise arc on the bottom we have 
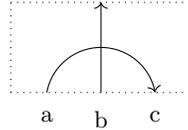
\[
\begin{tikzpicture}[scale=1.2]
\begin{scope}[dotted]
          \draw (1,0) to (1,1);
          \draw (1,0) to (3,0);
          \draw (3,0) to (3,1);
          \draw (1,1) to (3,1);
\end{scope}
    \draw[->] (1.4,0) to [curve through={(2,0.5)}] (2.6,0);
    \draw[<-] (2,1) to (2,0);
    \node[label=below:{\small a}] (a) at (1.4,0){};
    \node[label=below:{\small b}] (b) at (2,0){};
    \node[label=below:{\small c}] (c) at (2.6,0){};
\end{tikzpicture}
\]
Note that the parity of the number of endpoints on the bottom row strictly between $a$ and $b$ must be different to the partity of endpoints strictly between $b$ and $c$. Thus on can deduce that such an endpoint must be a target to a string which intersects the arc, and such an intersection would be accounted for by Case 2 or 3, hence a contradiction. The same argument can be used to show that the case of an up string intersecting an anti-clockwise arc on the bottom is also impossible. Note all intersections involving arcs have now been accounted for. Lastly assume an up string intersects a down string. We have two cases, one of which is
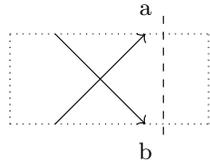
\[ 
\begin{tikzpicture}[scale=1.2, baseline=4ex]
\begin{scope}[dotted]
          \draw (1,0) to (1,1);
          \draw (1,0) to (3.2,0);
          \draw (3.2,0) to (3.2,1);
          \draw (1,1) to (3.2,1);
\end{scope}
\begin{scope}[dashed]
         \draw (2.7,1.2) to (2.7, -0.2);
\end{scope}
    \draw[->] (1.5,1) to (2.5,0);
    \draw[<-] (2.5,1) to (1.5,0);
    \node[label=above:{\small a}] (a) at (2.5,1){};
    \node[label=below:{\small b}] (b) at (2.5,0){};
\end{tikzpicture}. \]
The dashed vertical line is simply an aid for arguments to come, and has been draw so that the endpoints $a$ and $b$ are the closest endpoints to its left. The other case is given by rotating the above by $180^{\circ}$ and will follow analgously. The parity of the number of endpoints to the right of $a$ is odd, while the parity of the number of endpoints to the right of $b$ is even. This implies that there exists a string $s$ such that one of its endpoints belongs to the right of the dashed line, while the other belongs to the left. Moreover, since the right-most endpoint on the top and bottom row are coloured by $\downarrow$, we can say that the endpoint of $s$ which is to the right of the dashed line is a source while the endpoint to the left is a target. So $s$ must intersect one of the above strings, and must be a vertical string since all intersections with arcs are accounted for. Hence, colouring the string $s$ in red we have
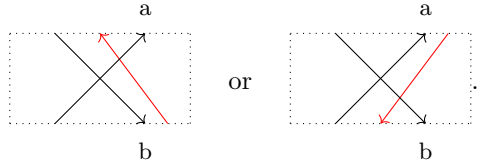
\[ 
\begin{tikzpicture}[scale=1.2, baseline=3ex]
\begin{scope}[dotted]
          \draw (1,0) to (1,1);
          \draw (1,0) to (3,0);
          \draw (3,0) to (3,1);
          \draw (1,1) to (3,1);
\end{scope}
\begin{scope}[red]
    \draw[->] (2.75,0) to (2,1);
\end{scope}
    \draw[->] (1.5,1) to (2.5,0);
    \draw[<-] (2.5,1) to (1.5,0);
    \node[label=above:{\small a}] (a) at (2.5,1){};
    \node[label=below:{\small b}] (b) at (2.5,0){};
\end{tikzpicture} \hspace{5mm} \text{or} \hspace{5mm}
\begin{tikzpicture}[scale=1.2, baseline=3ex]
\begin{scope}[dotted]
          \draw (1,0) to (1,1);
          \draw (1,0) to (3,0);
          \draw (3,0) to (3,1);
          \draw (1,1) to (3,1);
\end{scope}
\begin{scope}[red]
    \draw[->] (2.75,1) to (2,0);
\end{scope}
    \draw[->] (1.5,1) to (2.5,0);
    \draw[<-] (2.5,1) to (1.5,0);
    \node[label=above:{\small a}] (a) at (2.5,1){};
    \node[label=below:{\small b}] (b) at (2.5,0){};
\end{tikzpicture}.
\]
Note that $s$ may intersect both of the other strings and not just one, but it is always forced to intersect the string depicted. As such each situation exhibits an intersection accounted for in Cases 1 (or its $180^{\circ}$ counterpart), giving the desired contradiction. Thus removing all intersections of the types presented in Cases 1 to 3 (and their $180^{\circ}$ rotated counterparts) will result in a planar diagram, completing the proof. 

\end{proof}

Let $\mathsf{R}$ be an open subspace of $\mathbb{R}\times [0,1]$ and let $\alpha$ be an $(\bm{a},\bm{b})$-diagram. Examining $\alpha$ locally in $\mathsf{R}$ will give a configuration of curve segments, and we refer to such as a \emph{region} of $\alpha$. Within a given region we treat distinct curve segments as different curves, even if in $\alpha$ itself the two segments belong to the same curve. In particular, if in $\mathsf{R}$ two distinct curve segments intersect one another, and in $\alpha$ these two segments belong to the same curve, we will not call such an intersection a self-intersection in $\mathsf{R}$, but it is a self-intersection in $\alpha$. 

Recall that the local relation (H4) tells us that if a left curl appears in a diagram we may annihilate such a diagram. This relation asks that the region enclosed in the curl is absent of any other strings. The following result shows that even if such a region is non-empty, as long as its  contains no loops or self-intersections, we can annihilate the diagram.

\begin{lem} \label{TearDropLem}
Let $\alpha$ be an $(\bm{a},\bm{b})$-diagram containing a left curl where the region bounded by the curl contains no loops or self-intersecting curve segments, then $\alpha=0$.
\end{lem}

\begin{proof}
By assumption $\alpha$ contains a configuration of the form
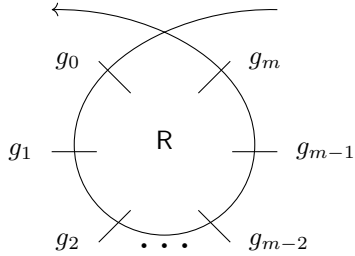
\[ 
\begin{tikzpicture}[scale=0.6]
    \draw[->] (6,6) to [curve through={(3.5,5.5) .. (1.5,3) .. (3.5,1) .. (5.5,3) .. (3.5,5.5)}] (1,6);
    \node[label=above:{$\mathsf{R}$}] (R) at (3.5,2.5){};
    \draw (1,2.85) to (2,2.85); 
    \draw (5,2.85) to (6,2.85); 
    \draw (2.75,4.149) to (2.0429,4.8561); 
    \draw (4.25,4.149) to (4.957,4.8561); 
    \draw (2.75,1.551) to (2.0429,0.8439); 
    \draw (4.25,1.551) to (4.957,0.8439); 
    \node[label=left:{$g_{0}$}] (g1) at (2.0429,4.8561){};
    \node[label=left:{$g_{1}$}] (R) at (1,2.85){};
    \node[label=left:{$g_{2}$}] (R) at (2.0429,0.8439){};
    \node[label=right:{$g_{m-2}$}] (R) at (4.957,0.8439){};
    \node[label=right:{$g_{m-1}$}] (R) at (6,2.85){};
    \node[label=right:{$g_{m}$}] (R) at (4.957,4.8561){};
    \node[label=above:{$\mathlarger{\mathlarger{\mathlarger{\mathlarger{\cdots}}}}$}] (dots) at (3.5,0.1){};
\end{tikzpicture}
\]
where we let $R$ denote the interior region bounded by the curl, which contains no loops or self-intersecting curve segments, and $g_{0},\dots,g_{m}$ account for all the intersections which occur on the curl. Note we have only drawn the segments of the $g_{i}$'s which realise the intersection on the curl. We prove the result by induction on the number of intersections occurring in $\mathsf{R}$. Assume that no intersections occur in $\mathsf{R}$, hence $\mathsf{R}$ gives a planar configuration of strings. One can deduce that there exists neighbours $g_{i\text{ mod}(m+1)}$ and $g_{(i+1)\text{mod}(m+1)}$ such that either
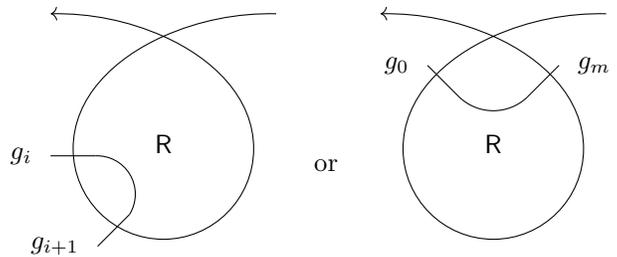
\[
\begin{tikzpicture}[scale=0.6, baseline=10ex]
    \draw[->] (6,6) to [curve through={(3.5,5.5) .. (1.5,3) .. (3.5,1) .. (5.5,3) .. (3.5,5.5)}] (1,6);
    \node[label=above:{$\mathsf{R}$}] (R) at (3.5,2.5){};
    \draw (1,2.85) to (2,2.85); 
    \draw (2.75,1.551) to (2.0429,0.8439); 
    \draw (2,2.85) to [curve through={(2.75,2.45)}] (2.75,1.551);
    \node[label=left:{$g_{i}$}] (R) at (1,2.85){};
    \node[label=left:{$g_{i+1}$}] (R) at (2.0429,0.8439){};
\end{tikzpicture} \hspace{5mm} \text{or} \hspace{5mm}
\begin{tikzpicture}[scale=0.6, baseline=10ex]
    \draw[->] (6,6) to [curve through={(3.5,5.5) .. (1.5,3) .. (3.5,1) .. (5.5,3) .. (3.5,5.5)}] (1,6);
    \node[label=above:{$\mathsf{R}$}] (R) at (3.5,2.5){};
    \draw (2.75,4.149) to (2.0429,4.8561); 
    \draw (4.25,4.149) to (4.957,4.8561); 
    \draw (2.75,4.149) to [curve through={(3.5,3.85)}] (4.25,4.149);
    \node[label=left:{$g_{0}$}] (g1) at (2.0429,4.8561){};
    \node[label=right:{$g_{m}$}] (R) at (4.957,4.8561){};
\end{tikzpicture}.
\]
In the former situation, since we are dealing with a left curl, one can check that regardless of the orientation of the depicted string in $\mathsf{R}$, it may be pulled outside the curl by (H2). For the latter situation we may employ (H1) to pull the string out of the curl over the crossing at the top. Continually pulling out such strings one at a time will result in making $\mathsf{R}$ empty, and then applying (H4) gives $\alpha=0$.

\vspace{2mm}
\noindent
Now suppose that the result holds whenever $\mathsf{R}$ contains $n$ or less intersections for some $n\geq 0$, and assume that $\mathsf{R}$ contains $n+1$ intersections. It is clear that there must exist an empty region $\mathsf{R}'$ in $\mathsf{R}$ bounded by the curl  and various segments. Diagrammatically we have
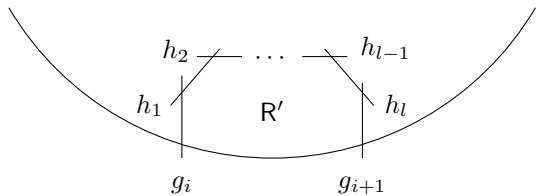
\[
\begin{tikzpicture}
    \draw (-1.5,2) to [curve through={(2,0)}] (5.5,2);
    \draw (0.8,0) to (0.8,1.1);
    \draw (3.2,0) to (3.2,1);
    \draw (0.65,0.7) to (1.3,1.45);
    \draw (3.35,0.7) to (2.7,1.45);
    \draw (1,1.35) to (1.6,1.35);
    \draw (3,1.35) to (2.4,1.35);
    \node[label=above:{$\cdots$}] (dots) at (2,1){};
    \node[label=above:{$\mathsf{R}'$}] (R') at (2,0.25){};
    \node[label=below:{$g_{i}$}] (gi) at (0.8,0){};
    \node[label=below:{$g_{i+1}$}] (gi) at (3.2,0){};
    \node[label=left:{$h_{1}$}] (h1) at (0.8,0.7){};
    \node[label=right:{$h_{l}$}] (hl) at (3.25,0.7){};
    \node[label=left:{$h_{2}$}] (h2) at (1.15,1.425){};
    \node[label=right:{$h_{l-1}$}] (hl-1) at (2.925,1.45){};
\end{tikzpicture}
\]
where $g_{i}, g_{i+1}$, and the (possibly empty) set of curve segments $H=\{h_{1},\dots,h_{l}\}$ make up the remainder of the boundary of $\mathsf{R}'$. Note such curve segments may not be pairwise distinct in $\mathsf{R}$. In the case when $H$ is empty, we simply have the situation
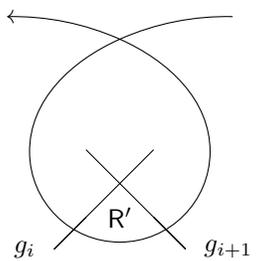
\[ 
\begin{tikzpicture}[scale=0.6]
    \draw[->] (6,6) to [curve through={(3.5,5.5) .. (1.5,3) .. (3.5,1) .. (5.5,3) .. (3.5,5.5)}] (1,6);
    \draw (2.75,1.551) to (2.0429,0.8439) to (4.25,3.05); 
    \draw (4.25,1.551) to (4.957,0.8439) to (2.75,3.05); 
    \node[label=left:{$g_{i}$}] (gi) at (2.0429,0.8439){};
    \node[label=right:{$g_{i+1}$}] (gi+1) at (4.957,0.8439){};
    \node[label=above:{$\mathsf{R'}$}] (R') at (3.5,0.9){};
\end{tikzpicture}.
\]
Since $\mathsf{R}'$ is empty we may pull this crossing out of the curl by (H1), which will decrease the number of intersection in $\mathsf{R}$ and thus by induction $\alpha=0$. Hence we may assume that $H$ is non-empty. The general case $H=\{h_{1},\dots,h_{l}\}$ is solved by focusing on $h_{1}$, and in fact solving the case $H=\{h_{1}\}$ is sufficient to understand the general case, hence we only prove this case. So we are working with the sitaution
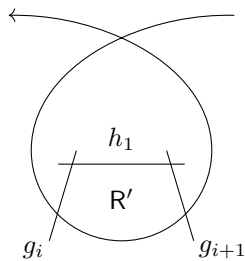
\[
\begin{tikzpicture}[scale=0.6]
    \draw[->] (6,6) to [curve through={(3.5,5.5) .. (1.5,3) .. (3.5,1) .. (5.5,3) .. (3.5,5.5)}] (1,6);
    \draw (1.9,1) to (2.5,3); 
    \draw (5.1,1) to (4.5,3); 
    \draw (2.1,2.7) to (4.9,2.7); 
    \node[label=left:{$g_{i}$}] (gi) at (2.2,0.8){};
    \node[label=right:{$g_{i+1}$}] (gi+1) at (4.8,0.8){};
    \node[label=above:{$\mathsf{R'}$}] (R') at (3.5,1.3){};
    \node[label=above:{$h_{1}$}] (h1) at (3.5,2.6){};
\end{tikzpicture}
\] 
There are two cases to consider based on the orientation of $h_{1}$. For the first case we have 
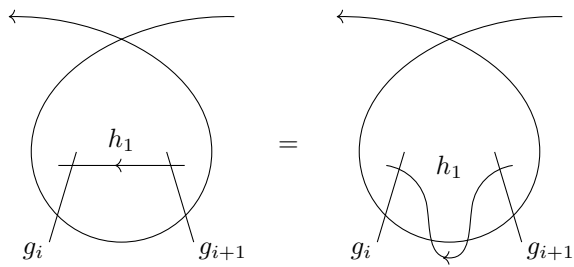
\[
\begin{tikzpicture}[scale=0.6, baseline=12ex]
    \draw[->] (6,6) to [curve through={(3.5,5.5) .. (1.5,3) .. (3.5,1) .. (5.5,3) .. (3.5,5.5)}] (1,6);
    \draw (1.9,1) to (2.5,3); 
    \draw (5.1,1) to (4.5,3); 
\begin{scope}[decoration={markings, mark=at position 0.5 with {\arrow{<}}}]
    \draw[postaction={decorate}] (2.1,2.7) to (4.9,2.7); 
\end{scope}
    \node[label=left:{$g_{i}$}] (gi) at (2.2,0.8){};
    \node[label=right:{$g_{i+1}$}] (gi+1) at (4.8,0.8){};
    \node[label=above:{$h_{1}$}] (h1) at (3.5,2.6){};
\end{tikzpicture} \hspace{2mm} = \hspace{4mm}
\begin{tikzpicture}[scale=0.6, baseline=12ex]
    \draw[->] (6,6) to [curve through={(3.5,5.5) .. (1.5,3) .. (3.5,1) .. (5.5,3) .. (3.5,5.5)}] (1,6);
    \draw (1.9,1) to (2.5,3); 
    \draw (5.1,1) to (4.5,3); 
\begin{scope}[decoration={markings, mark=at position 0.5 with {\arrow{<}}}]
    \draw[postaction={decorate}] (2.1,2.7) to [curve through={(3,1.8) .. (3.5,0.65) .. (4,1.8)}] (4.9,2.7); 
\end{scope}
    \node[label=left:{$g_{i}$}] (gi) at (2.2,0.8){};
    \node[label=right:{$g_{i+1}$}] (gi+1) at (4.8,0.8){};
    \node[label=above:{$h_{1}$}] (h1) at (3.5,2){};
\end{tikzpicture} 
\] 
by (H2). Then we may pull the crossing between either $g_{i}$ and $h_{1}$, or $g_{i+1}$ and $h_{1}$ out of the curl by (H1), which will decrease the number of intersections in $\mathsf{R}$ by one and so $\alpha=0$ by induction. With the opposite orientation on $h_{1}$ we have
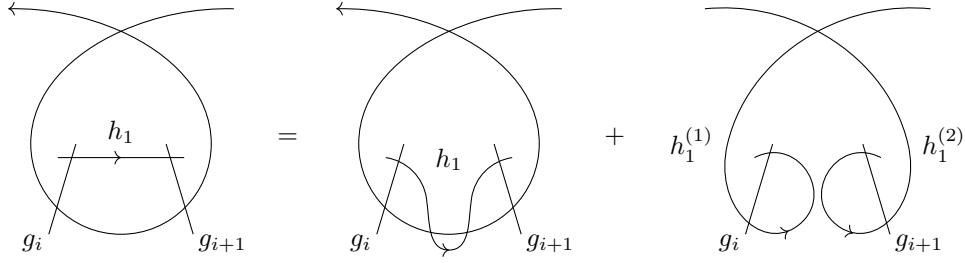
\[
\begin{tikzpicture}[scale=0.6, baseline=12ex]
    \draw[->] (6,6) to [curve through={(3.5,5.5) .. (1.5,3) .. (3.5,1) .. (5.5,3) .. (3.5,5.5)}] (1,6);
    \draw (1.9,1) to (2.5,3); 
    \draw (5.1,1) to (4.5,3); 
\begin{scope}[decoration={markings, mark=at position 0.5 with {\arrow{>}}}]
    \draw[postaction={decorate}] (2.1,2.7) to (4.9,2.7); 
\end{scope}
    \node[label=left:{$g_{i}$}] (gi) at (2.2,0.8){};
    \node[label=right:{$g_{i+1}$}] (gi+1) at (4.8,0.8){};
    \node[label=above:{$h_{1}$}] (h1) at (3.5,2.6){};
\end{tikzpicture} \hspace{2mm} = \hspace{4mm}
\begin{tikzpicture}[scale=0.6, baseline=12ex]
    \draw[->] (6,6) to [curve through={(3.5,5.5) .. (1.5,3) .. (3.5,1) .. (5.5,3) .. (3.5,5.5)}] (1,6);
    \draw (1.9,1) to (2.5,3); 
    \draw (5.1,1) to (4.5,3); 
\begin{scope}[decoration={markings, mark=at position 0.5 with {\arrow{>}}}]
    \draw[postaction={decorate}] (2.1,2.7) to [curve through={(3,1.8) .. (3.5,0.65) .. (4,1.8)}] (4.9,2.7); 
\end{scope}
    \node[label=left:{$g_{i}$}] (gi) at (2.2,0.8){};
    \node[label=right:{$g_{i+1}$}] (gi+1) at (4.8,0.8){};
    \node[label=above:{$h_{1}$}] (h1) at (3.5,2){};
\end{tikzpicture} \hspace{2mm} + \hspace{4mm}
\begin{tikzpicture}[scale=0.6, baseline=12ex]
    \draw (1.9,1) to (2.5,3); 
    \draw (5.1,1) to (4.5,3); 
\begin{scope}[decoration={markings, mark=at position 0.75 with {\arrow{>}}}]
    \draw[postaction={decorate}] (6,6) to [curve through={(3.5,5.5) .. (1.5,3) .. (3,1.15)}] (2.1,2.7); 
\end{scope}
\begin{scope}[decoration={markings, mark=at position 0.75 with {\arrow{<}}}]
    \draw[postaction={decorate}] (1,6) to [curve through={(3.5,5.5) .. (5.5,3) .. (4,1.15)}] (4.9,2.7); 
\end{scope}
    \node[label=left:{$g_{i}$}] (gi) at (2.2,0.8){};
    \node[label=right:{$g_{i+1}$}] (gi+1) at (4.8,0.8){};
    \node[label=left:{$h_{1}^{(1)}$}] (h11) at (1.6,3){};
    \node[label=right:{$h_{1}^{(2)}$}] (h12) at (5.4,3){};
\end{tikzpicture}
\]
by (H3). Here denote the first diagram on the right of the above equation by $\alpha_{1}$ and the second by $\alpha_{2}$. For $\alpha_{1}$, as was done in the previous case we may pull one of the crossings outside of the curl, and thus decrease the number of intersections in $\mathsf{R}$ by one, and hence $\alpha_{1}=0$ by induction. For $\alpha_{2}$ the curve containing $h_{1}$ and the original left curl have been turned into the two new curves $h_{1}^{(1)}$ and $h_{1}^{(2)}$. Note the original left curl is no longer present, but regardless of how the original curve containing $h_{1}$ intersected the curl, at least one of the new curves $h_{1}^{(1)}$ and $h_{1}^{(2)}$ must form a new, smaller, left curl. The region bounded by this new curl is a subregion of $\mathsf{R}$ containing strictly less number of intersections. Hence by induction $\alpha_{2}=0$, and so collectively $\alpha=\alpha_{1}+\alpha_{2}=0$ completing the proof by induction. Note the general case for $H=\{h_{1},\dots,h_{l}\}$ is tackled in the exact same manner by pulling $h_{1}$ out of the curl, the diagrammatics are just more cluttered, but the remaining segments $h_{2},\dots,h_{l}$ do not interfer with the above argumenets.

\end{proof}

Let $\bm{a}=a_{1}\cdots a_{k}$ and $\bm{b}=b_{1}\cdots b_{l}$ for $a_{i},b_{i}\in\{\uparrow,\downarrow\}$, and consider the map $\mathsf{deg}:\mathsf{SSim}(\bm{a},\bm{b})\rightarrow \mathbb{Z}_{\geq 0}\times \mathbb{Z}_{\geq 0}$ be given by $\mathsf{deg}(\alpha)=(A(\alpha),C(\alpha))$ where $A(\alpha)$ is the number of arcs in $\alpha$, and $C(\alpha)$ is the number of clockwise arcs in $\alpha$. We order the degrees by using the lexicographical ordering $<$ on $\mathbb{Z}_{\geq 0}\times \mathbb{Z}_{\geq 0}$. 
Note that for any $\alpha\in \mathsf{SSim}(\bm{a}, \bm{b})$ we have $\mathsf{deg}(\overline{\alpha}) = \mathsf{deg}(\alpha)$. 

\begin{prop} \label{SSLeadingTerm}
Let $\alpha\in\mathsf{SSim}(\bm{a},\bm{b})$. Then 
\[ \alpha \hspace{2mm}=\hspace{2mm} \overline{\alpha}\hspace{2mm}+\sum\limits_{\substack{\beta\in\mathsf{Sim}(\bm{a},\bm{b}) \\ \mathsf{deg}(\beta)>\mathsf{deg}(\alpha)}}c_{\beta}\beta \]
where $c_{\beta}\in\mathbb{Z}$.
\end{prop}

\begin{proof}
Given two distinct strings $s$ and $t$ in $\alpha$, let $n$ be the number of intersections occurring between the two strings. If $n$ is even set $\mu(\{s,t\})=n$, while if $n$ is odd set $\mu(\{s,t\})=n-1$. Note $\mu(\{s,t\})$ is always even. We let
\[ \eta(\alpha)=\sum_{s,t}\mu(\{s,t\}), \]
where the sum runs over all unordered pairs of distinct strings of $\alpha$. Informally, $\eta(\alpha)$ is the number of intersections of $\alpha$ which prevent it from being simple, in particular $\eta(\alpha)=0$ if and only if $\alpha\in\mathsf{Sim}(\bm{a},\bm{b})$. We will prove this proposition by induction on $\eta(\alpha)$, where the base case of $\eta(\alpha)=0$ follows immediately since $\alpha=\overline{\alpha}$. Assume the result holds for all $\alpha\in\mathsf{SSim}(\bm{a},\bm{b})$ such that $\eta(\alpha)<n$ for some $n>0$. Now let $\alpha$ be such that $\eta(\alpha)=n$. Pick two strings $s$ and $t$ in $\alpha$ such that $\mu(\{s,t\})\geq 2$. Order the points of intersections between $s$ and $t$ according to when they appear as one travels from the source of $s$ to its target. Under this ordering pick two neighbouring points of intersection $p$ and $q$. Then diagrammatically we have a configuration of strings of the form
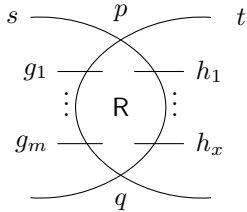
\[
\begin{tikzpicture}[scale=0.6]
    \draw (0,4) to [curve through={(2,3.5) .. (3,2) .. (2,0.5)}] (0,0);
    \draw (4,4) to [curve through={(2,3.5) .. (1,2) .. (2,0.5)}] (4,0);
    \draw (0.6,2.8) to (1.6,2.8);
    \draw (0.6,1.2) to (1.6,1.2);
    \draw (2.3,2.8) to (3.4,2.8);
    \draw (2.3,1.2) to (3.4,1.2);
    \node[label=above:{$\mathsf{R}$}] (R) at (2,1.4){};
    \node[label=left:{$g_{1}$}] (g1) at (0.8,2.8){};
    \node[label=left:{$g_{m}$}] (gm) at (0.8,1.2){};
    \node[label=right:{$h_{1}$}] (h1) at (3.25,2.8){};
    \node[label=right:{$h_{x}$}] (hx) at (3.25,1.2){};
    \node[label=above:{$\vdots$}] (dotsl) at (0.8,1.4){};
    \node[label=above:{$\vdots$}] (dotsr) at (3.2,1.4){};
    \node[label=left:{$s$}] (s) at (0.15,4){};
    \node[label=right:{$t$}] (t) at (4.15,4){};
    \node[label=above:{$p$}] (p) at (2,3.5){};
    \node[label=below:{$q$}] (q) at (2,0.5){};
\end{tikzpicture}
\]
where $\mathsf{R}$ is the interior region bounded by the curve segments of $s$ and $t$ between the points of intersection $p$ and $q$, and the two (possible empty) sets of string segments $G=\{g_{1},\dots,g_{m}\}$ and $H=\{h_{1},\dots,h_{x}\}$ account for all the intersections of the boundary of $\mathsf{R}$ through $t$ and $s$ respectively. We may assume that we are not in one of the following three situations:
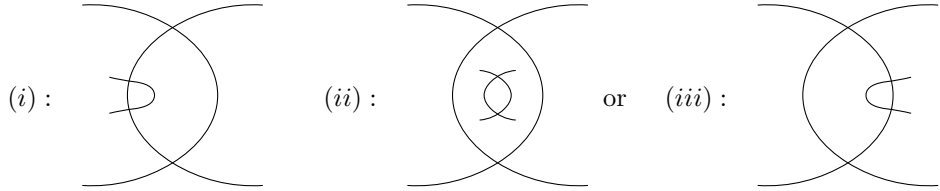
\[
(i): \hspace{3mm} \begin{tikzpicture}[scale=0.6, baseline=7ex]
    \draw (0,4) to [curve through={(2,3.5) .. (3,2) .. (2,0.5)}] (0,0);
    \draw (4,4) to [curve through={(2,3.5) .. (1,2) .. (2,0.5)}] (4,0);
    \draw (0.6,1.6) to [curve through={(1.15,1.7) .. (1.6,2) .. (1.15,2.3)}] (0.6,2.4);
\end{tikzpicture} \hspace{8mm}
(ii): \hspace{3mm} \begin{tikzpicture}[scale=0.6, baseline=7ex]
    \draw (0,4) to [curve through={(2,3.5) .. (3,2) .. (2,0.5)}] (0,0);
    \draw (4,4) to [curve through={(2,3.5) .. (1,2) .. (2,0.5)}] (4,0);
    \draw (1.6,1.45) to [curve through={(2.15,1.7) .. (2.3,2) .. (2.15,2.3)}] (1.6,2.55);
    \draw (2.4,1.45) to [curve through={(1.85,1.7) .. (1.7,2) .. (1.85,2.3)}] (2.4,2.55);
\end{tikzpicture} \hspace{2mm} \text{or} \hspace{5mm}
(iii): \hspace{3mm} \begin{tikzpicture}[scale=0.6, baseline=7ex]
    \draw (0,4) to [curve through={(2,3.5) .. (3,2) .. (2,0.5)}] (0,0);
    \draw (4,4) to [curve through={(2,3.5) .. (1,2) .. (2,0.5)}] (4,0);
    \draw (3.4,1.6) to [curve through={(2.85,1.7) .. (2.4,2) .. (2.85,2.3)}] (3.4,2.4);
\end{tikzpicture}
\]
since otherwise we may pick the more nested pair of intersections to work with instead. Since situations $(i)$ and $(iii)$ are not present, any string segment $g_{i}$ must connect to a $h_{j}$ (rather than another segment in $G$). Hence $m=x$ and $\mathsf{R}$ realises a pairing of the string segments $G$ with $H$. Diagrammatically we have
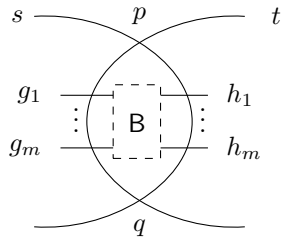
\[
\begin{tikzpicture}[scale=0.7]
    \draw (0,4) to [curve through={(2,3.5) .. (3,2) .. (2,0.5)}] (0,0);
    \draw (4,4) to [curve through={(2,3.5) .. (1,2) .. (2,0.5)}] (4,0);
    \draw (0.5,2.5) to (1.5,2.5);
    \draw (0.5,1.5) to (1.5,1.5);
    \draw (2.4,2.5) to (3.3,2.5);
    \draw (2.4,1.5) to (3.3,1.5);
    \node[label=left:{$g_{1}$}] (g1) at (0.5,2.5){};
    \node[label=left:{$g_{m}$}] (gm) at (0.5,1.5){};
    \node[label=right:{$h_{1}$}] (h1) at (3.3,2.5){};
    \node[label=right:{$h_{m}$}] (hl) at (3.3,1.5){};
    \node[label=above:{$\vdots$}] (dotsl) at (0.8,1.45){};
    \node[label=above:{$\vdots$}] (dotsr) at (3.2,1.45){};
    \node[label=left:{$s$}] (s) at (0.15,4){};
    \node[label=right:{$t$}] (t) at (4.15,4){};
    \node[label=above:{$p$}] (p) at (2,3.5){};
    \node[label=below:{$q$}] (q) at (2,0.5){};
\begin{scope}[dashed]
    \draw (1.5,2.7) rectangle (2.4,1.3);
\end{scope}
    \node[label=above:{$\mathsf{B}$}] (P) at (1.95,1.45){};
\end{tikzpicture}
\]
where $\mathsf{B}$ is some permutation connecting segments in $G$ with those in $H$. Moreover, since situation $(ii)$ is not present, this means that no string segments in $\mathsf{B}$ can intersect more that once. In other words $\mathsf{B}$ is built out of crossings, and so we may pull all of $\mathsf{B}$ outside of the region $\mathsf{R}$ one crossing at a time by (H1), and thus obtain
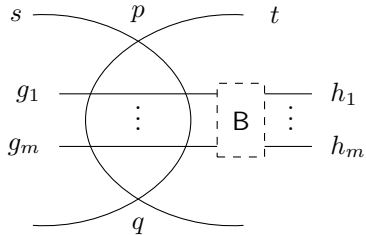
\[
\begin{tikzpicture}[scale=0.7]
    \draw (0,4) to [curve through={(2,3.5) .. (3,2) .. (2,0.5)}] (0,0);
    \draw (4,4) to [curve through={(2,3.5) .. (1,2) .. (2,0.5)}] (4,0);
    \draw (0.5,2.5) to (3.5,2.5);
    \draw (0.5,1.5) to (3.5,1.5);
    \draw (4.4,2.5) to (5.3,2.5);
    \draw (4.4,1.5) to (5.3,1.5);
    \node[label=left:{$g_{1}$}] (g1) at (0.5,2.5){};
    \node[label=left:{$g_{m}$}] (gm) at (0.5,1.5){};
    \node[label=right:{$h_{1}$}] (h1) at (5.3,2.5){};
    \node[label=right:{$h_{m}$}] (hl) at (5.3,1.5){};
    \node[label=above:{$\vdots$}] (dotsl) at (2,1.45){};
    \node[label=above:{$\vdots$}] (dotsr) at (4.9,1.45){};
    \node[label=left:{$s$}] (s) at (0.15,4){};
    \node[label=right:{$t$}] (t) at (4.15,4){};
    \node[label=above:{$p$}] (p) at (2,3.5){};
    \node[label=below:{$q$}] (q) at (2,0.5){};
\begin{scope}[dashed]
    \draw (3.5,2.7) rectangle (4.4,1.3);
\end{scope}
    \node[label=above:{$\mathsf{B}$}] (P) at (3.95,1.45){};
\end{tikzpicture}
\]
Lastly we may pull these horizontal strings out of $\mathsf{R}$ through the top or bottom crossing by (H1). Hence we have emptied $\mathsf{R}$ by employing only local relation (H1), and so the value $\eta(\alpha)$ has remained the same. Now there are four different cases depending on the orientations of the strings $s$ and $t$. In three of these cases, since $\mathsf{R}$ is empty, we may pull the strings $s$ and $t$ apart by applying (H2), and thus remove the two intersections $p$ and $q$. This decreases $\eta(\alpha)$ by two, and so the result follows by induction. The last case is given with orientations as follows
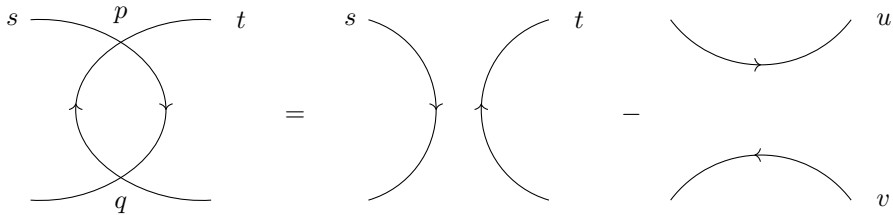
\[
\begin{tikzpicture}[scale=0.6, baseline=7ex]
\begin{scope}[decoration={markings, mark=at position 0.5 with {\arrow{>}}}]
    \draw[postaction={decorate}] (0,4) to [curve through={(2,3.5) .. (3,2) .. (2,0.5)}] (0,0);
\end{scope}
\begin{scope}[decoration={markings, mark=at position 0.5 with {\arrow{<}}}]
    \draw[postaction={decorate}] (4,4) to [curve through={(2,3.5) .. (1,2) .. (2,0.5)}] (4,0);
\end{scope}
    \node[label=left:{$s$}] (s) at (0.15,4){};
    \node[label=right:{$t$}] (t) at (4.15,4){};
    \node[label=above:{$p$}] (p) at (2,3.5){};
    \node[label=below:{$q$}] (q) at (2,0.5){};
\end{tikzpicture} \hspace{3mm} = \hspace{3mm}
\begin{tikzpicture}[scale=0.6, baseline=7ex]
\begin{scope}[decoration={markings, mark=at position 0.5 with {\arrow{>}}}]
    \draw[postaction={decorate}] (0,4) to [curve through={(1.5,2)}] (0,0);
\end{scope}
\begin{scope}[decoration={markings, mark=at position 0.5 with {\arrow{<}}}]
    \draw[postaction={decorate}] (4,4) to [curve through={(2.5,2)}] (4,0);
\end{scope}
    \node[label=left:{$s$}] (s) at (0.15,4){};
    \node[label=right:{$t$}] (t) at (4.15,4){};
\end{tikzpicture} \hspace{3mm} - \hspace{3mm}
\begin{tikzpicture}[scale=0.6, baseline=7ex]
\begin{scope}[decoration={markings, mark=at position 0.5 with {\arrow{>}}}]
    \draw[postaction={decorate}] (0,4) to [curve through={(2,3)}] (4,4);
\end{scope}
\begin{scope}[decoration={markings, mark=at position 0.5 with {\arrow{<}}}]
    \draw[postaction={decorate}] (0,0) to [curve through={(2,1)}] (4,0);
\end{scope}
    \node[label=right:{$u$}] (s) at (4.15,4){};
    \node[label=right:{$v$}] (t) at (4.15,0){};
\end{tikzpicture} 
\]
where we have applied (H3). Let the two diagrams on the right hand side of the above equation be denoted by $\alpha_{1}$ and $\alpha_{2}$ respectively, hence $\alpha=\alpha_{1}-\alpha_{2}$. It is clear that $\overline{\alpha_{1}}=\overline{\alpha}$ and $\mathsf{deg}(\alpha_{1})=\mathsf{deg}(\alpha)$. Moreover we have that $\eta(\alpha_{1})=\eta(\alpha)-2$, and so by induction
\begin{equation} \label{alpha1}
\alpha_{1} = \overline{\alpha} + \sum\limits_{\substack{\beta\in\mathsf{Sim}(\bm{a},\bm{b}) \\ \mathsf{deg}(\beta)>\mathsf{deg}(\alpha)}}d_{\beta}\beta
\end{equation}
where $d_{\beta}\in\mathbb{Z}$. As for $\alpha_{2}$, the original strings $s$ and $t$ have been replaced by $u$ and $v$. Although the points of intersection $p$ and $q$ have been removed, in general we cannot apply the inductive step for $\alpha_{2}$ as it may not be semisimple, since the new strings $u$ and $v$ may contain self-intersections. This occurs precisely when there are more intersections between the strings $s$ and $t$ than just $p$ and $q$. So we break this situation into two cases:

\vspace{2mm}
\noindent
(Case 1) Assume that $p$ and $q$ are the only intersections between the strings $s$ and $t$ in $\alpha$, and so $\alpha_{2}$ is semisimple. Thus by induction we have
\begin{equation} \label{alpha2}
\alpha_{2} = \overline{\alpha_{2}} + \sum\limits_{\substack{\beta\in\mathsf{Sim}(\bm{a},\bm{b}) \\ \mathsf{deg}(\beta)>\mathsf{deg}(\alpha_{2})}}f_{\beta}\beta
\end{equation}
where $f_{\beta}\in\mathbb{Z}$. We seek to show that $\mathsf{deg}(\alpha_{2})>\mathsf{deg}(\alpha)$, and then subtracking \emph{\Cref{alpha2}} away from \emph{\Cref{alpha1}} will prove this case. One can show this by comparing the string types of the sets $\{s,t\}$ and $\{u,v\}$. We have the following to consider:
\begin{itemize}
\item[(1)] The set $\{s,t\}$ contains a down and up string.
\item[(2)] The set $\{s,t\}$ contains a vertical string and clockwise arc.
\item[(3)] The set $\{s,t\}$ contains two arcs on the same row, but not both anti-clockwise.
\end{itemize}
Note $\{s,t\}$ cannot contain a top and bottom arc since $\alpha$ is semisimple. The remaining cases which have been left out are due to the fact they can never realise the orientated double crossing of the strings $s$ and $t$ which we are considering. For (1) it is easy to see that $\{u,v\}$ consists of two arcs. For (2) one can deduce that $\{u,v\}$ contains a vertical string and anti-clockwise arc. For (3), when $\{s,t\}$ consists of two clockwise arcs one can check that $\{u,v\}$ consists of a clockwsie arc and an anti-clockwise arc. When $\{s,t\}$ contains a clockwise and anti-clockwise arc, one can check that $\{u,v\}$ consists of two anti-clockwise arcs. For all these case we have $\mathsf{deg}(\alpha_{2})>\mathsf{deg}(\alpha)$, completing Case 1.

\vspace{2mm}
\noindent
(Case 2) Assume now that there is at least one more point of intersection between the strings $s$ and $t$ beside $p$ or $q$. In the ordering of intersections discussed previously, pick a neighbouring point which either preceeds $p$ or proceeds $q$, say $y$. Without loss of generality assume $y$ preceeds $p$. Then diagrammatically the equation $\alpha=\alpha_{1}-\alpha_{2}$ is given by
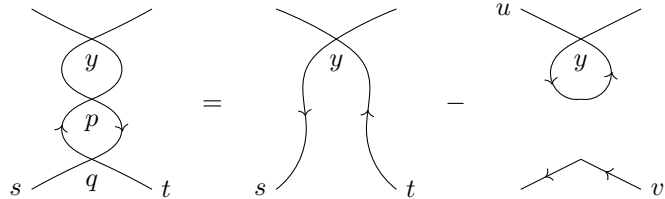
\[
\begin{tikzpicture}[scale=0.8, baseline=7ex]
\begin{scope}[decoration={markings, mark=at position 0.37 with {\arrow{<}}}]
    \draw[postaction={decorate}] (0,0) to [curve through={(1,0.5) .. (1.5,1) .. (1,1.5) .. (0.5,2) .. (1,2.5)}] (2,3);
\end{scope}
\begin{scope}[decoration={markings, mark=at position 0.37 with {\arrow{>}}}]
    \draw[postaction={decorate}] (2,0) to [curve through={(1,0.5) .. (0.5,1) .. (1,1.5) .. (1.5,2) .. (1,2.5)}] (0,3);
\end{scope}
    \node[label=below:{$y$}] (y) at (1,2.55){};
    \node[label=below:{$p$}] (p) at (1,1.55){};
    \node[label=below:{$q$}] (q) at (1,0.55){};
    \node[label=left:{$s$}] (s) at (0.15,0){};
    \node[label=right:{$t$}] (t) at (1.85,0){};
\end{tikzpicture} \hspace{2mm} = \hspace{2mm}
\begin{tikzpicture}[scale=0.8, baseline=7ex]
\begin{scope}[decoration={markings, mark=at position 0.37 with {\arrow{<}}}]
    \draw[postaction={decorate}] (0,0) to [curve through={(0.5,1) .. (0.5,2) .. (1,2.5)}] (2,3);
\end{scope}
\begin{scope}[decoration={markings, mark=at position 0.37 with {\arrow{>}}}]
    \draw[postaction={decorate}] (2,0) to [curve through={(1.5,1) .. (1.5,2) .. (1,2.5)}] (0,3);
\end{scope}
    \node[label=below:{$y$}] (y) at (1,2.55){};
    \node[label=left:{$s$}] (s) at (0.15,0){};
    \node[label=right:{$t$}] (t) at (1.85,0){};
\end{tikzpicture} \hspace{2mm} - \hspace{2mm}
\begin{tikzpicture}[scale=0.8, baseline=7ex]
\begin{scope}[decoration={markings, mark=at position 0.5 with {\arrow{<}}}]
    \draw[postaction={decorate}] (0,0) to (1,0.5);
    \draw[postaction={decorate}] (1,0.5) to (2,0);
\end{scope}
\begin{scope}[decoration={markings, mark=at position 0.75 with {\arrow{<}}}]
    \draw[postaction={decorate}] (0,3) to [curve through={(1,2.5) .. (1.5,2)}] (1,1.5);
\end{scope}
\begin{scope}[decoration={markings, mark=at position 0.25 with {\arrow{<}}}]
    \draw[postaction={decorate}] (1,1.5) to [curve through={(0.5,2) .. (1,2.5)}] (2,3);
\end{scope}
    \node[label=below:{$y$}] (y) at (1,2.55){};
    \node[label=left:{$u$}] (s) at (0.15,3){};
    \node[label=right:{$v$}] (t) at (1.85,0){};
\end{tikzpicture} 
\]
by (H3). In $\alpha_{2}$ the interior region bounded by the left curl cannot contain loops or string segments with self-intersections since $\alpha$ is semisimple. Hence by \emph{\Cref{TearDropLem}} $\alpha_{2}=0$, and so $\alpha=\alpha_{1}$ and thus the result follows by \emph{\Cref{alpha1}}. 

\end{proof}

\begin{thm} \label{VarphiSurj}
The algebra homomorphism $\varphi:\mathcal{A}_{2k}^{\emph{aff}}\rightarrow \mathsf{End}_{\mathsf{Heis}}((\uparrow\downarrow)^{k})$ given in \emph{\Cref{APADiaHom}} is surjective.
\end{thm}

\begin{proof}
As discussed previously, this will follow by showing that $\alpha \in \langle \varphi(e_{i}), \varphi(\tau_{j})\rangle_{i,j}$ for all $\alpha \in \mathsf{Sim}((\uparrow\downarrow)^{k})$. We prove this by downwards induction on $\mathsf{deg}(\alpha)$. It's easy to see that the maximum degree is $\mathsf{deg}(\alpha)=(2k,2k)$.  By considering what endpoints can be targets and sources of clockwise arcs, one can deduce the only element $\alpha\in \mathsf{Sim}((\uparrow\downarrow)^{k})$ satisfying $\mathsf{deg}(\alpha)=(2k,2k)$ is given by 
\[ 
\begin{tikzpicture}[scale=0.6, baseline=1ex]
    \draw[->] (0,0) to [bend left] (1,0);
    \draw[->] (2,0) to [bend left] (3,0);
    \draw[<-] (0,1) to [bend right] (1,1);
    \draw[<-] (2,1) to [bend right] (3,1);
    \node[label=above:{$\cdots$}] (dots) at (3.5,0){};
    \draw[->] (4,0) to [bend left] (5,0);
    \draw[<-] (4,1) to [bend right] (5,1);
\end{tikzpicture} = \varphi\left(\prod_{i\in[k]}e_{2i-1}\right) 
\]
This completes the base case. 
Now, pick $\alpha$ such that $\mathsf{deg}(\alpha)=(x,y)<(2k,2k)$ and assume that $\gamma\in\langle \varphi(e_{i}), \varphi(\tau_{j})\rangle_{i,j}$ for all $\gamma \in \mathsf{Sim}((\uparrow\downarrow)^{k})$ such that $\mathsf{deg}(\gamma)>(x,y)$. By \emph{\Cref{SimBarDecomp}} there exists $\pi,\sigma\in S(k)$ and a planar diagram $\beta\in\mathsf{Sim}((\uparrow\downarrow)^{k})$ such that $\pi^{\uparrow}\beta\sigma^{\downarrow}$ is semisimple and $\alpha=\overline{\pi^{\uparrow}\beta\sigma^{\downarrow}}$, in particular $\mathsf{deg}(\alpha)=\mathsf{deg}(\pi^{\uparrow}\beta\sigma^{\downarrow})$. Hence by \emph{\Cref{SSLeadingTerm}} we have that
\begin{equation} \label{FiltSum}
\pi^{\uparrow}\beta\sigma^{\downarrow} = \alpha + \sum\limits_{\substack{\gamma\in\mathsf{Sim}((\uparrow\downarrow)^{k}) \\ \mathsf{deg}(\gamma)>\mathsf{deg}(\alpha)}}c_{\gamma}\gamma,
\end{equation}
where $c_{\gamma}\in\mathbb{Z}$. By induction all the simple terms in the above summation belong to $\mathsf{Im}(\varphi)$. Also from previous discussions we know that $\pi^{\uparrow}\beta\sigma^{\downarrow}\in\mathsf{Im}(\varphi)$, thus rearranging the above equation shows that $\alpha\in\mathsf{Im}(\varphi)$, completing the proof by induction. 

\end{proof}

\begin{rmk}
\emph{\Cref{FiltSum}} is the key to \emph{\Cref{VarphiSurj}}, and follows from \emph{\Cref{SSLeadingTerm}}. This proposition applies to all semisimple diagrams which are much more general than those appearing here. Ideally, one would like to prove  that \emph{\Cref{FiltSum}} holds for $\pi^{\uparrow}\beta\sigma^{\downarrow}$ by some inductive argument without needing to show it for all semisimple diagrams. However it is a very delicate task to check which properties are preserved by an inductive process. So we ended up using this more general approach instead, even though many of the cases considered in proving \emph{\Cref{SSLeadingTerm}} probably won't occur in this case.
\end{rmk}


\subsection{The affine partition category of Brundan and Vargas}


In this last section we relate our affine partition algebra to the work of J. Brundan and M. Vargas in \cite{BV21} and prove a new result on their category.
We start by recalling the definition of their affine partition category $\mathsf{APar}$ as a subcategory of $\mathsf{Heis}$ generated by certain objects and morphisms, and  of their affine partition algebra $AP_{k}$, which is  an endomorphism algebra within $\mathsf{APar}$. 

\begin{defn} \cite[Definition 4.6 and Equation 4.47]{BV21}
The \emph{affine partition category} $\mathsf{APar}$ is the monoidal subcategory of $\mathsf{Heis}$ generated by the object $\uparrow\downarrow$ and the following morphisms:
\begin{equation} \label{HCox}
\begin{tikzpicture}[scale=0.6, baseline=1.5ex]
    \draw[->] (1,0) to (3,1);
    \draw[<-] (2,0) to (4,1);
    \draw[->] (3,0) to (1,1);
    \draw[<-] (4,0) to (2,1);
\end{tikzpicture} \hspace{2mm} + \hspace{2mm}
\begin{tikzpicture}[scale=0.6, baseline=1.5ex]
    \draw[->] (1,0) to (1,1);
    \draw[<-] (2,0) to [bend left] (3,0);
    \draw[->] (2,1) to [bend right] (3,1);
    \draw[<-] (4,0) to (4,1);
\end{tikzpicture}
\end{equation}

\begin{equation} \label{HTri}
\begin{tikzpicture}[scale=0.6]
    \draw[->] (0,0) to (1,1);
    \draw[->] (2,1) to (3,0);
    \draw[->] (2,0) to [bend right] (1,0);
\end{tikzpicture}, \hspace{5mm}
\begin{tikzpicture}[scale=0.6]
    \draw[<-] (0,1) to (1,0);
    \draw[->] (1,1) to [bend right] (2,1);
    \draw[->] (3,1) to (2,0);
\end{tikzpicture}
\end{equation}

\begin{equation} \label{CupsCaps}
\begin{tikzpicture}[scale=0.6]
    \draw[->] (1,0) to [curve through={(1.5,0.35)}] (2,0);
\end{tikzpicture}, \hspace{5mm}
\begin{tikzpicture}[scale=0.6]    
    \draw[<-] (1,0) to [curve through={(1.5,-0.35)}] (2,0);
\end{tikzpicture}
\end{equation}

\begin{equation} \label{HAff}
\begin{tikzpicture}[scale=0.6, baseline=1.5ex]
    \draw[->] (1,0) to (1,1);
    \draw[->] (2,1) to (2,0);
    \filldraw (1,0.5) circle (2pt);
\end{tikzpicture} \hspace{2mm} + \hspace{2mm}
\begin{tikzpicture}[scale=0.6, baseline=1.5ex]    
    \draw[->] (1,0) to (1,1);
    \draw[->] (2,1) to (2,0);
\end{tikzpicture}, \hspace{5mm}
\begin{tikzpicture}[scale=0.6, baseline=1.5ex]
    \draw[->] (1,0) to (1,1);
    \draw[->] (2,1) to (2,0);
    \filldraw (2,0.5) circle (2pt);
\end{tikzpicture} \hspace{2mm} + \hspace{2mm}
\begin{tikzpicture}[scale=0.6, baseline=1.5ex]    
    \draw[->] (1,0) to (1,1);
    \draw[->] (2,1) to (2,0);
\end{tikzpicture}
\end{equation}

\begin{equation} \label{HTau}
\begin{tikzpicture}[scale=0.6, baseline=2ex]
    \draw[->] (1,0) to (3,1);
    \draw[<-] (2,0) to [curve through={(1.9,0.2) .. (1.6,0.5) .. (1.9,0.8)}] (2,1);
    \draw[->] (3,0) to (1,1);
    \draw[->] (4,1) to (4,0);
\end{tikzpicture} \hspace{2mm} + \hspace{2mm}
\begin{tikzpicture}[scale=0.6, baseline=1.5ex]
    \draw[->] (1,0) to (1,1);
    \draw[<-] (2,0) to [bend left] (3,0);
    \draw[->] (2,1) to [bend right] (3,1);
    \draw[<-] (4,0) to (4,1);
\end{tikzpicture}, \hspace{5mm}
\begin{tikzpicture}[scale=0.6, baseline=1.5ex]
    \draw[<-] (1,0) to (3,1);
    \draw[->] (2,0) to [curve through={(1.9,0.2) .. (1.6,0.5) .. (1.9,0.8)}] (2,1);
    \draw[<-] (3,0) to (1,1);
    \draw[->] (0,0) to (0,1);
\end{tikzpicture} \hspace{2mm} + \hspace{2mm}
\begin{tikzpicture}[scale=0.6, baseline=1.5ex]
    \draw[->] (1,0) to (1,1);
    \draw[<-] (2,0) to [bend left] (3,0);
    \draw[->] (2,1) to [bend right] (3,1);
    \draw[<-] (4,0) to (4,1);
\end{tikzpicture}
\end{equation}

\vspace{2mm}
\noindent
The \emph{affine partition algebra} is defined to be $AP_{k}:=\mathsf{End}_{\mathsf{APar}}((\uparrow\downarrow)^{k})$.

\end{defn}

We can generalise the arguments in the proof of \Cref{VarphiSurj} to show the following result.

\begin{thm}
The affine partition category $\mathsf{APar}$ is the full monoidal subcategory of $\mathsf{Heis}$ generated by the object $\uparrow \downarrow$.
\end{thm}

\begin{proof}
We need to show that 
$${\rm Hom}_{\mathsf{APar}}((\uparrow \downarrow)^k, (\uparrow \downarrow)^l) = {\rm Hom}_{\mathsf{Heis}}((\uparrow \downarrow)^k, (\uparrow \downarrow)^l).$$
Using \Cref{BasisThm}, we need to show that any element of the form 
$$c_w^{k_w} \ldots c_1^{k_1} c_0^{k_0} r_1^{s_1} r_3^{s_3} \ldots r_{2k-1}^{s_{2k-1} } \alpha r_2^{t_2} r_4^{t_4} \ldots r_{2l}^{t_{2l}}$$
where $\alpha \in \mathsf{Sim}((\uparrow \downarrow)^k, (\uparrow \downarrow)^l)$ can be written in terms of the generating morphisms in $\mathsf{APar}$.
The morphisms $r_i$ can be obtained by tensoring the generators (\ref{HAff}) with the appropriate identity morphisms on the left and right (and subtracting the identity). Moreover, the  morphisms $c_i$ can be obtained by concatenating $r_1^i$ with the generators (\ref{CupsCaps}) on top and bottom. Thus, it remains to show that any diagram $\alpha \in \mathsf{Sim}((\uparrow \downarrow)^k, (\uparrow \downarrow)^l)$ can be written in terms of the generating morphism in $\mathsf{APar}$. A generalisation of Jones' normal form shows that any planar $\alpha \in \mathsf{Sim}((\uparrow \downarrow)^k, (\uparrow \downarrow)^l)$ can be written in terms of the generators (\ref{HTri}) and (\ref{CupsCaps}) (see for example \cite[Proof of Lemma 2.1]{RSA14} for an explicit construction). Now \Cref{SimBarDecomp} allows us the write any $\alpha \in \mathsf{Sim}((\uparrow \downarrow)^k, (\uparrow \downarrow)^l)$ as  $\alpha = \overline{ \pi ^{\uparrow} \beta  \sigma^{\downarrow}}$ where $\pi \in S(k), \sigma\in S(l)$ and $\beta$ is planar. Note that $s_i^\uparrow$ and $s_i^\downarrow$ can be written using the generators (\ref{HTau}) and the composition of the generators (\ref{HTri}) (and tensoring with the appropriate identity morphism on the left and right). So using the discussion following \Cref{permutation} we know that $\pi^\uparrow$ and $\sigma^\downarrow$ belong to ${\rm Hom}_{\mathsf{APar}}((\uparrow \downarrow)^k, (\uparrow \downarrow)^l)$. Now we can follow exactly the same proof as for \Cref{VarphiSurj} noting that in this case the maximum degree is $(k+l, k+l)$  and the only simple diagram with that degree is the one containing $k$ consecutive arcs at the top and $l$ consecutive arcs at the bottom, which is planar. The rest of the proof can be followed verbatum simply replacing ${\rm Im \varphi}$ by ${\rm Hom}_{\mathsf{APar}}((\uparrow \downarrow)^k, (\uparrow \downarrow)^l)$.
\end{proof}

We immediately obtain the following consequences.

\begin{cor}
The map $\varphi$ gives a surjective homomorphism for $\mathcal{A}_{2k}^{\text{aff}}$ to $AP_k$.
\end{cor}

\begin{cor}
 The set $\mathsf{B}((\uparrow \downarrow)^k)$ gives a basis for $AP_k$.
\end{cor}

We do not know whether the map $\varphi$ is an isomorphism. If it were, then we would also have a presentation for $AP_k$.

\end{document}